\documentclass[11pt,a4paper]{article}
\usepackage[T1]{fontenc}
\usepackage[latin9]{inputenc}
\usepackage{color}
\usepackage{varioref}
\usepackage{amsthm}
\usepackage{amsmath}
\usepackage{amssymb}
\usepackage{esint}
\usepackage[all]{xy}
\usepackage[unicode=true,pdfusetitle,
 bookmarks=true,bookmarksnumbered=false,bookmarksopen=false,
 breaklinks=false,pdfborder={0 0 1},backref=false,colorlinks=true]
 {hyperref}

\makeatletter

\pdfpageheight\paperheight
\pdfpagewidth\paperwidth

\numberwithin{equation}{section}
\numberwithin{figure}{section}
\theoremstyle{plain}
\newtheorem{thm}{\protect\theoremname}[section]
  \theoremstyle{plain}
  \newtheorem{question}[thm]{\protect\questionname}
  \theoremstyle{remark}
  \newtheorem{rem}[thm]{\protect\remarkname}
  \theoremstyle{definition}
  \newtheorem{defn}[thm]{\protect\definitionname}
  \theoremstyle{plain}
  \newtheorem{lem}[thm]{\protect\lemmaname}
  \theoremstyle{definition}
  \newtheorem{example}[thm]{\protect\examplename}
  \theoremstyle{plain}
  \newtheorem{prop}[thm]{\protect\propositionname}
  \theoremstyle{plain}
  \newtheorem{cor}[thm]{\protect\corollaryname}

\usepackage{a4wide}

\usepackage{type1cm}
\usepackage[T1]{fontenc}
\usepackage{courier}
\usepackage{sectsty}
\allsectionsfont{\usefont{OT1}{ptm}{bc}{n}\selectfont}
\usepackage{bbold}
\usepackage{calrsfs}

\renewenvironment{proof}[1][\proofname]{\par
      \pushQED{\qed}%
      \normalfont \topsep6\p@\@plus6\p@\relax
      \trivlist
      \item[\hskip\labelsep
            \bfseries
        #1\@addpunct{.}]\ignorespaces
    }{%
  \popQED\endtrivlist\@endpefalse
}

\makeatother

  \providecommand{\corollaryname}{Corollary}
  \providecommand{\definitionname}{Definition}
  \providecommand{\examplename}{Example}
  \providecommand{\lemmaname}{Lemma}
  \providecommand{\propositionname}{Proposition}
  \providecommand{\questionname}{Question}
  \providecommand{\remarkname}{Remark}
\providecommand{\theoremname}{Theorem}

\begin{document}

\title{Lagrangian Rabinowitz Floer homology and twisted cotangent bundles}

\author{Will J. Merry}
\maketitle
\begin{abstract}
We study the following rigidity problem in symplectic geometry:\emph{
}can one displace a Lagrangian submanifold from a hypersurface? We
relate this to the Arnold Chord Conjecture, and introduce a refined
question about the existence of relative leaf-wise intersection points,
which are the Lagrangian-theoretic analogue of the notion of leaf-wise
intersection points defined by Moser \cite{Moser1978}. Our tool is
Lagrangian Rabinowitz Floer homology, which we define first for Liouville
domains and exact Lagrangian submanifolds with Legendrian boundary.
We then extend this to the `virtually contact' setting. By means of
an Abbondandolo-Schwarz short exact sequence we compute the Lagrangian
Rabinowitz Floer homology of certain regular level sets of Tonelli
Hamiltonians of sufficiently high energy in twisted cotangent bundles,
where the Lagrangians are conormal bundles. We deduce that in this
situation a generic Hamiltonian diffeomorphism has infinitely many
relative leaf-wise intersection points.
\end{abstract}

\section{\label{sec:Overview-of-Part 1}Introduction }

The aim of this paper is to study the following rigidity problem in
symplectic geometry: \emph{can one displace a Lagrangian submanifold
from a hypersurface}? In order to fix the ideas, let us start with
the following simple situation. Suppose $(X_{0},\lambda_{0})$ is
\emph{Liouville domain}, that is, $X_{0}$ is compact manifold with
boundary $\Sigma:=\partial X_{0}$, and $d\lambda_{0}$ is a symplectic
form on $X_{0}$ such that $\eta:=\lambda_{0}|_{\Sigma}$ is a positive
contact form on $\Sigma$. We attach to $X_{0}$ the positive part
of the symplectization of $\Sigma$ to form $X:=X_{0}\cup_{\Sigma}(\Sigma\times[1,\infty))$.
We extend $\lambda_{0}$ to a 1-form $\lambda$ defined on all of
$X$ by setting $\lambda:=r\eta$ on $\Sigma\times\{r\geq1\}$, and
call $(X,\lambda)$ the \emph{completion }of the Liouville domain.\emph{
}Now suppose $L\subset X$ is an exact Lagrangian submanifold that
is transverse to $\Sigma$ and is such that $K:=\Sigma\cap L$ is
a closed Legendrian submanifold of $(\Sigma,\eta)$. Denote by $ $$\theta^{t}:\Sigma\rightarrow\Sigma$
the Reeb flow of $\eta$. One can ask the following basic question.
\begin{question}
\label{q2}Is it possible to displace $\Sigma$ from $L$ via a compactly
supported Hamiltonian diffeomorphism?
\end{question}
This is related to the following better known question.
\begin{question}
\label{q1}Must there exist a Reeb chord with endpoints in $K$? That
is, a point $p\in K$ such that $\theta^{\tau}(p)\in K$ for some
$\tau\ne0$. 
\end{question}
The \emph{Arnold Chord Conjecture}, which is still open, asserts a
positive answer to Question \ref{q1} for any Legendrian $K$ in any
contact manifold $(\Sigma,\eta)$. This conjecture was originally
stated for Legendrian knots in $S^{3}$ (equipped with the standard
contact structure) by Arnold in \cite{Arnold1986}. In dimension 3
the conjecture has been completely proved by Hutchings and Taubes
\cite{HutchingsTaubes2011,HutchingsTaubes2011a}. In higher dimensions
Mohnke \cite{Mohnke2001} proved that the answer to Question \ref{q1}
is `yes' whenever the contact manifold arises as the boundary of a
subcritical Stein manifold of odd dimension. There is also a Floer-theoretic
proof that covers certain special cases of Mohnke's result which is
due to Cieliebak \cite{Cieliebak2002}. Other results are due to \cite{Abbas1999},
Ginzburg-Givental \cite{Givental1990,Givental1990a}, and more recently,
Bourgeois-Ekholm-Eliashberg \cite{BourgeoisEkholmEliashberg2009}
and Ritter \cite{Ritter2010}. 

Our first result is that a positive answer to Question \ref{q2} implies
a positive answer to Question \ref{q1}. 
\begin{thm}
\label{thm:baby 1}Suppose $(X,\lambda)$ is a completion of a Liouville
domain as above, and suppose that $L\subset X$ is an exact Lagrangian
submanifold transverse to $\Sigma$ with the property that $K:=\Sigma\cap L$
is a closed Legendrian submanifold of $(\Sigma,\eta:=\lambda|_{\Sigma})$.
If one can displace $\Sigma$ from $L$ via a compactly supported
Hamiltonian diffeomorphism, then there exists a Reeb chord of $\eta$
with endpoints in $K$. \end{thm}
\begin{rem}
\label{rem alex}In fact, Theorem \ref{thm:baby 1} can be deduced
from Ritter's result \cite{Ritter2010} alluded to above. The main
step in our proof Theorem \ref{thm:baby 1} is to show that if the
answer to Question \ref{q2} is `yes' then the \emph{Lagrangian Rabinowitz
Floer homology }$\mbox{RFH}_{*}(\Sigma,L,X)$ of $(\Sigma,L,X)$ vanishes.
Then we observe that vanishing of the Lagrangian Rabinowitz Floer
homology implies a positive answer to Question \ref{q1}. In \cite{Ritter2010}
Ritter proved that if the \emph{wrapped Floer homology} $\mbox{HW}_{*}(L)$
of $L$ vanishes then the answer to Question \ref{q1} is `yes'. Current
work in progress of Bounya \cite{Bounya} shows that $\mbox{RFH}_{*}(\Sigma,L,X)=0$
if and only if $\mbox{HW}_{*}(L)=0$. He proves this by constructing
a short exact sequence relating the Lagrangian Rabinowitz Floer homology
and the wrapped Floer homology, in a similar vein to Cieliebak, Frauenfelder
and Oancea's result \cite{CieliebakFrauenfelderOancea2010}, which
relates Rabinowitz Floer homology with symplectic homology. In this
sense Theorem \ref{q1} covers exactly the same cases of the Chord
Conjecture as Ritter's result. 
\end{rem}
Let us now discuss a refinement of Questions \ref{q2} and \ref{q1}.
Given a compactly supported Hamiltonian diffeomorphism $\psi:X\rightarrow X$,
we say that a point $p\in K$ is a \emph{relative leaf-wise intersection
point} of $\psi$ if the orbit $\{\theta^{t}(p)\}_{t\in\mathbb{R}}$
intersects $\psi^{-1}(L)$. Equivalently, a point $p\in K$ is a relative
leaf-wise intersection point if there exists $\tau\in\mathbb{R}$
such that 
\[
\psi(\theta^{\tau}(p))\in L.
\]
If for a given pair $\Sigma,L$ the answer to Question \ref{q2} is
`no', then it makes sense to ask the following question.
\begin{question}
\label{q3}Suppose it is not possible to displace $\Sigma$ from $L$
via a compactly supported Hamiltonian diffeomorphism. Is it then true
that every $\psi$ has a relative leaf-wise intersection point?
\end{question}
Note a Reeb chord is just the special case $\psi=\mbox{id}$. In both
Question \ref{q1} and Question \ref{q3} one can also ask for multiplicity
results. Many of the references given above (e.g.. \cite{Cieliebak2002})
prove the existence of more than just one Reeb chord. We are interested
in cases where there are infinitely many chords. 
\begin{question}
\label{q4}When is it true that for a generic $\psi$ there always
exist infinitely many relative leaf-wise intersection points? 
\end{question}
It is possible to see relative leaf-wise intersection points for a
given Hamiltonian diffeomorphism $\psi$ as critical points of a free-time
Hamiltonian action functional $\mathcal{A}_{\psi}$ called the \emph{Rabinowitz
action functional},\emph{ }that we define in Section \ref{sub:The-method-of}.
We construct a Floer theory for the functional $\mathcal{A}_{\psi}$,
which we call the \emph{Lagrangian Rabinowitz Floer homology} of $\mathcal{A}_{\psi}$.
This is the Lagrangian intersection theoretic version of \emph{Rabinowitz
Floer homology}, which was introduced by Cieliebak and Frauenfelder
in \cite{CieliebakFrauenfelder2009}, and used to detect (periodic)
leaf-wise intersections by Albers and Frauenfelder in \cite{AlbersFrauenfelder2010c}. 
\begin{rem}
As mentioned above, there are many approaches to answering Question
\ref{q1}, including various Floer-theoretic ones. We believe that
the main value of using Lagrangian Rabinowitz Floer homology in this
setting is that it allows us attack the more refined Question \ref{q3}.
To the best of our knowledge, none of the previous methods can be
directly used for Question \ref{q3}. Another advantage is that Lagrangian
Rabinowitz Floer homology can be in settings than symplectic/wrapped
Floer homology cannot be. See Section \ref{sub:A-more-complicated}
below.
\end{rem}
In certain situations it is possible to compute the Lagrangian Rabinowitz
Floer homology, which allows us to give an affirmative answer to Question
\ref{q3} and establish a partial answer to \ref{q4}. Here is one
such setting: let $(M^{n},g)$ denote a closed connected orientable
Riemannian manifold of dimension $n\geq2$, and consider the cotangent
bundle $T^{*}M$ equipped with its canonical symplectic structure
$d\lambda_{\textrm{can}}$, where $\lambda_{\textrm{can}}$ is the
Liouville 1-form on $T^{*}M$. Recall that if $S^{d}\subset M$ is
any closed connected submanifold, then the \emph{conormal bundle }$N^{*}S$
is the submanifold of $T^{*}M$ given by 
\[
N^{*}S:=\left\{ (q,p)\in T^{*}M\mid q\in S,\, p|_{T_{q}S}=0\right\} .
\]
It is easy to see that $N^{*}S$ is a Lagrangian submanifold of $T^{*}M$.
For instance, if $S=\{q\}$ is a point, then $N^{*}S=T_{q}^{*}M$,
and if $S=M$, then $N^{*}S$ is the zero section $M\subset T^{*}M$.
Denote by $P(M,S)$ the set of all smooth paths $q:[0,1]\rightarrow M$
with $q(0)$ and $q(1)$ both lying in $S$. We prove the following
result, which is based on the work of Abbondandolo and Schwarz \cite{AbbondandoloSchwarz2009,AbbondandoloSchwarz2006}. 
\begin{thm}
\label{thm:baby 2}Let $(M^{n},g)$ denote a closed connected orientable
Riemannian manifold of dimension $n\geq2$, and let $U^{*}M$ denote
the unit cotangent bundle. Let $S^{d}\subseteq M$ denote a closed
connected submanifold. Assume that one of the following two conditions
hold:
\begin{enumerate}
\item $d<n/2$, or $d=n/2$ and $n\geq4$,
\item The double coset space $\pi_{1}(S)\backslash\pi_{1}(M)/\pi_{1}(S)$
is non-trivial.
\end{enumerate}
Then it is not possible to displace $U^{*}M$ from $N^{*}S$, and
the answer to Question \ref{q3} is `yes'. Moreover, if $\dim\,\mbox{\emph{H}}_{*}(P(M,S);\mathbb{Z}_{2})=\infty$
and the pair $(U^{*}M,N^{*}S)$ is non-degenerate (cf. Section \ref{sub:The-Rabinowitz-action}
\textemdash{} this condition is satisfied generically), then a generic
Hamiltonian diffeomorphism has infinitely many relative leaf-wise
intersection points.\end{thm}
\begin{rem}
The second condition in Theorem \ref{thm:baby 2} is equivalent to
the statement that $P(M,S)$ is \emph{not }connected.
\end{rem}

\subsection{\label{sub:A-more-complicated}A more complicated setting}

In the main body of the paper we work in a somewhat more general setting
than the one described above. This is due to our particular interest
in \emph{twisted cotangent bundles}. We introduce these shortly, but
roughly speaking, the goal is to define Lagrangian Rabinowitz Floer
homology and prove Theorem \ref{thm:baby 1} in a sufficiently general
setting that it applies to twisted cotangent bundles. We shall then
compute the Lagrangian Rabinowitz Floer homology for certain hypersurfaces
in twisted cotangent bundles, thus proving the analogue of Theorem
\ref{thm:baby 2}. Unfortunately, this extra level of generality necessitates
a less clean exposition. It is quite likely that some readers will
only be interested in the setting described above. If this is the
case, it is possible to skip large swathes of the paper, beginning
with the rest of Section \ref{sub:A-more-complicated}. We indicate
in the exposition those sections that may safely be omitted.

We now define twisted cotangent bundles. As before, let $M$ denote
a closed connected orientable $n$-dimensional manifold, where $n\geq2$.
Let $\pi:T^{*}M\rightarrow M$ denote the footpoint map $\pi(q,p)\mapsto q$,
and let $\rho:\widetilde{M}\rightarrow M$ denote the universal cover
of $M$. We write $\rho_{\sharp}:T^{*}\widetilde{M}\rightarrow T^{*}M$
for the map defined by $\rho_{\sharp}(p):=\left(D\rho(q)^{-1}\right)^{*}(p)$
for $p\in T_{\rho(q)}^{*}\widetilde{M}$. By convention, if $A\subseteq T^{*}M$
is any submanifold, we denote by $\widetilde{A}:=\rho_{\sharp}^{-1}(A)$.
Note that $\widetilde{A}$ is in general \emph{not} the universal
cover of $A$ (see Remark \ref{rem univ cover}). Suppose $\sigma\in\Omega^{2}(M)$
is a \emph{closed }2-form. We pull $\sigma$ back to $T^{*}M$ and
add it to the canonical symplectic form $d\lambda_{\textrm{can}}$
to obtain a new symplectic form 
\[
\omega:=d\lambda_{\textrm{can}}+\pi^{*}\sigma
\]
 on $T^{*}M$. One calls $\omega$ a\emph{ twisted symplectic form}
or a \emph{magnetic symplectic form}. The latter terminology comes
from viewing the Hamiltonian system on $(T^{*}M,\omega)$ defined
by the Hamiltonian $(q,p)\mapsto\frac{1}{2}\left|p\right|^{2}$ as
modeling the motion of a particle moving on $M$ under the effect
of a magnetic field, represented by $\sigma$. We refer the reader
to \cite{ArnoldGivental1990,Ginzburg1996} for more information on
twisted cotangent bundles. In this paper $\sigma$ may or may not
be exact, but we always insist that $\sigma$ is\textbf{\emph{ }}\emph{weakly
exact}, that is, the lift $\widetilde{\sigma}:=\rho^{*}\sigma\in\Omega^{2}(\widetilde{M})$
is exact (which is equivalent to requiring that $\sigma|_{\pi_{2}(M)}=0$).
In fact, we will always make the additional assumption that $\widetilde{\sigma}$
admits a \emph{bounded} primitive: there exists $\varphi\in\Omega^{1}(\widetilde{M})$
such that $d\varphi=\widetilde{\sigma}$ and such that 
\begin{equation}
\sup_{q\in\widetilde{M}}\left|\varphi_{q}\right|<\infty,\label{eq:bounded-1}
\end{equation}
where the norm $\left|\cdot\right|$ is given by the lift of any Riemannian
metric on $M$ to $\widetilde{M}$. Note that if $\sigma$ is not
exact, then $\omega$ is also not exact. Thus twisted cotangent bundles
do not fit into the setting discussed in the previous section. Following
\cite{CieliebakFrauenfelderPaternain2010}, we will develop the theory
for hypersurfaces that are \emph{virtually contact}, which is when
the lifted hypersurface $\widetilde{\Sigma}:=\rho_{\sharp}^{-1}(\Sigma)$
is of contact type with respect to the lifted symplectic form $\widetilde{\omega}:=d\widetilde{\mu}_{\textrm{can}}+\widetilde{\pi}^{*}\widetilde{\sigma}$
on $T^{*}\widetilde{M}$ (which \emph{is }exact). Moreover the primitive
$\lambda$ of $\widetilde{\omega}$ that restricts to define a contact
form on $\widetilde{\Sigma}$ must be suitably bounded: 
\begin{equation}
\sup_{\widetilde{\Sigma}}|\lambda|<\infty,\ \ \ \inf_{\widetilde{\Sigma}}\lambda(R)>0,\label{eq:what it must satisfy}
\end{equation}
where $R$ is a vector field generating $\ker\,\omega|_{\Sigma}$
pulled back to $\widetilde{M}$. We work in the virtually contact
setting because wide class of physically relevant hypersurfaces fall
into this category: namely, all regular level sets of \emph{Tonelli
Hamiltonians} $H:T^{*}M\rightarrow\mathbb{R}$ for sufficiently high
energy levels. Here we recall that the classical Tonelli assumption
means that $H$ is \emph{fibrewise strictly convex}\textbf{ }and \emph{superlinear}.
In other words, the second differential $d^{2}(H|_{T_{q}^{*}M})$
of $H$ restricted to each tangent space $T_{q}^{*}M$ is positive
definite, and 
\[
\lim_{\left|p\right|\rightarrow\infty}\frac{H(q,p)}{\left|p\right|}=\infty
\]
uniformly for $q\in M$. As before, we are interested in applying
the theory in the case where the Lagrangian is a conormal bundle $N^{*}S$.
However since we are now working in a twisted cotangent bundle, not
all conormal bundles are Lagrangian submanifolds. In fact it is not
hard to see that $N^{*}S$ is a Lagrangian submanifold of $(T^{*}M,\omega)$
if and only if $\sigma|_{S}=0$. If $\sigma|_{S}=0$, then to any
Tonelli Hamiltonian $H$ there is an associated \emph{Ma\~n\'e critical
value}\textbf{ }$c(H,\sigma,S)\in\mathbb{R}\cup\{\infty\}$. The precise
definition of the Ma\~n\'e critical value $c(H,\sigma,S)$ is given
in Section \ref{sub:Ma-critical-values} below. Here we content ourselves
with saying only that the dynamics of the Hamiltonian flow $\phi_{H}^{t}|_{\Sigma}:\Sigma\cap N^{*}S\rightarrow\Sigma$
differ dramatically depending on whether $c(H,\sigma,S)$ is finite,
and if it is, whether it is positive or negative. In this paper we
are interested in the case where $c(H,\sigma,S)<0$, and hence we
make the following definition.
\begin{defn}
\label{def MSCP}Consider a closed connected hypersurface $\Sigma\subset T^{*}M$
and a closed connected submanifold $S$ such that $\sigma|_{S}=0$,
with $\Sigma\cap N^{*}S\ne\emptyset$ and $\Sigma\pitchfork N^{*}S$.
The pair $(\Sigma,S)$ is called a \emph{Ma\~n\'e supercritical}\textbf{
}\emph{pair} if there exists a Tonelli Hamiltonian $H:T^{*}M\rightarrow\mathbb{R}$
with $c(H,\sigma,S)<0$, and such that $\Sigma$ is the regular level
set $H^{-1}(0)$.
\end{defn}
Lemma \ref{lem:MSP IMPLIES RAT} below says that if $(\Sigma,S)$
is a Ma\~n\'e supercritical pair then $\Sigma$ is of virtual restricted
contact type. Let us now state versions of Theorem \ref{thm:baby 1}
and Theorem \ref{thm:baby 2} that are valid in this more general
setting \textemdash{} note that here a relative leaf-wise intersection
point $p$ of a Hamiltonian diffeomorphism $\psi:X\rightarrow X$
is a point $p\in\Sigma\cap L$ such that the characteristic chord
through $p$ intersects $\psi^{-1}(L)$. We refer the reader to Section
\ref{sec:Preliminaries} below for precise definitions of all the
terms involved.
\begin{thm}
\label{thm:big 1}Suppose $(X,\omega)$ is a geometrically bounded
symplectically aspherical symplectic manifold with $c_{1}(TX)=0$.
Let $\Sigma\subset X$ denote a closed connected $\pi_{1}$-injective
hypersurface that encloses a compact connected component of $X\backslash\Sigma$,
and let $L\subset X$ denote a $\pi_{1}$-injective Lagrangian submanifold
transverse to $\Sigma$ with $\Sigma\cap L\ne\emptyset$. Let $\widetilde{X}\rightarrow X$
denote the universal cover of $X$. Assume there exists a primitive
$\lambda$ of the lifted symplectic form $\widetilde{\omega}$ such
that:
\begin{enumerate}
\item $\sup_{\widetilde{\Sigma}}|\lambda|<\infty$ and $\inf_{\widetilde{\Sigma}}\lambda(R)>0$,
where $R$ is a vector field generating $\ker\,\omega|_{\Sigma}$
pulled back to $\widetilde{X}$.
\item $\lambda|_{\widetilde{L}}=d(\mbox{bounded function})$. 
\end{enumerate}
If there exists a compactly supported Hamiltonian diffeomorphism $\psi:X\rightarrow X$
with no relative leaf-wise intersection points (e.g. if one can displace
$\Sigma$ from $L$), then there exists a characteristic chord of
$\Sigma$ with endpoints in $\Sigma\cap L$.\end{thm}
\begin{rem}
The assumption that $c_{1}(TX)=0$ can be dropped, at the expense
of losing the $\mathbb{Z}$-grading on the Lagrangian Rabinowitz Floer
homology. The grading is inessential for Theorem \ref{thm:big 1}.
The construction is also perfectly valid if instead of a single Lagrangian
submanifold $L$, we work with a pair $L_{0}\pitchfork L_{1}$ of
$\pi_{1}$-injective Lagrangian submanifolds that intersect transversely
over $\Sigma$. Additional assumptions would be needed for a $\mathbb{Z}$-grading
on the Lagrangian Rabinowitz Floer homology, but again, this is not
required for Theorem \ref{thm:big 1}. In some sense the case $L_{0}\pitchfork L_{1}$
is easier, as then the Rabinowitz action functional can be assumed
to be Morse, rather than Morse-Bott (cf. Lemma \ref{lem:relating non degeneracy}).
Nevertheless, for the sake of a uniform presentation we work with
one Lagrangian throughout the paper.
\end{rem}
As the discussion above indicates, the setup described in Theorem
\ref{thm:big 1} is tailor-made to deal with twisted cotangent bundles.
We next state a version of Theorem \ref{thm:baby 2} valid for twisted
cotangent bundles.
\begin{thm}
\label{thm:big 2}Let $(M^{n},g)$ denote a closed connected orientable
Riemannian manifold of dimension $n\geq2$, and let $\sigma\in\Omega^{2}(M)$
denote a weakly exact 2-form whose lift to the universal cover admits
a bounded primitive. Equip $T^{*}M$ with the twisted symplectic form
$\omega:=d\mu_{\textrm{\emph{can}}}+\pi^{*}\sigma$. Let $S^{d}\subseteq M$
denote a closed connected submanifold such that $\sigma|_{S}=0$,
and let $\Sigma\subset T^{*}M$ denote a hypersurface such that
$(\Sigma,N^{*}S)$ form a Ma\~n\'e supercritical pair (cf. Definition
\ref{def MSCP}). Assume that one of the following two conditions
hold:
\begin{enumerate}
\item $d<n/2$, or $d=n/2$ and $n\geq4$,
\item The double coset space $\pi_{1}(S)\backslash\pi_{1}(M)/\pi_{1}(S)$
is non-trivial.
\end{enumerate}
Then it is not possible to displace $\Sigma$ from $N^{*}S$, and
the answer to Question \ref{q3} is `yes'. Moreover if $\dim\,\mbox{\emph{H}}_{*}(P(M,S);\mathbb{Z}_{2})=\infty$
and the pair $(\Sigma,N^{*}S)$ is non-degenerate (cf. Section \ref{sub:The-Rabinowitz-action}
\textemdash{} this condition is satisfied generically), then a generic
Hamiltonian diffeomorphism has infinitely many relative leaf-wise
intersection points.
\end{thm}

\subsection{\label{sub:The-method-of}The method of proof}

We conclude the Introduction with a brief explanation of the proofs
of Theorems \ref{thm:big 1} and \ref{thm:big 2}. Actually, for simplicity
here we restrict to the easier setting described at the beginning
of the Introduction, where $(X,\lambda)$ is a Liouville domain and
$L$ is an exact Lagrangian transversely intersecting $\Sigma=\partial X_{0}$
in a Legendrian submanifold $K$ (thus covering Theorems \ref{thm:baby 1}
and \ref{thm:baby 2} instead of Theorems \ref{thm:big 1} and \ref{thm:big 2}).
Define $P(X,L)$ to be the set of smooth paths $x:[0,1]\rightarrow X$
satisfying $x(0),x(1)\in L$. Let $H:X\rightarrow\mathbb{R}$ denote
a smooth function such that 
\[
H(p,r)=h(r)\ \ \ \mbox{for }(p,r)\in\Sigma\times(0,\infty),
\]
where $h:(0,\infty)\rightarrow\mathbb{R}$ is a smooth increasing
function that is constant on $\{r\leq1/4\}\cup\{r\geq3\}$ and equal
to $r-1$ on $\{\tfrac{1}{2}\leq r\leq2\}$. Suppose $\lambda|_{L}=dl$.
The \emph{Rabinowitz action functional }
\[
\mathcal{A}:P(X,L)\times\mathbb{R}\rightarrow\mathbb{R}
\]
is defined by 
\[
\mathcal{A}(x,\tau):=\int_{0}^{1}x^{*}\lambda+l(x(0))-l(x(1))-\tau\int_{0}^{1}H(x(t))dt.
\]
The critical points of $\mathcal{A}$ come in two flavors: if $(x,\tau)\in\mbox{Crit}(\mathcal{A})$
with $\tau\ne0$, then the path $\zeta(t):=x(t/\tau)$ is a Reeb chord
with endpoints in $L$. If $(x,0)\in\mbox{Crit}(\mathcal{A})$, then
$x(t)\equiv p$ for some point $p\in K$. The \emph{Lagrangian Rabinowitz
Floer homology }$\mbox{RFH}_{*}(\Sigma,L,X)$ is the homology of a
chain complex generated by the critical points of $\mathcal{A}$.
The boundary operator is defined by counting rigid solutions $u=(x,\tau):\mathbb{R}\rightarrow P(X,L)\times\mathbb{R}$
of the following pair of coupled second order non-linear elliptic
partial differential equations:
\begin{equation}
\begin{cases}
\partial_{s}x+J(x)\partial_{t}x=\tau\nabla H(x),\\
\partial_{s}\tau=\int_{0}^{1}H(x)dt,
\end{cases}\label{eg RF equation-1}
\end{equation}
which connect different critical points of $\mathcal{A}$. Here we
are simplifying the picture drastically, as critical points of $\mathcal{A}$
are never isolated. In reality we use Frauenfelder's theory \cite{Frauenfelder2004}
of gradient flow lines with \emph{cascades}. As mentioned in Remark
\ref{rem alex}, the proof of Theorem \ref{thm:baby 1} now uses the
following two ingredients.
\begin{enumerate}
\item If there are no Reeb chords of $\lambda|_{\Sigma}$ with end points
in $K$, then one has $\mbox{RFH}_{*}(\Sigma,L,X)\cong\mbox{H}_{*+(n-1)/2}(K;\mathbb{Z}_{2})$.
Indeed, in this case the critical point set of $\mathcal{A}$ can
be identified with $K$ itself, and the boundary operator reduces
to the boundary operator in Morse homology for a given Morse function
on $K$ (this is because we count gradient flow lines with cascades;
see Section \ref{sub:The-definition-of}).
\item If there exists a compactly supported Hamiltonian diffeomorphism with
no relative leaf-wise intersection points, then $\mbox{RFH}_{*}(\Sigma,L,X)=0$. 
\end{enumerate}
To motivate why (2) should be true, let us explain how the functional
$\mathcal{A}$ can be tweaked to detect relative leaf-wise intersection
points. This idea is due to Albers-Frauenfelder \cite{AlbersFrauenfelder2010c}.
Suppose $\psi:X\rightarrow X$ is a compactly supported Hamiltonian
diffeomorphism. Choose a compactly supported Hamiltonian $F_{t}:X\rightarrow\mathbb{R}$
such that $\psi=\phi_{F}^{1}$. Now define a new functional $\mathcal{A}_{\psi}$
by 
\[
\mathcal{A}_{\psi}(x,\tau):=\int_{0}^{1}x^{*}\lambda+l(x(0))-l(x(1))-\eta\int_{0}^{1}\beta(t)H(x(t))dt-\int_{0}^{1}\dot{\chi}(t)F_{\chi(t)}(x(t))dt.
\]
Here $\beta:S^{1}\rightarrow\mathbb{R}$ is a smooth function with
$\beta(t)=0$ for all $t\in[\tfrac{1}{2},1]$, and the integral $\int_{0}^{1}\beta(t)dt$
is equal to $1$, and $\chi:[0,1]\rightarrow[0,1]$ is a smooth monotone
map with $\chi(\tfrac{1}{2})=0$ and $\chi(1)=1$. The point of the
two cutoff functions $\beta$ and $\chi$ is to ensure that $\beta(t)H(x)$
and $\dot{\chi}(t)F_{\chi(t)}(x)$ have disjoint time support. This
implies that if $(x,\tau)\in\mbox{Crit}(\mathcal{A}_{\psi})$, then
$x(0)\in K$, and for $t\in[0,\tfrac{1}{2}]$ one has $x(t)=\theta^{\beta(t)}(x(0))$,
and for $t\in[\tfrac{1}{2},1]$ one has $x(t)=\phi_{F}^{\chi(t)}(x(\tfrac{1}{2}))$.
In other words, if $p:=x(0)$, then $p\in K$ and $\psi(\theta^{\tau}(p))\in L$.
Thus $p$ is a relative leaf-wise intersection point of $\psi$. The
key point now is that one can define the Rabinowitz Floer homology
$\mbox{RFH}_{*}(\mathcal{A}_{\psi})$ for $\mathcal{A}_{\psi}$ as
well, and in fact the Rabinowitz Floer homology is unchanged: 
\[
\mbox{RFH}_{*}(\mathcal{A}_{\psi})\cong\mbox{RFH}_{*}(\Sigma,L,X).
\]
This should be viewed in the same spirit as the fact that the Morse
{[}resp. Floer{]} homology of a closed {[}symplectic{]} manifold is
independent of the Morse {[}resp. Hamiltonian{]} function. Now statement
(2) above is clear: if $\psi$ has no relative leaf-wise intersection
points, then the corresponding functional $\mathcal{A}_{\psi}$ has
no critical points \textemdash{} and thus $\mbox{RFH}(\mathcal{A}_{\psi})=0$.
It remains to be explain how the computations of $\mbox{RFH}_{*}(\Sigma,N^{*}S,T^{*}M)$
is made. We extend to the Lagrangian setting the Abbondandolo-Schwarz
\cite{AbbondandoloSchwarz2009} short exact sequence, which relates
the Lagrangian Rabinowitz Floer chain complex to the \emph{Morse complex}\textbf{
}of an appropriate \emph{free time action functional}. The homology
of this complex is (roughly speaking) the singular homology of the
space $P(M,S)$. In our earlier paper \cite{Merry2011a} we extended
the short exact sequence from \cite{AbbondandoloSchwarz2009} to the
setting of twisted cotangent bundles, and the idea here is very similar. 

\emph{Acknowledgment: }I would like to thank my Ph.D. adviser Gabriel
P. Paternain for many helpful discussions. I am also extremely grateful
to Alberto Abbondandolo, Peter Albers and Urs Frauenfelder, together
with all the participants of the 2009-2010 Cambridge seminar on Rabinowitz
Floer homology, for several stimulating remarks and insightful suggestions,
and for pointing out errors in previous drafts of this work. This
work forms part of my PhD thesis \cite{Merry2011}. Finally, I am
grateful to Irida Altman and the anonymous referees for their useful
comments on making the paper more readable.

\section{\label{sec:Preliminaries}Preliminaries}

Here are some notational conventions.
\begin{itemize}
\item We denote by $C_{\textrm{ct}}^{\infty}(X,\mathbb{R})$ the set of
functions on $X$ which are constant outside of a compact set, and
by $C_{0}^{\infty}(X,\mathbb{R})\subset C_{\textrm{ct}}^{\infty}(X,\mathbb{R})$
the subset of compactly supported functions.
\item We use the (non-standard) sign convention that an almost complex structure
$J$ on a symplectic manifold $(X,\omega)$ is \textbf{\emph{$\omega$}}\emph{-compatible}\textbf{
}if $g_{J}:=\omega(J\cdot,\cdot)$ is a Riemannian metric on $X$.
We denote by $\mathcal{J}(X,\omega)$ the set of all $\omega$-compatible
almost complex structures on $X$. 
\item Given a family $\mathbf{J}=(J_{t})_{t\in[0,1]}\subset\mathcal{J}(X,\omega)$,
and $(x,\tau)\in C^{\infty}([0,1],X)\times\mathbb{R}$, we use the
special notation $\left\langle \left\langle \cdot,\cdot\right\rangle \right\rangle _{\mathbf{J}}$
to denote the inner product on $C^{\infty}(x^{*}TX)\times\mathbb{R}$
defined by
\begin{equation}
\left\langle \left\langle (\xi,h),(\xi',h')\right\rangle \right\rangle _{\mathbf{J}}:=\int_{0}^{1}g_{J_{t}}(\xi(t),\xi'(t))dt+hh'.\label{eq:two brackets metric}
\end{equation}

\item The \emph{symplectic gradient}\textbf{ }$X_{H}\in\mbox{Vect}(X)$
of a Hamiltonian $H:X\rightarrow\mathbb{R}$ is defined by $i_{X_{H}}\omega=-dH$.
Thus the gradient $\nabla H$ of $H$ with respect to $g_{J}$ is
given by $\nabla H=JX_{H}$. 
\end{itemize}
In this section we introduce the precise setting in which we define
the Lagrangian Rabinowitz Floer homology. We are aiming for Definition
\ref{rab adm lag-1}, which introduces the notion of a \emph{Rabinowitz
admissible triple }$(\Sigma,L,\alpha)$. This is the setting that
we will prove Theorem \ref{thm:big 1} in. 
\begin{rem}
\label{rem simple bit}If however the reader is only interested in
the setting described at the beginning of the Introduction, things
become much simpler, and the reader may skim this entire section apart
from Remark \ref{rem simplifying things}, where we explicitly this
special case.
\end{rem}
Let $(X^{2n},\omega)$ denote a connected non-compact symplectic manifold
such that:
\begin{enumerate}
\item $(X,\omega)$ is \emph{geometrically bounded}\textit{\emph{ - this
means that there exist $\omega$-compatible almost complex structures
$J$ on with the property that the Riemannian metric $g_{J}(\cdot,\cdot):=\omega(J\cdot,\cdot)$
is complete, has bounded sectional curvature and has injectivity radius
bounded away from zero}}
\item The first Chern class $c_{1}(TX,J)$ is zero (for some, and hence
any $J\in\mathcal{J}(X,\omega)$).
\item The symplectic form $\omega$ is \emph{symplectically aspherical}.
This means that for every smooth map $f:S^{2}\rightarrow X$, one
has $\int_{S^{2}}f^{*}\omega=0$.
\end{enumerate}
Assumption (2) is made for simplicity only, and could be weakened
at the expense of losing the $\mathbb{Z}$-grading on the Lagrangian
Rabinowitz Floer homology. Assumption (1) however is much more crucial,
and cannot be dispensed with. If we denote by $\rho:\widetilde{X}\rightarrow X$
the universal cover of $X$ and by $\widetilde{\omega}:=\mathsf{\rho}^{*}\omega\in\Omega^{2}(\widetilde{X})$
then Assumption (3) is equivalent to requiring that $\widetilde{\omega}$
is exact. Our main interest in such symplectic manifolds is due to
the fact that the twisted cotangent bundles introduced in Section
\ref{sub:A-more-complicated} satisfy these requirements \textemdash{}
see Section \ref{sub:Ma-critical-values}).
\begin{rem}
\label{rem univ cover}Suppose $A\subseteq X$ is a submanifold. We
denote by $A^{\textrm{univ}}\rightarrow A$ the universal cover of
$A$, and by $\widetilde{A}:=\rho^{-1}(A)\subseteq\widetilde{X}$.
Note in general, $A^{\textrm{univ}}\ne\widetilde{A}$ if $A\varsubsetneq X$.
We say that $A$ is \emph{$\pi_{1}$-injective }if the inclusion $A\hookrightarrow X$
induces an injection $\pi_{1}(A)\rightarrow\pi_{1}(X)$. In this case
$\widetilde{A}$ is a disjoint union of components each diffeomorphic
to $A^{\textrm{univ}}$. In particular, each component of $\widetilde{A}$
is simply connected, and thus $\mbox{H}^{1}(\widetilde{A};\mathbb{Z})=0$. 
\end{rem}
Recall that for a closed connected orientable hypersurface $\Sigma\subset X$,
there is a distinguished oriented line bundle $\ker\,\omega\rightarrow\Sigma$
over $\Sigma$ called the \emph{characteristic line bundle}. We denote
by $\mathcal{D}(\Sigma)$ the set of Hamiltonians $H\in C^{\infty}(X,\mathbb{R})$
with the property that $\Sigma$ is the regular energy level $H^{-1}(0)$,
and that the symplectic gradient $X_{H}|_{\Sigma}$ is a positively
oriented section of $\ker\,\omega$, and we write $\mathcal{D}_{\textrm{ct}}(\Sigma):=\mathcal{D}(\Sigma)\cap C_{\textrm{ct}}^{\infty}(X,\mathbb{R})$.
A \emph{characteristic chord}\textbf{ }of $\Sigma$ with endpoints
in some specified Lagrangian submanifold $L$ of $X$ is a flow line
of $\phi_{H}^{t}$ for some $H\in\mathcal{D}(\Sigma)$ which starts
and ends in $\Sigma\cap L$. This is independent of the choice of
$H\in\mathcal{D}(\Sigma)$ since for $H_{1},H_{2}\in\mathcal{D}(\Sigma)$
the flows $\phi_{H_{1}}^{t}|_{\Sigma}$ and $\phi_{H_{2}}^{t}|_{\Sigma}$
differ only by a time change. We are primarily interested in the case
when the hypersurface $\Sigma$ satisfies the following condition,
which was introduced by Cieliebak-Frauenfelder-Paternain in \cite{CieliebakFrauenfelderPaternain2010}.
\begin{defn}
\label{thm:VRCT}A closed connected hypersurface $\Sigma$ is of \emph{virtual
restricted contact type} if (a) $\Sigma$ is $\pi_{1}$-injective,
(b) $\Sigma$ encloses a compact connected component of $X\backslash\Sigma$
and (c), there exists a primitive $\lambda$ of $\widetilde{\omega}$
such that:
\begin{enumerate}
\item For some (and hence any) Riemannian metric $g$ on $\Sigma$, there
exists a constant $C=C(g)<\infty$ such that 
\begin{equation}
\sup_{x\in\widetilde{\Sigma}}|\lambda_{x}|\leq C,\label{eq:vc1}
\end{equation}
where $\left|\cdot\right|$ denotes the lift of $g$ to $\widetilde{\Sigma}$.
\item For some (and hence any) non-vanishing positively oriented section
$R$ of $\ker\,\omega$, there exists a constant $\varepsilon=\varepsilon(R)>0$
such that
\begin{equation}
\inf_{x\in\widetilde{\Sigma}}\lambda(\widetilde{R}(x))\geq\varepsilon,\label{eq:vc2}
\end{equation}
where $\widetilde{R}$ denotes a lift of $R$ to $\widetilde{X}$.
\end{enumerate}
\end{defn}
We now discuss the Lagrangians that we consider. All Lagrangian submanifolds
in this paper are assumed to be connected, even if this is not explicitly
stated. Suppose we are given a Lagrangian submanifold $L$ of $X$
that is $\pi_{1}$-injective. Since we assume that $\omega$ is symplectically
aspherical and $c_{1}(TX)=0$, the $\pi_{1}$-injectivity assumption
implies that $\omega|_{\pi_{2}(X,L)}=c_{1}|_{\pi_{2}(X,L)}=0$. Since
$\mbox{H}^{1}(\widetilde{L};\mathbb{Z})=0$ (cf. Remark \ref{rem univ cover})
and $\omega|_{L}=0$, if $\lambda\in\Omega^{1}(\widetilde{X})$ is
a primitive of $\widetilde{\omega}$, we can find a smooth function
$l:\widetilde{L}\rightarrow\mathbb{R}$ such that $\lambda|_{\widetilde{L}}=dl$. 
\begin{defn}
\label{def virtually exact}We say that a $\pi_{1}$-injective Lagrangian
$L$ is \emph{virtually exact}\textbf{ }if one can choose a primitive
$\lambda$ of $\widetilde{\omega}$ and a function $l$ such that
$\lambda|_{\widetilde{L}}=dl$, where $l\in C^{\infty}(\widetilde{L},\mathbb{R})$
is a \emph{bounded}\textbf{ }function. 
\end{defn}
Now we explain the notion of a \emph{good }primitive of $\widetilde{\omega}$.
\begin{defn}
\label{def good primitive}Suppose we have fixed a pair $\Sigma,L$
consisting of a hypersurface $\Sigma$ of virtual restricted contact
type and a virtually exact Lagrangian. A primitive $\lambda$ of $\widetilde{\omega}$
is called \emph{good with respect to $\Sigma,L$ }(or just \emph{good
}if $\Sigma$ and $L$ are understood) if $\lambda$ satisfies \eqref{eq:vc1}
and \eqref{eq:vc2}, \emph{and }has the property that $\lambda|_{\widetilde{L}}=d(\mbox{bounded function})$.
\end{defn}
Next, we define the appropriate notion of homotopy:
\begin{defn}
\label{def homotopy}Suppose that $\Sigma$ is a hypersurface of virtual
restricted contact type, $L$ is a virtually exact Lagrangian, and
$\lambda$ is a good primitive for $(\Sigma,L)$. Fix $H\in\mathcal{D}_{\textrm{ct}}(\Sigma)$.
A \emph{good homotopy} is a family $(H_{s},\lambda_{s})_{s\in(-\varepsilon,\varepsilon)}$
such that:
\begin{enumerate}
\item $(H_{s})$ is a smooth family of uniformly compactly supported Hamiltonians
such that $H_{0}=H$, and such that $\Sigma_{s}:=H_{s}^{-1}(0)$ is
of virtual restricted contact type for each $s\in(-\varepsilon,\varepsilon)$,
with $H_{s}\in\mathcal{D}_{\textrm{ct}}(\Sigma_{s})$;
\item $(\lambda_{s})$ is a smooth family of $1$-forms such that $\lambda_{0}=\lambda$
and $\lambda_{s}$ is a good primitive with respect to $\Sigma_{s},L$,
and such that the constants in \eqref{eq:vc1} and \eqref{eq:vc2}
may be taken independently of $s$, and if $\lambda_{s}|_{\widetilde{L}}=dl_{s}$
then $\sup_{s}\left\Vert l_{s}\right\Vert <\infty$. 
\end{enumerate}
\end{defn}
Let us fix once and for all a point $\star\in X$. When talking about
Lagrangian submanifolds $L$ of $X$, we shall always implicitly assume
that $\star\in L$. Let $P(X,L)$ denote the set of smooth maps $x:[0,1]\rightarrow X$
with $x(0)\in L$ and $x(1)\in L$. Define 
\begin{equation}
\Pi_{L}:=\pi_{1}(X,\star)/\sim,\label{eq:pi L-1}
\end{equation}
where for $\mathsf{a},\mathsf{b}\in\pi_{1}(X,\star)$ we have $\mathsf{a}\sim\mathsf{b}$
if and only if there exists $\mathsf{c}_{0},\mathsf{c}_{1}\in\pi_{1}(L,\star)$
such that 
\[
\mathsf{a}=\mathsf{c}_{0}\mathsf{b}\mathsf{c}_{1}
\]
(i.e. $\Pi_{L}$ is the double coset space $\pi_{1}(L,\star)\backslash\pi_{1}(X,\star)/\pi_{1}(L,\star)$).
It is not hard to see that $\Pi_{L}\cong\pi_{0}(P(X,L))$ (see for
instance \cite[Lemma 3.3.1]{Pozniak1999}). Given $\alpha\in\Pi_{L}$,
we denote by $P_{\alpha}(X,L)$ the connected component of $P(X,L)$
corresponding to $\alpha$, so that 
\[
P(X,L)=\bigsqcup_{\alpha\in\Pi_{L}}P_{\alpha}(X,L).
\]
Let us now fix for each $\alpha\in\Pi_{L}$ a smooth loop $x_{\alpha}:S^{1}\rightarrow X$
with $x_{\alpha}(0)=\star$ such that $x_{\alpha}$ represents $\alpha$.
It is convenient to choose these loops $x_{\alpha}$ so that the class
$0\in\Pi_{L}$ is represented by the constant path $x_{0}(t)\equiv\star$,
and such that $x_{\alpha}(t)=x_{-\alpha}(1-t)$. Fix a point $\widetilde{\star}\in\widetilde{X}$
that projects onto $\star$, and for each $\alpha\in\Pi_{L}$ let
$\widetilde{x}_{\alpha}:[0,1]\rightarrow\widetilde{X}$ denote the
unique lift of $x_{\alpha}$ with $\widetilde{x}_{\alpha}(0)=\widetilde{\star}$.
In particular, $\widetilde{x}_{0}(t)=\widetilde{\star}$ for all $t$.
Given $x\in P_{\alpha}(X,L)$, let us say that a map $\bar{x}:[0,1]\times[0,1]\rightarrow X$
is a \emph{filling}\textbf{ }of $x$ if $\bar{x}$ satisfies:
\begin{itemize}
\item $\bar{x}(0,t)=x(t)$,
\item $\bar{x}(1,t)=x_{\alpha}(t)$, and 
\item $\bar{x}([0,1]\times\{0,1\})\subset L$.
\end{itemize}
If $f:S^{1}\rightarrow P(X,L)$ is a smooth loop then we may alternatively
think of $f$ as a map $f:S^{1}\times[0,1]\rightarrow X$ with $f(S^{1}\times\{0,1\})\subset L$.
We will only work with classes $\alpha\in\Pi_{L}$ for which the following
condition is satisfied:
\begin{description}
\item [{\label{enu:If--is-2}(A)}] If $f:S^{1}\rightarrow P_{\alpha}(X,L)$
is any smooth loop then $\int_{S^{1}\times[0,1]}f^{*}\omega=0$.
\end{description}
Since $L$ is $\pi_{1}$-injective and $\omega|_{\pi_{2}(X)}=0$,
one has $\omega|_{\pi_{2}(X,L)}=0$, and thus \textbf{(A) }is satisfied
for the element $0\in\Pi_{L}$. When \textbf{(A) }is satisfied we
can define the \emph{symplectic area}\textbf{\emph{ }}\emph{functional}\textbf{
}\textit{\emph{$\Omega:P_{\alpha}(X,L)\rightarrow\mathbb{R}$ by}}\emph{
\begin{equation}
\Omega(x):=\int_{[0,1]\times[0,1]}\bar{x}^{*}\omega,\label{eq:symplectic area-1}
\end{equation}
}where $\bar{x}$ is any filling of $x$ (that is, \textbf{(A)} implies
that $\Omega$ is well defined). The precise conditions under which
we will define the Lagrangian Rabinowitz Floer homology is given by
the following definition.
\begin{defn}
\label{rab adm lag-1}A triple $(\Sigma,L,\alpha)$ is called \emph{Rabinowitz
admissible} if:
\begin{itemize}
\item $\Sigma$ is a hypersurface of virtual restricted contact type,
\item $L$ is a virtually exact Lagrangian submanifold,
\item $\Sigma\pitchfork L$, $\Sigma\cap L\ne\emptyset$,
\item $\alpha\in\Pi_{L}$ satisfies \textbf{(A)},
\item There exist good primitives $\lambda$ of $\widetilde{\omega}$.
\end{itemize}
\end{defn}
\begin{rem}
\label{rem simplifying things}As promised at the beginning of this
section (see Remark \ref{rem simple bit}), we summarize in this Remark
the simplifications that one can make to Definition \ref{rab adm lag-1}
if one works in the setting described at the beginning of the Introduction.
Recall here we take $X$ to be the completion of a Liouville domain.
Here one starts with a a compact manifold $X_{0}$ with boundary $\Sigma:=\partial X_{0}$,
equipped with an exact symplectic form $d\lambda_{0}$ such that $\eta:=\lambda_{0}|_{\Sigma}$
a positive contact form on $\Sigma$. We attach to $X_{0}$ the positive
part of the symplectization of $\Sigma$ to form $X:=X_{0}\cup_{\Sigma}(\Sigma\times[1,\infty))$.
We extend $\lambda_{0}$ to a 1-form $\lambda$ defined on all of
$X$ by setting $\lambda:=r\eta$ on $\Sigma\times\{r\geq1\}$. Suppose
$L\subset X$ is an exact Lagrangian submanifold which is transverse
to $\Sigma$ and is such that $K:=\Sigma\cap L$ is a closed Legendrian
submanifold of $(\Sigma,\eta)$. In addition we make the assumption
that 
\begin{equation}
L\cap\left(\Sigma\times\{r\geq1\}\right)=K\times\{r\geq1\}.\label{eq:nice condition}
\end{equation}
This condition \eqref{eq:nice condition} implies that one can write
$\lambda|_{L}=dl$ for some function $l:L\rightarrow\mathbb{R}$ that
vanishes to infinite order along $K$, and is the analogue in this
setting to asking that good primitives in the sense of Definition
\ref{def good primitive} exist. If $L$ does not satisfy \eqref{eq:nice condition}
then one can deform $L\cap X_{0}$ via a Hamiltonian isotopy of $X_{0}$
relative to $\Sigma$ to obtain a new Lagrangian submanifold $L'_{0}$
of $X_{0}$ with the property that if $L':=L_{0}'\cup_{K}(K\times[1,\infty))$
then $L'$ satisfies \eqref{eq:nice condition} (see \cite[Lemma 3.1]{AbouzaidSeidel2010}).
Since in this case the symplectic form on $X$ is exact, every class
$\alpha\in\Pi_{L}$ satisfies the condition \textbf{(A)} introduced
\vpageref{enu:If--is-2}. In this case we define the \emph{symplectic
area functional }$\Omega:P(X,L)\rightarrow\mathbb{R}$ by 
\[
\Omega(x):=\int_{0}^{1}x^{*}\lambda+l(x(0))-l(x(1)).
\]

\end{rem}

\subsection{\label{sub:The-Rabinowitz-action}The Rabinowitz action functional}

We now introduce the Rabinowitz action functional. To begin with let
us assume $L\subset X$ is a virtually exact Lagrangian and $\alpha\in\Pi_{L}$
is a class satisfying the condition \textbf{(A)} defined \vpageref{enu:If--is-2}.
Thus the symplectic area functional $\Omega:P_{\alpha}(X,L)\rightarrow\mathbb{R}$
from \eqref{eq:symplectic area-1} is well defined.
\begin{defn}
\label{def the rab action functional}Let $H\in C_{\textrm{ct}}^{\infty}(X,\mathbb{R})$
and assume that $0$ is a regular value of $H$, with $H^{-1}(0)\cap L\ne\emptyset$
and $H^{-1}(0)\pitchfork L$. The\textbf{ }\emph{Rabinowitz action
functional}\textbf{\emph{ }}$\mathcal{A}_{H}:P_{\alpha}(X,L)\times\mathbb{R}\rightarrow\mathbb{R}$
is defined by
\[
\mathcal{A}_{H}(x,\tau)=\Omega(x)-\tau\int_{0}^{1}H(x)dt.
\]

\end{defn}
An easy computation shows that the critical points of $\mathcal{A}_{H}$
are pairs $(x,\tau)$ such that 
\[
\dot{x}=\tau X_{H}(x(t))\ \ \ \mbox{for all }t\in[0,1],
\]
\[
\int_{0}^{1}H(x)dt=0.
\]
Since $H$ is invariant under its Hamiltonian flow, the second equation
implies 
\[
H(x(t))=0\ \ \ \mbox{for all }t\in[0,1],
\]
and so 
\[
x([0,1])\subset H^{-1}(0).
\]
Thus if we denote by $\mbox{Crit}^{\alpha}(\mathcal{A}_{H})$ the
set of critical points of $\mathcal{A}_{H}$ then $(x,\tau)$ belongs
to $\mbox{Crit}^{\alpha}(\mathcal{A}_{H})$ if and only if
\[
\dot{x}=\tau X_{H}(x),\ \ \ x([0,1])\subset H^{-1}(0).
\]
If $(x,\tau)\in\mbox{Crit}^{\alpha}(\mathcal{A}_{H})$ with $\tau\ne0$
then $\zeta(t):=x(t/\tau)$ is a flow line of $\phi_{H}^{t}$. If
$\alpha\ne0$ these are the only possible critical points. However
if $\alpha=0$ and $p\in H^{-1}(0)\cap L$ then $(p,0)\in\mbox{Crit}^{0}(\mathcal{A}_{H})$,
where $p$ is also thought of as the constant path $t\mapsto p$.
Note that if $(x,\tau)\in\mbox{Crit}^{\alpha}(\mathcal{A}_{H})$ then
\begin{equation}
\mathcal{A}_{H}(x,\tau)=\Omega(x).\label{eq:action value}
\end{equation}
Given $-\infty<a<b<\infty$, denote by 
\[
\mbox{Crit}^{\alpha}(\mathcal{A}_{H})_{a}^{b}:=\left\{ (x,\tau)\in\mbox{Crit}^{\alpha}(\mathcal{A}_{H})\mid a\leq\mathcal{A}_{H}(x,\tau)\leq b\right\} .
\]
We always implicitly assume when referring to action windows that
the endpoints $a$ and $b$ are not critical values of $\mathcal{A}_{H}$.
Suppose\textbf{ }$\mathbf{J}=(J_{t})_{t\in[0,1]}\subset\mathcal{J}(X,\omega)$.
We let $\nabla_{\mathbf{J}}\mathcal{A}_{H}$ denote the gradient of
$\mathcal{A}_{H}$ with respect to $\left\langle \left\langle \cdot,\cdot\right\rangle \right\rangle _{\mathbf{J}}$
(cf. \eqref{eq:two brackets metric}), so that
\[
\nabla_{\mathbf{J}}\mathcal{A}_{H}(x,\tau)=\left(\begin{array}{c}
J_{t}(x)(\dot{x}-\tau X_{H}(x))\\
-\int_{0}^{1}H(x)dt
\end{array}\right).
\]
Given $(x,\tau)\in\mbox{Crit}(\mathcal{A}_{H})$ let us denote by
\[
\nabla_{\mathbf{J}}^{2}\mathcal{A}_{H}(x,\tau):W^{1,r}(x^{*}TX)\oplus\mathbb{R}\rightarrow L^{r}(x^{*}TX)\oplus\mathbb{R}
\]
(for some fixed $r\geq2$) the operator obtained by linearizing $\nabla_{\mathbf{J}}\mathcal{A}_{H}(x,\tau)$,
that is, 
\begin{equation}
\nabla_{\mathbf{J}}^{2}\mathcal{A}_{H}(x,\tau)(\xi,h):=\frac{\partial}{\partial s}\Bigl|_{s=0}\nabla_{\mathbf{J}}\mathcal{A}_{H}(x_{s},\tau_{s}),\label{eq:linearize}
\end{equation}
where $(x_{s},\tau_{s})_{s\in(-\varepsilon,\varepsilon)}\subset P_{\alpha}(X,L)\times\mathbb{R}$
satisfies 
\[
\frac{\partial}{\partial s}\Bigl|_{s=0}(x_{s},\tau_{s})=(\xi,h).
\]
One computes that 
\begin{equation}
\nabla_{\mathbf{J}}^{2}\mathcal{A}_{H}(x,\tau)\left(\begin{array}{c}
\xi\\
h
\end{array}\right)=\left(\begin{array}{c}
J_{t}(x)\nabla_{t}\xi+(\nabla_{\xi}J_{t})\dot{x}-\tau\nabla_{\xi}\nabla_{J_{t}}H-h\nabla_{J_{t}}H\\
-\int_{0}^{1}dH(\xi)dt
\end{array}\right),\label{eq:hessian}
\end{equation}
where $\nabla_{J_{t}}H$ denotes the gradient of $H$ with respect
to the metric $g_{J_{t}}$. 
\begin{defn}
A \emph{gradient flow line }of $(H,\mathbf{J})$ is a smooth map $u=(x,\tau):\mathbb{R}\rightarrow P_{\alpha}(X,L)\times\mathbb{R}$
such that 
\[
\partial_{s}u+\nabla_{\mathbf{J}}\mathcal{A}_{H}(u(s))=0.
\]
In components this reads:
\[
\partial_{s}x+J_{t}(x)(\partial_{t}x-\tau X_{H}(x))=0;
\]
\[
\partial_{s}\tau-\int_{0}^{1}H(x)dt=0.
\]
Given $-\infty<a<b<\infty$, denote by $\mathcal{M}^{\alpha}(H,\mathbf{J})_{a}^{b}$
the set of gradient flow lines $u$ of $(H,\mathbf{J})$ that satisfy
$a\leq\mathcal{A}_{H}(u(s))\leq b$ for all $s\in\mathbb{R}$. 
\end{defn}
It will often be useful to let both $H$ and $\mathbf{J}$ depend
additionally on a parameter $s\in\mathbb{R}$. Suppose $(H_{s})_{s\in\mathbb{R}}\subset C_{\textrm{ct}}^{\infty}(X,\mathbb{R})$
is a smooth family of Hamiltonians, which is \emph{asymptotically
constant }in the sense that there exist $H_{\pm}\in C_{\textrm{ct}}^{\infty}(X,\mathbb{R})$
such that $H_{s}=H_{-}$ for $s\ll0$ and $H_{s}=H_{+}$ for $s\gg0$.
Assume that $0$ is a regular value of both $H_{-}$ and $H_{+}$,
with $H_{\pm}(0)\cap L\ne\emptyset$ and $H_{\pm}^{-1}(0)\pitchfork L$.
Similarly, suppose we are given a family $(\mathbf{J}_{s}=(J_{s,t}))_{s\in\mathbb{R}}\subset\mathcal{J}(X,\omega)$
of almost complex structures which is also asymptotically constant
in the sense above. One can then study the $s$-dependent equation\textbf{
}
\[
\partial_{s}u+\nabla_{\mathbf{J}_{s}}\mathcal{A}_{H_{s}}(u(s))=0,
\]
and given $-\infty<a<b<\infty$, we denote by $\mathcal{M}^{\alpha}(H_{s},\mathbf{J}_{s})_{a}^{b}$
the set of smooth maps $u=(x,\tau)$ that satisfy this equation together
with the asymptotic conditions
\[
\lim_{s\rightarrow-\infty}\mathcal{A}_{H_{s}}(u(s))\leq b,\ \ \ \lim_{s\rightarrow\infty}\mathcal{A}_{H_{s}}(u(s))\geq a.
\]
Note that if $(H_{s},\mathbf{J}_{s})=(H,\mathbf{J})$ does not depend
on $s$ then $\mathcal{M}^{\alpha}(H_{s},\mathbf{J}_{s})_{a}^{b}=\mathcal{M}^{\alpha}(H,\mathbf{J})_{a}^{b}$.
Given a gradient flow line $u$, we denote by 
\[
D_{u}:W^{1,r}(x^{*}TX)\oplus W^{1,r}(\mathbb{R},\mathbb{R})\rightarrow L^{r}(x^{*}TX)\oplus L^{r}(\mathbb{R},\mathbb{R})
\]
the linear operator given by 
\begin{equation}
D_{u}\left(\begin{array}{c}
\xi\\
h
\end{array}\right):=\left(\begin{array}{c}
\nabla_{s}\xi+J_{t}(x)\nabla_{t}\xi+(\nabla_{\xi}J_{t})\partial_{t}x-\tau\nabla_{\xi}\nabla_{J_{t}}H-h\nabla_{J_{t}}H\\
\partial_{s}h-\int_{0}^{1}dH(x)(\xi)dt
\end{array}\right).\label{eq:the operator Du}
\end{equation}
Note that in the special case where $u(s)=(x,\tau)\in\mbox{Crit}^{\alpha}(\mathcal{A}_{H})$
we have $D_{u}=\nabla_{\mathbf{J}}^{2}\mathcal{A}_{H}(x,\tau)$. Suppose
$(x,\tau)\in\mbox{Crit}^{\alpha}(\mathcal{A}_{H})$ with $\tau\ne0$.
The \emph{nullity }of $(x,\tau)$ is the integer 
\[
n(x,\tau):=\dim\, D\phi_{H}^{\tau}(x(0))(T_{x(0)}L)\cap T_{x(1)}L.
\]
\label{correction term}We say that $(x,\tau)$ is \emph{non-degenerate
}if $n(x,\tau)=0$. It is well known that this implies that $\nabla_{\mathbf{J}}^{2}\mathcal{A}_{H}(x,\tau)$
is bijective. We wish to associate an integer $\chi(x,\tau)\in\{-1,1\}$
to each non-degenerate critical point $(x,\tau)$ with $\tau\ne0$.
As proved by Albers-Frauenfelder \cite[Proposition B.1]{AlbersFrauenfelder2009},
if $(x,\tau)$ is non-degenerate then we can find a smooth family
$\bar{\zeta}_{s}\in P(X,L)$ and a smooth function $s\mapsto\tau(s)$
such that $\tau(0)=\left|\tau\right|$ and $\bar{\zeta}_{0}=x$. We
set 
\[
\zeta_{s}(t):=\bar{\zeta}_{s}(t/\tau(s)).
\]
Moreover one also has $\tau'(0)\ne0$, and the function 
\[
e(s):=H(\zeta_{s}(0))
\]
is smooth. Note that $e(0)=0$. In fact, one has $e'(0)\ne0$. Since
this last statement is not proved in \cite[Proposition B.1]{AlbersFrauenfelder2009},
let quickly show this: define $\xi\in C^{\infty}(x^{*}TX)$ by $\xi(t):=\frac{\partial}{\partial s}\Bigl|_{s=0}\zeta_{s}(t\tau(s))$.
Then a direct computation shows that 
\[
\nabla_{\mathbf{J}}^{2}\mathcal{A}_{H}(x,\tau)(\xi,\tau'(0))=(0,e'(0)),
\]
and thus as $\nabla_{\mathbf{J}}^{2}\mathcal{A}_{H}(x,\tau)$ is bijective
and $\tau'(0)\ne0$ we must have $e'(0)\ne0$. Anyway, the upshot
is that it makes sense to define the \emph{correction term}\textbf{
}
\begin{equation}
\chi(x,\tau):=\mbox{sign}(\tau)\cdot\mbox{sign}\left(-\frac{e'(0)}{\tau'(0)}\right).\label{eq:correction term}
\end{equation}

\begin{rem}
A priori, it would appear that the correction term $\chi(x,\tau)$
could depend on the choice of family $(\zeta_{s})$. In fact, this
is not the case, as is proved in \cite[Lemma 5.12]{Merry2011}.
\end{rem}

\begin{rem}
In the simpler setting described in Remark \ref{rem simplifying things}
one can show that $\chi(x,\tau)=\mbox{sign}(\tau)$. See \cite[Section 2.3]{Merry2011}.
\end{rem}
If $p\in H^{-1}(0)\cap L$ then $(p,0)\in\mbox{Crit}^{0}(\mathcal{A}_{H})$.
Since we are assuming that $H^{-1}(0)\pitchfork L$, it is reasonable
to call all these critical points non-degenerate as well. The next
Lemma motivates this.
\begin{lem}
\label{lem:relating non degeneracy}If $\alpha\ne0$ and all critical
points of $\mathcal{A}_{H}$ belonging to $P_{\alpha}(X,L)\times\mathbb{R}$
are non-degenerate then $\mathcal{A}_{H}:P_{\alpha}(X,L)\times\mathbb{R}\rightarrow\mathbb{R}$
is a Morse functional and $\mbox{\emph{Crit}}^{\alpha}(\mathcal{A}_{H})$
consists of an isolated collection of points. If all the critical
points of $\mathcal{A}_{H}$ belonging to $P_{0}(X,L)\times\mathbb{R}$
are non-degenerate then $\mathcal{A}_{H}:P_{0}(X,L)\times\mathbb{R}\rightarrow\mathbb{R}$
is a Morse-Bott functional and $\mbox{\emph{Crit}}^{0}(\mathcal{A}_{H})$
consists of an isolated collection of points and a copy of $\Sigma\cap L$. \end{lem}
\begin{proof}
The fact that elements $(x,\tau)\in\mbox{Crit}^{\alpha}(\mathcal{A}_{H})$
with $\tau\ne0$ are isolated follows easily from the bijectivity
of $\nabla_{\mathbf{J}}^{2}\mathcal{A}_{H}(x,\tau)$. Set $\Sigma:=H^{-1}(0)$.
Let us show that the set $\{(p,0)\mid p\in\Sigma\cap L\}$ is a Morse-Bott
component of $\mbox{Crit}^{0}(\mathcal{A}_{H})$. Since $\Sigma\pitchfork L$,
it suffices to show that for all $p\in\Sigma\cap L$ we have
\begin{equation}
\ker\,\nabla_{\mathbf{J}}^{2}\mathcal{A}_{H}(p,0)=T_{p}(\Sigma\cap L)\times\{0\}\subset T_{(p,0)}(P_{0}(X,L)\times\mathbb{R}).\label{eq:morse bott cpt}
\end{equation}
If $p\in\Sigma\cap L$, an element $(\xi,h)\in T_{(p,0)}(P_{0}(X,L)\times\mathbb{R})$
is in the kernel of $\nabla_{\mathbf{J}}^{2}\mathcal{A}_{H}(p,0)$
if and only if $(\xi,h)$ solves the equations: 
\[
\dot{\xi}=hX_{H}(p),
\]
\[
\int_{0}^{1}dH(x)(\xi)dt=0.
\]
Integrating the first equation, we see that $\xi(t)=\xi(0)+hX_{H}(p)$.
Since $\xi(0)\in T_{p}L$ and $X_{H}(p)\notin T_{p}L$ as $H^{-1}(0)\pitchfork L$,
we must have $h=0$. Thus $\xi(t)=\xi(0)$ is constant. The second
equation tells us that $dH(x)(\xi(0))=0$, and hence $\xi(0)\in T_{p}\Sigma$. 
\end{proof}
A standard Sard-Smale argument proves the following result (see for
instance \cite{McDuffSalamon2004}):
\begin{lem}
\label{lem:more nondeg}There is a generic subset of $C_{\textrm{ct}}^{\infty}(X,\mathbb{R})$
of Hamiltonians $H$ for which
\begin{enumerate}
\item $0$ is a regular energy value of $H$, 
\item $H^{-1}(0)\cap L\ne\emptyset$ and $H^{-1}(0)\pitchfork L$, 
\item all the critical points of $\mathcal{A}_{H}$ belonging to $P_{\alpha}(X,L)\times\mathbb{R}$
are non-degenerate.
\end{enumerate}
\end{lem}

It is well known that if every critical point in $\mbox{Crit}^{\alpha}(\mathcal{A}_{H})_{a}^{b}$
is non-degenerate then for any choice of $\mathbf{J}=(J_{t})\subset\mathcal{J}(X,\omega)$,
every element $u\in\mathcal{M}^{\alpha}(H,\mathbf{J})_{a}^{b}$ is
asymptotically convergent at each end to elements of $\mbox{Crit}^{\alpha}(\mathcal{A}_{H})_{a}^{b}$.
That is, the limits 
\[
\lim_{s\rightarrow\pm\infty}u(s,t)=(x_{\pm}(t),\tau_{\pm}),\ \ \ \lim_{s\rightarrow\infty}\partial_{t}u(s,t)=0,
\]
exist, and the convergence is uniform in $t$, and the limits $(x_{\pm},\tau_{\pm})$
belong to $\mbox{Crit}^{\alpha}(\mathcal{A}_{H})_{a}^{b}$ (see for
instance \cite{Salamon1999}). If $E_{\mathbf{J}}(u)$ denotes the
\emph{energy}\textbf{ }of a gradient flow line:
\[
E_{\mathbf{J}}(u):=\int_{-\infty}^{\infty}\left\Vert \partial_{s}u(s)\right\Vert _{\mathbf{J}}^{2}ds,
\]
then if $u\in\mathcal{M}^{\alpha}(H,\mathbf{J})_{a}^{b}$ is asymptotically
convergent to $(x_{\pm},\tau_{\pm})\in\mbox{Crit}(\mathcal{A}_{H})_{a}^{b}$
one has 
\[
E_{\mathbf{J}}(u)=\mathcal{A}_{H}(x_{-},\tau_{-})-\mathcal{A}_{H}(x_{+},\tau_{+}),
\]
and hence $0\leq E_{\mathbf{J}}(u)\leq b-a$.
\begin{defn}
\label{def of nondegen rat}We say that a Rabinowitz admissible triple
$(\Sigma,L,\alpha)$ is a \emph{non-degenerate Rabinowitz admissible
triple }if for some (and hence any) $H\in\mathcal{D}_{\textrm{ct}}(\Sigma)$,
every critical point of $\mathcal{A}_{H}:P_{\alpha}(X,L)\times\mathbb{R}\rightarrow\mathbb{R}$
is non-degenerate.\end{defn}
\begin{rem}
In the simpler setting described in \ref{rem simplifying things}
we equivalently say that $(\Sigma,L,\alpha)$ is \emph{non-degenerate}
if the Reeb chords of $\eta$ with endpoints in $K$ in the homotopy
class $\alpha$ are isolated on $\Sigma$, and finally if $\zeta:[0,\tau]\rightarrow\Sigma$
is any such chord then 
\[
D\eta^{\tau}(\zeta(0))\left(T_{\zeta(0)}L\right)\pitchfork T_{\zeta(\tau)}L.
\]

\end{rem}

\subsection{\label{sec:Compactness}Compactness}

Recall that an $\omega$-compatible almost complex structure $J$
is\textbf{ }\emph{geometrically bounded} if the corresponding Riemannian
metric $g_{J}:=\omega(J\cdot,\cdot)$ is complete, has bounded sectional
curvature and has injectivity radius bounded away from zero. By our
initial assumption on $X$ such almost complex structures exist; let
us fix once and for all such an almost complex structure $J_{\textrm{gb}}$.
We denote by $\mathcal{J}_{\textrm{gb}}(X,\omega;J_{\textrm{gb}})\subset\mathcal{J}(X,\omega)$
the set of almost complex structures $J\in\mathcal{J}(X,\omega)$
for which there exists a compact set $K\subset X$ (depending on
$J$) such that $J=J_{\textrm{gb}}$ on $X\backslash K$. Since in
general it is unknown whether the set of all geometrically bounded
almost complex structures on $(X,\omega)$ is connected, it is possible
that everything we do will depend on our initial choice of geometrically
bounded almost complex structure $J_{\textrm{gb}}$. The general consensus
however seems to be that this is unlikely. Regardless, we will ignore
this subtlety throughout. 

The following two compactness results are key to everything that follows.
The first result is for gradient flow lines of a pair $(H,\mathbf{J})$;
the second result is for $s$-dependent trajectories. These results
were originally proved in the periodic case for hypersurfaces of restricted
contact type in \cite{CieliebakFrauenfelder2009}. A full proof in
this setting can be found in \cite{Merry2011}. We remark that it
is these results where the hypothesis that $\Sigma$ is of virtual
restricted contact type is used, and where we use the fact that there
exist good primitives of $\widetilde{\omega}$.
\begin{thm}
\label{thm:first linfty}Assume $(\Sigma,L,\alpha)$ is a non-degenerate
Rabinowitz admissible triple.\textbf{\emph{ }}Let $H\in\mathcal{D}_{\textrm{\emph{ct}}}(\Sigma)$
and $\mathbf{J}=(J_{t})\subset\mathcal{J}_{\textrm{\emph{gb}}}(X,\omega;J_{\textrm{\emph{gb}}})$.
Fix $-\infty<a<b<\infty$ and suppose $(u^{\nu}=(x^{\nu},\tau^{\nu}))_{\nu\in\mathbb{N}}\subset\mathcal{M}^{\alpha}(H,\mathbf{J})_{a}^{b}$.
Then for any sequence $(s^{\nu})\subset\mathbb{R}$, the reparametrized
sequence $u^{\nu}(\cdot+s^{\nu})$ has a subsequence which converges
in $C_{\textrm{\emph{loc}}}^{\infty}(\mathbb{R}\times[0,1],X)\times C_{\textrm{\emph{loc}}}^{\infty}(\mathbb{R},\mathbb{R})$.
\end{thm}
Recall the notion of a good homotopy from Definition \ref{def homotopy}.
\begin{thm}
\label{thm:second linfty}Assume $(\Sigma_{\pm},L,\alpha)$ are both
non-degenerate Rabinowitz admissible triples. Fix Hamiltonians $H_{\pm}\in\mathcal{D}_{\textrm{\emph{ct}}}(\Sigma_{\pm})$
and suppose that there exists a good homotopy $(H_{s},\lambda_{s})_{s\in\mathbb{R}}$
such that $(H_{s},\lambda_{s})=(H_{-},\lambda_{-})$ for $s\ll0$
and $(H_{s},\lambda_{s})=(H_{+},\lambda_{+})$ for $s\gg0$, and such
that $H_{s}$ has compact support uniformly in $s$. Fix $\mathbf{J}_{\pm}=(J_{\pm,t})\subset\mathcal{J}_{\textrm{\emph{gb}}}(X,\omega;J_{\textrm{\emph{gb}}})$,
and choose a smooth family $(\mathbf{J}_{s}=(J_{s,t}))_{s\in\mathbb{R}}\subset\mathcal{J}_{\textrm{\emph{gb}}}(X,\omega;J_{\textrm{\emph{gb}}})$
such that there exists a compact set $K\subset X$ such that $J_{s,t}=J_{\textrm{\emph{gb}}}$
on $X\backslash K$ for all $(s,t)\in\mathbb{R}\times[0,1]$, and
such that $\mathbf{J}_{s}=\mathbf{J}_{-}$ for $s\ll0$ and $\mathbf{J}_{s}=\mathbf{J}_{+}$
for $s\gg0$.

There exists a constant $\kappa>0$ such that if $\left\Vert \partial_{s}H_{s}\right\Vert _{L^{\infty}}<\kappa$
then the conclusion of the previous theorem holds. That is, if $\left\Vert \partial_{s}H_{s}\right\Vert _{L^{\infty}}<\kappa$
then for any sequence $(u^{\nu}=(x^{\nu},\tau^{\nu}))_{\nu\in\mathbb{N}}\subset\mathcal{M}^{\alpha}(H_{s},\mathbf{J}_{s})_{a}^{b}$
and any sequence $(s^{\nu})\subset\mathbb{R}$, the reparametrized
sequence $u^{\nu}(\cdot+s^{\nu})$ has a subsequence which converges
in $C_{\textrm{\emph{loc}}}^{\infty}(\mathbb{R}\times[0,1],X)\times C_{\textrm{\emph{loc}}}^{\infty}(\mathbb{R},\mathbb{R})$.\end{thm}
\begin{rem}
\label{rem:why constant is good}We remark that because we are assuming
that all our Hamiltonians are constant outside of a compact set, the
only thing one needs to prove in the above two theorems is that the
Lagrange multiplier component $\tau$ of a flow line $u=(x,\tau)$
is uniformly bounded. The bound on the loop component $x$ comes essentially
``for free'' from our assumption that the almost complex structures
we work with are all geometrically bounded outside of a compact set;
see for instance \cite{CieliebakGinzburgKerman2004}. Later on we
will need to work with Hamiltonians that are \emph{not}\textbf{\emph{
}}constant outside a compact set; hence more work will need to be
done here (cf. the discussion in Section \ref{sub:Lagrangian-Rabinowitz-Floer with tonelli}).
\end{rem}

\begin{rem}
\label{thm:crit a,b is finite}It follows from Theorem \ref{thm:first linfty}
that given $-\infty<a<b<\infty$, the subset $\mbox{Crit}^{\alpha}(\mathcal{A}_{H})_{a}^{b}$
is compact (by the Arzel\`a-Ascoli theorem). Thus as $\mathcal{A}_{H}$
is Morse {[}resp. Morse-Bott if $\alpha=0${]},\textbf{\emph{ }}the
set $\mbox{Crit}^{\alpha}(\mathcal{A}_{H})_{a}^{b}$ is at most finite
{[}resp. has at most finitely many components{]}.
\end{rem}

\subsection{\label{sub:The-definition-of}The definition of $\mbox{RFH}_{*}^{\alpha}(H)$}

Assume $(\Sigma,L,\alpha)$ is a non-degenerate Rabinowitz admissible
triple. In particular $\Sigma$ is transverse to $L$. The Rabinowitz
action functional $\mathcal{A}_{H}$ is Morse-Bott by Lemma \ref{lem:relating non degeneracy}
(in fact, Morse if $\alpha\ne0$). The aim of this subsection is introduce
a Floer homology theory for the functional $\mathcal{A}_{H}$. There
are various ways to deal with the problem that $\mathcal{A}_{H}:P_{0}(X,L)\times\mathbb{R}\rightarrow\mathbb{R}$
is not Morse. One possibility is to choose an additional small perturbation
to make the Rabinowitz action functional a Morse function. This was
first done in a Floer homological setting by Pozniak in his thesis
\cite{Pozniak1999}, and was carried out in the context of Rabinowitz
Floer homology by Cieliebak-Frauenfelder-Paternain \cite{CieliebakFrauenfelderPaternain2010}.
Another option is to introduce an auxiliary Morse function on $\mbox{Crit}^{\alpha}(\mathcal{A}_{H})$
and take as generators of the Rabinowitz Floer complex the critical
points of the Morse function. This approach was studied originally
by Frauenfelder \cite{Frauenfelder2004} in the finite dimensional
case, and also Bourgeois \cite{Bourgeois2003} and Bourgeois-Oancea
\cite{BourgeoisOancea2009} in the infinite dimensional case. Cieliebak
and Frauenfelder used this construction in their original approach
to Rabinowitz Floer homology (see \cite[Appendix A]{CieliebakFrauenfelder2009}),
and in this article we will do the same. We emphasize though that
if $\alpha\ne0$ then the Rabinowitz action functional is actually
Morse, and in this case one can just work with normal gradient flow
lines. Thus the reader should bear in mind that a lot of what follows
can be considerably simplified when $\alpha\ne0$.

Let $f:\mbox{Crit}^{\alpha}(\mathcal{A}_{H})\rightarrow\mathbb{R}$
denote a Morse function and a fix Riemannian metric $m$ on $\mbox{Crit}^{\alpha}(\mathcal{A}_{H})$
such that the flow $\varphi^{t}:\mbox{Crit}^{\alpha}(\mathcal{A}_{H})\rightarrow\mbox{Crit}^{\alpha}(\mathcal{A}_{H})$
of $-\nabla f:=-\nabla_{m}f$ is Morse-Smale. We abbreviate
\[
C^{\alpha}(f):=\mbox{Crit}(f)\subset\mbox{Crit}^{\alpha}(\mathcal{A}_{H})
\]
\[
C^{\alpha}(f)_{a}^{b}:=\mbox{Crit}(f)\cap\mbox{Crit}^{\alpha}(\mathcal{A}_{H})_{a}^{b}\ \ \ .
\]
Note that $C^{\alpha}(f)_{a}^{b}$ is finite (cf. Remark \ref{thm:crit a,b is finite}).
Fix $\mathbf{J}=(J_{t})\subset\mathcal{J}_{\textrm{gb}}(X,\omega;J_{\textrm{gb}})$.
\begin{defn}
\label{thm:the moduli spaces}Fix $m\in\mathbb{N}$. Given $(x_{\pm},\tau_{\pm})\in C^{\alpha}(f)_{a}^{b}$,
a \emph{flow line from $(x_{-},\tau_{-})$ to $(x_{+},\tau_{+})$
with $k$ cascades}\textbf{\emph{ }}is a $k$-tuple $(\mathbf{u},\mathbf{T})=((u_{j})_{j=1,\dots,k},(T_{j})_{j=1,\dots,k-1})$
of gradient flow lines $u_{j}:\mathbb{R}\rightarrow P_{\alpha}(X,L)\times\mathbb{R}$
of $(H,\mathbf{J})$ and real numbers $T_{j}\ge0$ such that 
\[
\lim_{s\rightarrow-\infty}u_{1}(s)\in W^{u}((x_{-},\tau_{-});-\nabla f),\ \ \ \lim_{s\rightarrow+\infty}u_{k}(s)\in W^{s}((x_{+},\tau_{+});-\nabla f),
\]
\[
\lim_{s\rightarrow-\infty}u_{j+1}(s)=\varphi^{T_{j}}\left(\lim_{s\rightarrow\infty}u_{j}(s)\right)\ \ \ \mbox{for }j=1,\dots,k-1.
\]
We denote the space of gradient flow lines with $k$ cascades from
the critical point $(x_{-},\tau_{-})$ to the critical point $(x_{+},\tau_{+})$
by $\widetilde{\mathcal{M}}_{k}((x_{-},\tau_{-}),(x_{+},\tau_{+}))$,
and we denote by $\mathcal{M}_{k}((x_{-},\tau_{-}),(x_{+},\tau_{+}))$
the quotient $\widetilde{\mathcal{M}}_{k}((x_{-},\tau_{-}),(x_{+},\tau_{+}))/\mathbb{R}^{k}$,
where $\mathbb{R}^{k}$ acts by reparametrization on each of the $k$
cascades. We define a flow line with zero cascades to be a gradient
flow line of $-\nabla f$, and denote by $\widetilde{\mathcal{M}}_{0}((x_{-},\tau_{-}),(x_{+},\tau_{+}))$
the set of flow lines with zero cascades that are asymptotically equal
to $(x_{\pm},\tau_{\pm})$. We put $\mathcal{M}_{0}((x_{-},\tau_{-}),(x_{+},\tau_{+})):=\widetilde{\mathcal{M}}_{0}((x_{-},\tau_{-}),(x_{+},\tau_{+}))/\mathbb{R}$.
Finally we define 
\[
\mathcal{M}((x_{-},\tau_{-}),(x_{+},\tau_{+})):=\bigcup_{k\in\mathbb{N}\cup\{0\}}\mathcal{M}_{k}((x_{-},\tau_{-}),(x_{+},\tau_{+})).
\]

\end{defn}

\begin{defn}
\label{mu}Given a non-degenerate critical point $(x,\tau)\in\mbox{Crit}^{\alpha}(\mathcal{A}_{H})$
with $\tau\ne0$, set 
\[
\mu(x,\tau):=\mu_{\mathsf{Ma}}(x,\tau)-\frac{1}{2}\chi(x,\tau),
\]
where $\mu_{\mathsf{Ma}}(x,\tau)$ is the \emph{Maslov index}\textbf{\emph{
}}of the path $\zeta(t):=x(t/\tau)$ (see \cite{RobbinSalamon1993}
for the definition, and \cite[Section 5.5]{Merry2011} for the precise
sign conventions we are using), and the correction term $\chi(x,\tau)$
was defined in \eqref{eq:correction term}. We set 
\[
\mu(p,0):=-\frac{n-1}{2}.
\]
If $(x,\tau)\in C^{\alpha}(f)$ define $\mu_{f}(x,\tau):=\mu(x,\tau)+i_{f}(x,\tau)$,
where $i_{f}(x,\tau)$ is the Morse index of $(x,\tau)$ as a critical
point of $f$ (thus $i_{f}(x,\tau)=0$ whenever $\tau\ne0$). Our
sign conventions imply that for all $(x,\tau)\in C^{\alpha}(f)$,
\[
\mu_{f}(x,\tau)\in\begin{cases}
\mathbb{Z}, & \mbox{if }n\mbox{ is odd,}\\
\frac{1}{2}\mathbb{Z}\backslash\mathbb{Z}, & \mbox{if }n\mbox{ is even.}
\end{cases}
\]

\end{defn}
The following theorem is part of the standard Floer homology package,
the key ingredient being Theorem \ref{thm:first linfty}. The index
computation is probably the most non-routine element \textemdash{}
full details of this aspect can be found in \cite{Merry2011}. 
\begin{thm}
\label{thm:moduli space manifold}For a generic choice of $\mathbf{J}$
and a generic Morse-Smale metric $m$ on $\mbox{\emph{Crit}}^{\alpha}(\mathcal{A}_{H})$
the moduli spaces $\mathcal{M}((x_{-},\tau_{-}),(x_{+},\tau_{+}))$
for $(x_{\pm},\tau_{\pm})\in C^{\alpha}(f)$ are smooth manifolds
of finite dimension 
\[
\dim\mathcal{M}((x_{-},\tau_{-}),(x_{+},\tau_{+}))=\mu_{f}(x_{-},\tau_{-})-\mu_{f}(x_{+},\tau_{+})-1.
\]
Moreover if $\mu_{f}(x_{-},\tau_{-})=\mu_{f}(x_{+},\tau_{+})+1$ then
$\mathcal{M}((x_{-},\tau_{-}),(x_{+},\tau_{+}))$ is compact, and
hence a finite set.
\end{thm}
Denote by 
\[
\mbox{CRF}_{*}^{\alpha}(H,f)_{a}^{b}:=C_{*}^{\alpha}(f)_{a}^{b}\otimes\mathbb{Z}_{2},
\]
where the grading $*$ is given by the function $\mu_{f}$ from Definition
\ref{mu}. Given $(x_{\pm},\tau_{\pm})\in C^{\alpha}(f)_{a}^{b}$
with $\mu_{f}(x_{-},\tau_{-})=\mu_{f}(x_{+},\tau_{+})+1$, we define
the number $n((x_{-},\tau_{-}),(x_{+},\tau_{+}))\in\mathbb{Z}_{2}$
to be the parity of the finite set $\mathcal{M}((x_{-},\tau_{-}),(x_{+},\tau_{+}))$.
If $(x_{+},\tau_{+})\in C^{\alpha}(f)_{a}^{b}$ has $\mu_{f}(x_{-},\tau_{-})\ne\mu_{f}(x_{+},\tau_{+})+1$,
set $n((x_{-},\tau_{-}),(x_{+},\tau_{+}))=0$. Now we define the boundary
operator 
\[
\partial_{a}^{b}=\partial_{a}^{b}(H,\mathbf{J},f,m):\mbox{CRF}_{*}^{\alpha}(H,f)_{a}^{b}\rightarrow\mbox{CRF}_{*-1}^{\alpha}(H,f)_{a}^{b}
\]
as the linear extension of 
\[
(x_{-},\tau_{-})\mapsto\sum_{(x_{+},\tau_{+})\in C(f)_{a}^{b}}n((x_{-},\tau_{-}),(x_{+},\tau_{+}))(x_{+},\tau_{+}).
\]
The usual argument shows that $\partial_{a}^{b}\circ\partial_{a}^{b}=0$,
and thus $\{\mbox{CRF}_{*}^{\alpha}(H,f)_{a}^{b},\partial_{a}^{b}\}$
carries the structure of a differential $\mathbb{Z}_{2}$-vector space.
We denote by $\mbox{RFH}_{*}^{\alpha}(H,\mathbf{J},f,m)_{a}^{b}$
its homology. The standard theory of continuation homomorphisms in
Floer theory show that the homology $\mbox{RFH}_{*}^{\alpha}(H,\mathbf{J},f,m)_{a}^{b}$
is independent up to canonical isomorphism of the choices of $f$,
$\mathbf{J}$, and $m$, and thus we omit them from the notation and
write simply $\mbox{RFH}_{*}^{\alpha}(H)_{a}^{b}$. Next, suppose
$-\infty<a<a'<b<\infty$. Let $\mathtt{p}_{a,a'}^{b}:\mbox{CRF}_{*}^{\alpha}(H,f)_{a}^{b}\rightarrow\mbox{CRF}_{*}^{\alpha}(H,f)_{a'}^{b}$
denote the projection along $\mbox{CRF}_{*}^{\alpha}(H,f)_{a}^{a'}$.
Since the action decreases along gradient flow lines, $\mathtt{p}_{a,a'}^{b}$
commutes with the boundary operators $\partial_{a}^{b}$ and $\partial_{a'}^{b}$,
and hence induces a map
\[
\mathtt{p}_{a,a'}^{b}:\mbox{RFH}_{*}^{\alpha}(H)_{a}^{b}\rightarrow\mbox{RFH}_{*}^{\alpha}(H)_{a'}^{b}.
\]
 Similarly given $-\infty<a<b<b'<\infty$ the inclusion $C^{\alpha}(f)_{a}^{b}\hookrightarrow C^{\alpha}(f)_{a}^{b'}$
induces maps 
\[
\mathtt{i}_{a}^{b,b'}:\mbox{RFH}_{*}^{\alpha}(H)_{a}^{b}\rightarrow\mbox{RFH}_{*}^{\alpha}(H)_{a}^{b'}.
\]
The complexes $\{\mbox{RFH}_{*}^{\alpha}(H)_{a}^{b},\mathtt{p},\mathtt{i}\}$
form a bidirect system of $\mathbb{Z}_{2}$-vector spaces, and hence
we can define 
\[
\mbox{RFH}_{*}^{\alpha}(H):=\underset{a\downarrow-\infty}{\underrightarrow{\lim}}\underset{b\uparrow\infty}{\underleftarrow{\lim}}\mbox{RFH}_{*}^{\alpha}(H)_{a}^{b}.
\]
In fact, suppose that $(\Sigma_{\pm},L,\alpha)$ are both non-degenerate
Rabinowitz admissible triples. Fix $H_{\pm}\in\mathcal{D}_{\textrm{ct}}(\Sigma_{\pm})$
and $\lambda_{\pm}$ are good primitives with respect to $(\Sigma_{\pm},L)$,
and  assume there exists a good homotopy $(H_{s},\lambda_{s})_{s\in\mathbb{R}}$
with $(H_{s},\lambda_{s})=(H_{-},\lambda_{-})$ for $s\leq0$ and
$(H_{s},\lambda_{s})=(H_{+},\lambda_{+})$ for $s\geq0$. Then one
can prove 
\[
\mbox{RFH}_{*}^{\alpha}(H_{-})\cong\mbox{RFH}_{*}^{\alpha}(H_{+}).
\]
This is a standard Floer theoretical argument, the key ingredient
being Theorem \ref{thm:second linfty}. Details can be found in \cite{Merry2011}.
It follows in particular that if $H\in\mathcal{D}_{\textrm{ct}}(\Sigma)$
then $\mbox{RFH}_{*}^{\alpha}(H)$ depends only on $\Sigma,L$,$X$
and $\alpha$. Thus we can finally make the following definition:
\begin{defn}
If $(\Sigma,L,\alpha)$ is a non-degenerate Rabinowitz admissible
triple, we define the \emph{Lagrangian Rabinowitz Floer homology}\textbf{\emph{
}}of $(\Sigma,L,X,\alpha)$ by 
\[
\mbox{RFH}_{*}^{\alpha}(\Sigma,L,X):=\mbox{RFH}_{*}^{\alpha}(H)\mbox{\ \ \ for any }H\in\mathcal{D}_{\textrm{\emph{ct}}}(\Sigma).
\]

\end{defn}
Moreover, since $\mbox{RFH}_{*}^{\alpha}(H)$ is invariant under good
homotopies we can even define the Lagrangian Rabinowitz Floer homology
$\mbox{RFH}_{*}^{\alpha}(\Sigma,L,X)$ even when $(\Sigma,L,\alpha)$
is not non-degenerate, simply by first isotopying $\Sigma$ through
good homotopies to a new hypersurface $\Sigma'$ such that $(\Sigma',L,\alpha)$
is Rabinowitz admissible and non-degenerate (such a hypersurface $\Sigma'$
exists by Lemma \ref{lem:more nondeg}), and then \emph{defining}
\[
\mbox{RFH}_{*}^{\alpha}(\Sigma,L,X):=\mbox{RFH}_{*}^{\alpha}(\Sigma',L,X).
\]

\subsection{\label{sec:Leaf-wise-intersection-points}Relative leaf-wise intersection
points}

Suppose $\psi:X\rightarrow X$ is a compactly supported Hamiltonian
diffeomorphism, and let $L$ denote a Lagrangian submanifold of $X$,
and $\Sigma$ a hypersurface whose intersection with $L$ is transverse
and non-empty. Fix $H\in\mathcal{D}_{\textrm{ct}}(\Sigma)$. Recall
from the Introduction that a \emph{relative leaf-wise intersection
point $p$ }of $\psi$ is a point $p\in\Sigma\cap L$ such that the
characteristic chord through $p$ intersects $\psi^{-1}(L)$. Equivalently,
$p\in\Sigma\cap L$ is a point with the property that there exists
$\tau\in\mathbb{R}$ such that 
\[
\psi(\phi_{H}^{\tau}(p))\in L.
\]

\begin{rem}
\label{rem:uniquely determined-1}It is of interest to know whether
$\tau$ is uniquely determined by $p$. This could fail if $\#\left(\{\phi_{H}^{t}(p)\}\cap\psi^{-1}(L)\right)>1$
or if the orbit$\{\phi_{H}^{t}(p)\}_{t\in\mathbb{R}}$ is closed.
However if $\dim\, X\geq4$ and the hypersurface $\Sigma$ itself
is non-degenerate (i.e. all periodic orbits of the flow $\phi_{H}^{t}$
are isolated on $\Sigma$ and after choosing a local transversal section
to a periodic orbit the linearized flow along this periodic orbit
has no eigenvalue equal to $1$ \textemdash{} see \cite[Section 1]{HoferWysockiZehnder1998}
for a precise definition) then for a generic choice of Hamiltonian
diffeomorphism neither of these things happen. These statements are
proved by arguing as in \cite[Theorem 3.3]{AlbersFrauenfelder2008}
and \cite[Lemma 8.2]{AbouzaidSeidel2010}. Since a generic hypersurface
is non-degenerate (see e.g. \cite[Appendix B]{CieliebakFrauenfelder2009}),
it follows that if $\dim\, X\geq4$ then `generically' $\tau$ is
uniquely determined by $p$.
\end{rem}
Suppose $p$ is a relative leaf-wise intersection point for which
the corresponding $\tau$ \emph{is }uniquely determined. Write $\psi=\phi_{F}^{1}$
for $F_{t}:X\rightarrow\mathbb{R}$ a compactly supported Hamiltonian
function. Consider the (not necessarily smooth) path $\zeta$ in $X$
which first travels from $x$ to $\phi_{H}^{\tau}(x)$ via $\phi_{H}^{t}$,
and then travels from $\phi_{H}^{\tau}(x)$ to $\psi(\phi_{H}^{\tau}(x))$
via $\phi_{F}^{t}$. Although $\zeta$ depends on the choice of Hamiltonians
$H$ and $F$, the class $\alpha\in\Pi_{L}$ does not. This is a standard
argument, which uses the fact any 1-periodic compactly supported Hamiltonian
function on $X$ has at least one contractible periodic orbit in the
interior of its support. Details can be found in several places; see
for instance \cite[Proposition 3.1]{Schwarz2000} or \cite[Lemma 3.7]{MacariniMerryPaternain2011}
(the latter reference deals specifically with leaf-wise intersection
points). In either case we say that the relative leaf-wise intersection
point \emph{belongs}\textbf{ }to the class $\alpha\in\Pi_{L}$. 

We now recall from Section \ref{sub:The-method-of} how a suitable
perturbation of the Rabinowitz action functional $\mathcal{A}_{H}$
gives rise to a new functional which detects the relative leaf-wise
intersection points of $\psi$ which belong to a given $\alpha\in\Pi_{L}$.
Let $\beta:S^{1}\rightarrow\mathbb{R}$ denote a smooth function with
\[
\beta(t)=0\ \forall t\in[\tfrac{1}{2},1],\ \ \ \mbox{and}\ \ \ \int_{0}^{1}\beta(t)dt=1,
\]
and let $\chi:[0,1]\rightarrow[0,1]$ is a smooth monotone map with
$\chi(\tfrac{1}{2})=0$ and $\chi(1)=1$. Fix $\alpha\in\Pi_{L}$
satisfying condition \textbf{(A) }\vpageref{enu:If--is-2} and define
\[
\mathcal{A}_{H}^{F}:P_{\alpha}(X,L)\times\mathbb{R}\rightarrow\mathbb{R}
\]
by 
\[
\mathcal{A}_{H}^{F}(x,\tau):=\Omega(x)-\tau\int_{0}^{1}\beta(t)H(x(t))dt-\int_{0}^{1}\dot{\chi}(t)F_{\chi(t)}(x(t))dt.
\]
Denote by $\mbox{Crit}^{\alpha}(\mathcal{A}_{H}^{F})$ the set of
critical points of $\mathcal{A}_{H}^{F}$. Note that a pair $(x,\tau)$
belongs to $\mbox{Crit}(\mathcal{A}_{H}^{F})$ if and only if:
\[
\dot{x}=\tau\beta X_{H}(x)+\dot{\chi}X_{F_{\chi}}(x);
\]
\[
\int_{0}^{1}\beta(t)H(x)dt=0.
\]
Since $\beta(t)H(x)$ and $\dot{\chi}(t)F_{\chi(t)}(t,x)$ have disjoint
time support, critical points of $\mathcal{A}_{H}^{F}$ first follow
the flow of $\phi_{H}^{\int_{0}^{t}\beta(s)ds}$ in time $[0,\tfrac{1}{2}]$
and then follow the flow of $\phi_{F}^{\chi(t)}$ in time $[\tfrac{1}{2},1]$.
This leads to the following observation, which is proved in the same
way as \cite[Proposition 2.4]{AlbersFrauenfelder2010c}, and explains
why Lagrangian Rabinowitz Floer homology is useful in the study of
relative leaf-wise intersection points.
\begin{lem}
\label{lwip}There is a surjective map 
\[
e:\mbox{\emph{Crit}}^{\alpha}(\mathcal{A}_{H}^{F})\rightarrow\{\mbox{relative leaf-wise intersection points of }\psi\mbox{ belonging to }\alpha\}
\]
 given by 
\[
e(x,\tau):=x(0).
\]
If $\dim\, X\geq4$ then generically (in the sense of Remark \ref{rem:uniquely determined-1}),
the map $e$ is a bijection.\end{lem}
\begin{defn}
A critical point $(x,\tau)\in\mbox{Crit}^{\alpha}(\mathcal{A}_{H}^{F})$
is called \emph{non-degenerate}\textbf{\emph{ }}if $\nabla_{\mathbf{J}}^{2}\mathcal{A}_{H}^{F}(x,\tau)$
is injective (cf. \eqref{eq:linearize}).
\end{defn}
The proof of the following result, which is the analogue of Lemma
\ref{lem:more nondeg}, is very similar to \cite[Appendix A]{AlbersFrauenfelder2008}.
\begin{thm}
\label{thm:the morse theorem}There is a generic subset of $C_{0}^{\infty}(S^{1}\times X,\mathbb{R})$
with the property that if $F_{t}$ belongs to this set then every
critical point of the corresponding perturbed Rabinowitz action functional
$\mathcal{A}_{H}^{F}$ is non-degenerate. 
\end{thm}
Suppose now that $(\Sigma,L,\alpha)$ is a Rabinowitz admissible triple,
and $H\in\mathcal{D}_{\textrm{ct}}(\Sigma)$. If every critical point
of the perturbed Rabinowitz action functional $\mathcal{A}_{H}^{F}$
is non-degenerate, we can define the Lagrangian Rabinowitz Floer homology
$\mbox{RFH}_{*}^{\alpha}(H,F)$. This is defined in exactly the same
way as before, only since we are now in a Morse situation, no additional
Morse function $f$ is needed. Moreover, by choosing an $s$-dependent
homotopy $F_{s}$ from $F$ to $0$, one sees that the usual continuation
homomorphisms are well-defined and isomorphisms. Thus we conclude:
\[
\mbox{RFH}_{*}^{\alpha}(H,F)\cong\mbox{RFH}_{*}^{\alpha}(H)
\]
(see \cite[Section 2.3]{AlbersFrauenfelder2008} or \cite[Section 6]{Merry2011}
for more information). In particular, if one can find a Hamiltonian
diffeomorphism $\psi:X\rightarrow X$ such that there are no relative
leaf-wise intersection points belonging to $\alpha$, then $\mbox{Crit}^{\alpha}(\mathcal{A}_{H}^{F})=\emptyset$
for any function $F_{t}$ such that $\psi=\phi_{F}^{1}$. In this
case $\mathcal{A}_{H}^{F}$ is trivially Morse, and hence $\mbox{RFH}_{*}^{\alpha}(H,F)$
is defined and equal to zero. We can now complete the proof of Theorem
\ref{thm:big 1}. Let us recall the statement (written more concisely).
\begin{thm}
\label{thm:big 1-1}Assume that $(\Sigma,L,0)$ is a Rabinowitz admissible
triple. If there exists a compactly supported Hamiltonian diffeomorphism
$\psi:X\rightarrow X$ with no relative leaf-wise intersection points
(e.g. if one can displace $\Sigma$ from $L$) then there exists a
characteristic chord in $\Sigma$ with endpoints in $\Sigma\cap L$.\end{thm}
\begin{proof}
Suppose there are no characteristic chords in $\Sigma$ with endpoints
in $\Sigma\cap L$. Then $(\Sigma,L,0)$ is trivially non-degenerate.
Moreover for any $H\in\mathcal{D}_{\textrm{ct}}(\Sigma)$ and any
Morse function $f$ on $\mbox{Crit}^{0}(\mathcal{A}_{H})\cong\Sigma\cap L$
it is clear that the Rabinowitz Floer complex $\{\mbox{CRF}_{*}^{0}(H,f),\partial\}$
reduces to the Morse complex $\{\mbox{CM}_{*+(n-1)/2}(f),\partial^{\textrm{Morse}}\}$,
and thus the Rabinowitz Floer homology agrees with the Morse homology
of $f$ (modulo a grading shift). In particular, it is non-zero. But
now if there existed a compactly supported Hamiltonian diffeomorphism
$\psi:X\rightarrow X$ with no relative leaf-wise intersection points
belonging to $0$, then as we have just seen this would imply that
$\mbox{RFH}_{*}^{0}(\Sigma,L,X)=0$.\end{proof}
\begin{rem}
The proof of Theorem \ref{thm:baby 1} goes along exactly the same
lines, although note that in order to get started one first isotopes
$L$ relative to $\partial X_{0}$ so that \ref{eq:nice condition}
is satisfied \textemdash{} this can be done without affecting whether
$\Sigma$ is displaceable from $L$ or not. 
\end{rem}

\section{Computing the Lagrangian Rabinowitz Floer homology on twisted cotangent
bundles}

We now move on to the setting of twisted cotangent bundles, with the
aim of proving Theorem \ref{thm:big 2}. As in the previous section,
much of the following material can be simplified if the reader is
only interested in Theorem \ref{thm:baby 2}, and we will point some
of these out in Remark \ref{rem simple case comp}. We now recall
the setup: let $M$ denote a closed connected orientable $n$-dimensional
manifold, where $n\geq2$. Let $\pi:T^{*}M\rightarrow M$ denote the
footpoint map $\pi(q,p)\mapsto q$, and let $\rho:\widetilde{M}\rightarrow M$
denote the universal cover of $M$. We write $\rho_{\sharp}:T^{*}\widetilde{M}\rightarrow T^{*}M$
for the map defined by $\rho_{\sharp}(p):=\left(D\rho(q)^{-1}\right)^{*}(p)$
for $p\in T_{\rho(q)}^{*}\widetilde{M}$. Let $\lambda_{\textrm{can}}\in\Omega^{1}(T^{*}M)$
denote the Liouville 1-form, defined by $\lambda_{\textrm{can}}=\sum_{j}p_{j}dq_{j}$
in local coordinates $(q,p)$. Suppose $\sigma\in\Omega^{2}(M)$ is
a \emph{closed }2-form. We pull $\sigma$ back to $T^{*}M$ and add
it to $d\lambda_{\textrm{can}}$ to obtain a new symplectic form 
\[
\omega:=d\lambda_{\textrm{can}}+\pi^{*}\sigma
\]
 on $T^{*}M$. We will always insist that $\sigma$ is\textbf{\emph{
}}\emph{weakly exact}, that is, the lift $\widetilde{\sigma}:=\rho^{*}\sigma\in\Omega^{2}(\widetilde{M})$
is exact (this is equivalent to requiring that $\sigma|_{\pi_{2}(M)}=0$).
In fact, we will always make the additional assumption that $\widetilde{\sigma}$
admits a \emph{bounded} primitive: there exists $\varphi\in\Omega^{1}(\widetilde{M})$
such that $d\varphi=\widetilde{\sigma}$ and such that 
\begin{equation}
\sup_{q\in\widetilde{M}}\left|\varphi_{q}\right|<\infty,\label{eq:bounded-1-1}
\end{equation}
where the norm $\left|\cdot\right|$ is given by the lift of any Riemannian
metric on $M$ to $\widetilde{M}$. When discussing cotangent bundles,
it is more convenient to fix once and for all a point $\star\in M$
as a reference point, and then take $0_{\star}\in T_{\star}^{*}M$
to be our fixed reference point in $T^{*}M$. When discussing submanifolds
$S$ of $M$, we always implicitly assume that $\star\in S$ (note
this implies $0_{\star}\in N^{*}S$). We also fix a point $\widetilde{\star}\in\widetilde{M}$
that projects onto $\star$. We denote by $P(M,S)$ the space of smooth
paths $q:[0,1]\rightarrow M$ with $q(0)\in S$ and $q(1)\in S$.
We define $\Pi_{S}$ in exactly the same way as $\Pi_{L}$ was defined
in \eqref{eq:pi L-1}, only with $M$ and $S$ replacing $X$ and
$L$. Then $\Pi_{S}$ indexes the connected components of $P(M,S)$,
and given $\alpha\in\Pi_{S}$ we let $P_{\alpha}(M,S)$ denote the
corresponding connected component. If $x=(q,p)\in P(T^{*}M,N^{*}S)$
then $q\in P(M,S)$, and under the obvious identification $\Pi_{S}\cong\Pi_{N^{*}S}$,
if $x\in P_{\alpha}(T^{*}M,N^{*}S)$ then $q\in P_{\alpha}(M,S)$.
In particular, if we write our reference loops $x_{\alpha}$ as $(q_{\alpha},p_{\alpha})$,
then $q_{\alpha}$ serves as a reference loop in $P_{\alpha}(M,S)$.
One nice consequence of \eqref{eq:bounded-1-1} is that every class
$\alpha\in\Pi_{S}$ satisfies the condition \textbf{(A) }\vpageref{enu:If--is-2}.
\begin{lem}
\emph{\label{lem:key observation-1-1}}Suppose $S\subseteq M$ is
a closed connected submanifold such that $\sigma|_{S}=0$ and such
that $N^{*}S$ is a virtually exact Lagrangian submanifold of $(T^{*}M,\omega_{\sigma})$.
Then for every path $f:S^{1}\times[0,1]\rightarrow T^{*}M$ with $f(S^{1}\times\{0,1\})\subset N^{*}S$,
one has $\int_{S^{1}\times[0,1]}f^{*}\omega_{\sigma}=0$. \end{lem}
\begin{proof}
The symplectic area functional $\Omega$ of $(X,\omega_{\sigma})$
can be expressed as 
\[
\Omega=\Omega_{\textrm{can}}+\pi^{*}\Omega_{\sigma},
\]
where 
\[
\Omega_{\textrm{can}}(x):=\int_{0}^{1}x^{*}\lambda_{\textrm{can}},
\]
and $\Omega_{\sigma}$ is the \emph{$\sigma$-area}\textbf{ }defined
by 
\[
\Omega_{\sigma}(q):=\int_{[0,1]\times[0,1]}\bar{q}^{*}\sigma,
\]
where $\bar{q}$ is any filling of $q$ (i.e. any smooth map $\bar{q}:[0,1]\times[0,1]\rightarrow M$
with $\bar{q}(0,t)=q(t),$ $\bar{q}(1,t)=q_{\alpha}(t)$ and $\bar{q}([0,1]\times\{0,1\})\subset S$).
It thus suffices to show that if $f:S^{1}\times[0,1]\rightarrow M$
satisfies $f(S^{1}\times\{0,1\})\subset S$ then $\int_{S^{1}\times[0,1]}f^{*}\sigma=0$.
Fix a bounded primitive $\varphi$ of $\widetilde{\sigma}$ with the
property that $\varphi|_{\widetilde{S}}=ds$ for some bounded function
$s\in C^{\infty}(\widetilde{S},\mathbb{R})$. Consider $G:=f_{*}(\pi_{1}(S^{1}\times[0,1]))\leq\pi_{1}(M).$
Then $G$ is \emph{amenable}, since $\pi_{1}(S^{1}\times[0,1])=\mathbb{Z}$,
which is amenable. Then \cite[Lemma 5.3]{Paternain2006} tells us
that since $\left\Vert \varphi\right\Vert _{L^{\infty}}<\infty$ and
$\left\Vert s\right\Vert _{L^{\infty}}<\infty$, we can replace $\varphi$
by a $G$-invariant primitive $\varphi'$ of $\widetilde{\sigma}$,
and $s$ by a $G$-invariant function $s'$ satisfying $\varphi'|_{\widetilde{S}}=ds'$.
Thus $\varphi'$ and $s'$ descend to define a primitive $\varphi''\in\Omega^{1}(S^{1}\times[0,1])$
of $f^{*}\sigma$ and a function $s''\in C^{\infty}(S^{1}\times\{0,1\},\mathbb{R})$
with the property that $\varphi''|_{S^{1}\times\{0,1\}}=ds''$. Hence
by Stokes' Theorem, $\int_{S^{1}\times[0,1]}f^{*}\sigma=0$ as required.
\end{proof}

\subsection{\label{sub:Ma-critical-values}The Ma\~n\'e critical value}

We now recall the definition of the critical value $c(H,\sigma)$,
as introduced by Ma\~n\'e in \cite{Mane1996}. General references
for the results stated below are \cite[Proposition 2-1.1]{ContrerasIturriaga1999}
or \cite[Appendix A]{BurnsPaternain2002}. We then explain how to
modify the definition of the critical value $c(H,\sigma)$ to take
into account a given $\pi_{1}$-injective submanifold $S\subseteq M$
for which $\sigma|_{S}=0$. This leads to a new critical value $c(H,\sigma,S)$.
Fix an autonomous Tonelli Hamiltonian $H\in C^{\infty}(T^{*}M,\mathbb{R})$,
and denote by $\widetilde{H}\in C^{\infty}(T^{*}\widetilde{M},\mathbb{R})$
the lift of $H$ to the universal cover $T^{*}\widetilde{M}$. We
define the the \emph{Ma\~n\'e critical value }associated to $H$
and $\sigma$ by\emph{ }
\begin{equation}
c(H,\sigma):=\inf_{\varphi}\sup_{q\in\widetilde{M}}\widetilde{H}(q,-\varphi_{q}),\label{eq:manc-1}
\end{equation}
where the infimum is taken over all primitives $\varphi$ of $\widetilde{\sigma}$.
Since $H$ is superlinear, $c(H,\sigma)<\infty$ if and only if $\widetilde{\sigma}$
admits a bounded primitive.
\begin{rem}
The strange looking sign convention in \eqref{eq:manc-1} is due to
the fact that we are using the ``unnatural'' sign convention that
the canonical symplectic form on $T^{*}M$ is given by $d\lambda_{\textrm{can}}$
(rather than $-d\lambda_{\textrm{can}}$).
\end{rem}
Suppose now we bring into the picture a closed connected $\pi_{1}$-injective
submanifold $S\subseteq M$ such that $\sigma|_{S}=0$. Then we are
only permitted to use primitives $\varphi$ of $\widetilde{\sigma}$
with the property that $\varphi|_{\widetilde{S}}=ds$ for some bounded
function $s:\widetilde{S}\rightarrow\mathbb{R}$. Taking the infimum
over only these primitives we obtain a new critical value $c(H,\sigma,S)$.
If no such primitives exist we set $c(H,\sigma,S)=\infty$. Clearly
\[
c(H,\sigma)\leq c(H,\sigma,S),
\]
and $c(H,\sigma,S)$ is finite if and only if $N^{*}S$ is virtually
exact in the sense of Definition \ref{def virtually exact}. Now recall
from Section \ref{sub:A-more-complicated} the following definition:
\begin{defn}
\label{def MSCP-1}Consider a closed connected hypersurface $\Sigma\subset T^{*}M$
and a closed connected submanifold $S\subseteq M$ such that $\sigma|_{S}=0$,
with $\Sigma\cap N^{*}S\ne\emptyset$ and $\Sigma\pitchfork N^{*}S$.
The pair $(\Sigma,S)$ is called a \emph{Ma\~n\'e supercritical}\textbf{
}\emph{pair} if there exists a Tonelli Hamiltonian $H:T^{*}M\rightarrow\mathbb{R}$
with $c(H,\sigma,S)<0$, and such that $\Sigma$ is the regular level
set $H^{-1}(0)$.
\end{defn}
We now prove that Ma\~n\'e supercritical\textbf{\emph{ }}pairs exactly
fit into the framework of Theorem \ref{thm:big 1}. Recall the notion
of a Rabinowitz admissible triple from Definition \ref{rab adm lag-1}.
\begin{lem}
\label{lem:MSP IMPLIES RAT}Suppose that $(\Sigma,S)$ is a Ma\~n\'e
supercritical\textbf{\emph{ }}pair. Then for any $\alpha\in\Pi_{S}$,
the triple $(\Sigma,N^{*}S,\alpha)$ is a Rabinowitz admissible triple.\end{lem}
\begin{proof}
First let us show that if $H$ is a Tonelli Hamiltonian such that
$c(H,\sigma)<0$ then $\Sigma:=H^{-1}(0)$ is a hypersurface of virtual
restricted contact type (cf. Definition \ref{thm:VRCT}). By the definition
\eqref{eq:manc-1} of $c(H,\sigma)$ there exists $\varepsilon>0$
and a bounded primitive $\varphi$ of $\widetilde{\sigma}$ such that
the lift $\widetilde{H}$ of $H$ satisfies $\widetilde{H}(q,-\varphi_{q})<-\varepsilon$
for all $q\in\widetilde{M}$. Set $\lambda:=\widetilde{\mu}_{\textrm{can}}+\widetilde{\pi}^{*}\varphi$,
where $\widetilde{\mu}_{\textrm{can}}$ is the Liouville 1-form on
$T^{*}\widetilde{M}$. Since $\varphi$ is bounded, we need only check
that 
\begin{equation}
\inf_{(q,p)\in\widetilde{\Sigma}}\lambda(X_{\widetilde{H}}(q,p))>0,\label{eq:to show}
\end{equation}
where $X_{\widetilde{H}}$ is the symplectic gradient of $\widetilde{H}$
with respect to the lifted symplectic form $\widetilde{\omega}:=d\widetilde{\mu}_{\textrm{can}}+\widetilde{\pi}^{*}\widetilde{\sigma}$.
Fix $(q,p)\in\widetilde{\Sigma}$, and let 
\[
f(s):=\widetilde{H}(q,(1-s)\varphi_{q}+sp).
\]
A simple computation yields
\[
\lambda(X_{\widetilde{H}}(q,p))=f'(1).
\]
Now note that $f(0)<-\varepsilon$ and $f(1)=0$, and since $H$ is
Tonelli, $f$ is convex and thus we must have $f'(1)>\varepsilon$.
This proves \eqref{eq:to show}. Moreover if we assume the stronger
assumption that $c(H,\sigma,S)<0$ then exactly the same argument
shows that there exists good primitives (in the sense of Definition
\ref{def good primitive}). Combined with Lemma \ref{lem:key observation-1-1}
this completes the proof. \end{proof}
\begin{example}
\label{thm:Alberto-example}Here is an example (due to Alberto Abbondandolo)
that illustrates the difference between simply asking that $c(H,\sigma)<0$
and asking that $c(H,\sigma,S)<0$. Take $M=\mathbb{T}^{n}$ and $\sigma=0$,
and take $S=S^{1}\times\{\mbox{pt}\}$. Define $H:T^{*}\mathbb{T}^{n}\rightarrow\mathbb{R}$
by 
\[
H(q,p):=\frac{1}{2}\left|p-dq_{1}\right|^{2}.
\]
One easily sees that 
\[
c(H,\sigma)=0,
\]
but that 
\[
c(H,\sigma,S)=1/2.
\]
In fact, $H^{-1}(k)\cap N^{*}S=\emptyset$ if $k<1/2$. For $k>1/2$,
not only is $H^{-1}(k)\cap N^{*}S$ non-empty, but it follows from
Theorem \ref{thm:big 2} that the hypersurface $H^{-1}(k)$ can never
be displaced from $N^{*}S$ by an element of $\mbox{Ham}_{c}(T^{*}M,d\lambda_{\textrm{can}})$.
$ $
\end{example}

\subsection{\label{sub:Lagrangian-Rabinowitz-Floer with tonelli}Lagrangian Rabinowitz
Floer homology with Tonelli Hamiltonians}

Since from now on we will be working exclusively with Ma\~n\'e supercritical\textbf{\emph{
}}pairs, it would be nice to work directly with the Tonelli Hamiltonian
$H$ which cuts $\Sigma$ out as its regular level set $H^{-1}(0)$.
Such a Tonelli Hamiltonian $H$ belongs to $\mathcal{D}(\Sigma)$,
but\textbf{ }since Tonelli Hamiltonians are \emph{not}\textbf{\emph{
}}constant outside a compact set, it does not belong to $\mathcal{D}_{\textrm{ct}}(\Sigma)$.
Thus it is not a priori clear that one can use $H$ to define the
Lagrangian Rabinowitz Floer homology of $(\Sigma,N^{*}S,T^{*}M)$,
and even if we could, whether it would yield the same Lagrangian Rabinowitz
Floer homology as the one developed in Section \ref{sub:The-definition-of}.
The key difficulty here is that as $H$ is no longer constant outside
a compact set, a lot more work is required to prove the the compactness
results in Theorem \ref{thm:first linfty} and Theorem \ref{thm:second linfty}
(see Remark \ref{rem:why constant is good}). In \cite{AbbondandoloSchwarz2006}
Abbondandolo and Schwarz showed how such compactness could still be
obtained (in the setting of ``standard'' Floer homology on cotangent
bundles equipped with standard symplectic form $d\lambda_{\textrm{can}}$)
for a wide class of Hamiltonians. Roughly speaking, they proved $L^{\infty}$
estimates for Hamiltonians that, outside of a compact set, are \emph{quadratic}\textbf{
}in the fibres (see \cite[Section 1.5]{AbbondandoloSchwarz2006} for
the precise definition). Their idea is based upon isometrically embedding
$T^{*}M$ into $\mathbb{R}^{2N}$ (via Nash's theorem), and combining
Calderon-Zygmund estimates for the Cauchy-Riemann operator with certain
interpolation inequalities. We remark that in order for these $L^{\infty}$
estimates to hold it is important that the almost complex structure
we choose lies sufficiently close (in the $L^{\infty}$ norm) to the
\emph{metric almost complex structure}\textbf{ }$J_{g}$ associated
to some fixed Riemannian metric $g=\left\langle \cdot,\cdot\right\rangle $
on $M$. This is the unique almost complex structure on $T^{*}M$
with the property that under the splitting $TT^{*}M\cong TM\oplus T^{*}M$
determined by the metric (see Section \ref{sub:The sigma action}
below), $J_{g}$ acts as 
\[
J_{g}=\left(\begin{array}{cc}
0 & -\mathbb{I}\\
\mathbb{I} & 0
\end{array}\right).
\]
A Tonelli Hamiltonian $H\in C^{\infty}(T^{*}M,\mathbb{R})$ is \emph{electromagnetic
at infinity} (with respect to $g$) if there exists a positive function
$a\in C^{\infty}(M,\mathbb{R}^{+})$, a 1-form $\beta\in\Omega^{1}(M)$,
a function $V\in C^{\infty}(M,\mathbb{R})$, and a real number $R>0$
such that 
\[
H(q,p)=\frac{1}{2}a(q)\left|p-\beta_{q}\right|^{2}+V(q)\ \ \ \mbox{for all }(q,p)\in T^{*}M\mbox{ with }\left|p\right|\geq R.
\]
The following result is a minor variant of \cite[Corollary 20]{ContrerasIturriagaPaternainPaternain2000}. 
\begin{prop}
\label{RI}Suppose $\Sigma=H^{-1}(0)$ is a regular energy value of
a Tonelli Hamiltonian $H\in C^{\infty}(T^{*}M,\mathbb{R})$ with $c(H,\sigma,S)<0$.
Then there exists another Tonelli Hamiltonian $\overline{H}$ that
is electromagnetic at infinity and satisfies:
\[
H\equiv\overline{H}\ \ \ \mbox{on }\{H\leq1\};
\]
\[
c(H,\sigma,S)=c(\overline{H},\sigma,S).
\]

\end{prop}
In \cite{Merry2011} we use a version of the argument of Abbondandolo
and Schwarz mentioned above to show that the Lagrangian Rabinowitz
Floer homology $\mbox{RFH}_{*}^{\alpha}(H)$ is well defined when
$H$ is a Tonelli Hamiltonian which is electromagnetic at infinity
and satisfies $c(H,\sigma,S)<0$, and moreover that this Lagrangian
Rabinowitz Floer homology is the same as the one defined using Hamiltonians
which are constant outside a compact set. Actually strictly speaking
in order for this result to hold, one may need to rescale $\sigma$
(this is so $\omega$-compatible almost complex structures that are
sufficiently close in the $L^{\infty}$-norm to the metric almost
complex structure $J_{g}$ exist); this does not actually entail any
loss of generality, as the Lagrangian Rabinowitz Floer homology of
Section \ref{sub:The-definition-of} is invariant under such rescaling.
As such we will ignore this subtlety throughout. See \cite[Lemma 8.12]{Merry2011}. 

From now on we fix a Ma\~n\'e supercritical pair $(\Sigma,S)$. Without
loss of generality (as far as the Lagrangian Rabinowitz Floer homology
is concerned), we can and will assume that $(\Sigma,N^{*}S,\alpha)$
is non-degenerate for every $\alpha\in\Pi_{S}$. Proposition \ref{RI}
implies that we may choose a Tonelli Hamiltonian $H\in C^{\infty}(T^{*}M,\mathbb{R})$
that is electromagnetic at infinity and satisfies $\Sigma=H^{-1}(0)$
with $c(H,\sigma,S)<0$, and thus we may compute the Lagrangian Rabinowitz
Floer homology $\mbox{RFH}_{*}^{\alpha}(\Sigma,N^{*}S,T^{*}M)$ using
$H$:
\[
\mbox{RFH}_{*}^{\alpha}(H)\cong\mbox{RFH}_{*}^{\alpha}(\Sigma,N^{*}S,T^{*}M).
\]
The aim of the rest of this paper is to compute $\mbox{RFH}_{*}^{\alpha}(H)$. 
\begin{rem}
\label{rem simple case comp}For the reader only interested in Theorem
\ref{thm:baby 2}, where $\sigma=0$ and $\Sigma=U^{*}M$, we can
drop the assumption that $S$ is $\pi_{1}$-injective, since we no
longer need to lift anything to the universal cover. For our Hamiltonian
$H$ we take $H(q,p)=\frac{1}{2}\left|p\right|^{2}-\frac{1}{2}$.
In this case for any submanifold $S\subseteq M$ one has $c(H,0,S)=-\frac{1}{2}$.
Thus (so far) there are no restrictions on $S$ apart from our standing
assumption that $U^{*}M\cap N^{*}S\ne\emptyset$ and $U^{*}M\pitchfork N^{*}S$
, although as stated in Theorem \ref{thm:baby 2} we will eventually
need to make an additional assumption on $S$ in order to compute
the Lagrangian Rabinowitz Floer homology \textemdash{} see Theorem
\ref{thm:theorem A precise}.
\end{rem}

\subsection{Grading}

Before getting started on computing $\mbox{RFH}_{*}^{\alpha}(H)$,
we will spend a little time discussing the grading on Lagrangian Rabinowitz
Floer homology in the specialized situation we are now working in.
In fact, there is a particularly satisfying solution to the grading
issue on twisted cotangent bundles. This is because every twisted
cotangent bundle possesses a \emph{Lagrangian distribution}, namely
the vertical distribution $T^{\mathsf{v}}T^{*}M$ (i.e. the tangent
spaces to the fibres: $T_{(q,p)}^{\mathsf{v}}T^{*}M:=T_{(q,p)}T_{q}^{*}M$).
The vertical distribution singles out a distinguished class of symplectic
trivializations \textemdash{} those that are vertical preserving.
Namely, if $x\in P(T^{*}M,N^{*}S)$, a trivialization $\Phi:[0,1]\times\mathbb{R}^{2n}\rightarrow x^{*}TT^{*}M$
is called \emph{vertical preserving}\textbf{ }if 
\[
\Phi(t,V_{0})=T_{x(t)}^{\mathsf{v}}T^{*}M\ \ \ \mbox{for all }t\in[0,1],
\]
where $V_{0}:=\{0\}\times\mathbb{R}^{n}$ is the vertical subspace.
Such trivializations always exist (cf. \cite[Lemma 1.2]{AbbondandoloSchwarz2006}).
Suppose we are now given a critical point $(x,\tau)$ of the Rabinowitz
action functional $\mathcal{A}_{H}$. Let $\Phi:[0,1]\times\mathbb{R}^{2n}\rightarrow x^{*}TT^{*}M$
denote a vertical preserving trivialization, and define a path $\vartheta:[0,1]\rightarrow\mathsf{Lag}(\mathbb{R}^{2n},\omega_{\textrm{std}})$
by 
\[
\Phi(t,\vartheta(t))=D\phi_{H}^{\tau t}(x(0))(T_{x(0)}^{\mathsf{v}}T^{*}M).
\]
Now define 
\[
\mu_{\mathsf{Ma}}(x,\tau):=\mu_{\mathsf{RS}}(\vartheta,V_{0}),
\]
where $\mu_{\mathsf{RS}}$ is the \emph{Robbin-Salamon index}\textbf{
}\cite{RobbinSalamon1993} (although be warned - our sign convention
for $\mu_{\textrm{RS}}$ matches \cite{AbbondandoloPortaluriSchwarz2008}
rather than \cite{RobbinSalamon1993}). This index $\mu_{\mathsf{Ma}}(x,\tau)$
is independent of the vertical preserving trivialization $\Phi$ (cf.
\cite[Lemma 1.3.(ii)]{AbbondandoloSchwarz2006}). In fact it will
also be convenient to introduce a grading shift of $d-\frac{n-1}{2}$
(recall $d=\dim\, S$). This choice is motivated by Theorem \ref{thm:(The-Morse-index}
below, and it also ensures our grading is always $\mathbb{Z}$-valued.
Thus for the remainder of the paper, instead of using the convention
from Definition \ref{mu} we define
\[
\mu(x,\tau):=\begin{cases}
\mu_{\mathsf{Ma}}(x,\tau)-\frac{1}{2}\chi(x,\tau)+d-\frac{n-1}{2}, & \tau\ne0,\\
d-n+1, & \tau=0.
\end{cases}
\]

\subsection{\label{sub:The sigma action}The free time action functional}

In this section we work on the tangent bundle $TM$ instead of the
cotangent bundle $T^{*}M$. Denote again by $\pi$ the footpoint map
$TM\rightarrow M$. Let us fix once and for all an auxiliary Riemannian
metric $g$ on $M$. The Riemannian metric $g$ defines a \emph{horizontal-vertical}
splitting of $TTM$: given $(q,v)\in TM$ we write
\[
T_{(q,v)}TM=T_{(q,v)}^{\mathsf{h}}TM\oplus T_{(q,v)}^{\mathsf{v}}TM\cong T_{q}M\oplus T_{q}M;
\]
here $T_{(q,v)}^{\mathsf{h}}TM=\ker(\kappa_{g}:T_{(q,v)}TM\rightarrow T_{q}M)$,
where $\kappa_{g}$ is the connection map of the Levi-Civita connection
$\nabla$ of $g$, and $T_{(q,v)}^{\mathsf{v}}TM=\ker(D\pi(q,v):T_{(q,v)}TM\rightarrow T_{q}M)$.
Given $\xi\in TTM$ we denote by $\xi^{\mathsf{h}}$ and $\xi^{\mathsf{v}}$
the horizontal and vertical components. The \emph{Sasaki metric}\textbf{
}$g_{TM}$ on $TM$ is defined by 
\[
g_{TM}(\xi,\vartheta):=\left\langle \xi^{\mathsf{h}},\vartheta^{\mathsf{h}}\right\rangle +\left\langle \xi^{\mathsf{v}},\vartheta^{\mathsf{v}}\right\rangle .
\]
Suppose $f\in C^{\infty}(TM,\mathbb{R})$ is an arbitrary smooth function.
Then $df(q,v)\in T_{(q,v)}^{*}TM$, and thus its gradient $\nabla f(q,v)=\nabla_{g_{TM}}f(q,v)$
lies in $T_{(q,v)}TM$. Thus we can speak of the horizontal and vertical
components 
\[
\nabla f^{\mathsf{h}}(q,v):=[\nabla f(q,v)]^{\mathsf{h}}\in T_{q}M;
\]
\[
\nabla f^{\mathsf{v}}(q,v):=[\nabla f(q,v)]^{\mathsf{v}}\in T_{q}M.
\]
Let us go back to our fixed Hamiltonian $H$. The fact that $H$ is
Tonelli implies there exists a unique Tonelli Lagrangian (that is,
fibrewise strictly convex and superlinear) $L\in C^{\infty}(TM,\mathbb{R})$
called the \emph{Fenchel dual}\textbf{\emph{ }}\emph{Lagrangian} to
$H$, which is related to $H$ by
\[
H(q,p):=p(v)-L(q,v),\ \ \ \mbox{where }\nabla L^{\mathsf{v}}(q,v)=p.
\]
Since $H$ is electromagnetic at infinity, so is $L$. That is, there
exists a positive function $a\in C^{\infty}(M,\mathbb{R}^{+})$, a
1-form $\beta\in\Omega^{1}(M)$, a function $V\in C^{\infty}(M,\mathbb{R})$,
and a real number $R>0$ 
\[
L(q,v)=\frac{1}{2}a(q)\left|v\right|^{2}+\beta_{q}(v)-V(q)\ \ \ \mbox{for all }(q,v)\in TM\mbox{ with }\left|p\right|\geq R.
\]

\begin{defn}
We denote by $\mathcal{P}(M,S)$ the the $W^{1,2}$ Sobolev completion
$\mathcal{P}(M,S)$ of $P(M,S)$. Abbreviating $\mathbb{R}^{+}:=(0,\infty)$,
we equip $\mathcal{P}(M,S)\times\mathbb{R}^{+}$ with the natural
product Riemannian structure 
\begin{equation}
\left\langle \left\langle (\eta,h),(\eta',h')\right\rangle \right\rangle _{W^{1,2}}:=\int_{0}^{1}\left\langle \eta,\eta'\right\rangle dt+\int_{0}^{1}\left\langle \nabla_{t}\eta,\nabla_{t}\eta'\right\rangle dt+hh',\label{eq:W12}
\end{equation}
where $\nabla$ denotes the Levi-Civita connection of $(M,g)$. 
\end{defn}
Recall from the proof of Lemma \ref{lem:key observation-1-1} that
\emph{$\sigma$-area}\textbf{ }$\Omega_{\sigma}:P(M,S)\rightarrow\mathbb{R}$
is defined by 
\[
\Omega_{\sigma}(q):=\int_{[0,1]\times[0,1]}\bar{q}^{*}\sigma,
\]
where $q\in P_{\alpha}(M,S)$ and $\bar{q}$ is any filling of $q$
(i.e. any smooth map such that $\bar{q}(0,t)=q(t)$, $\bar{q}(1,t)=q_{\alpha}(t)$
and $\bar{q}([0,1]\times\{0,1\})\subset S$). Let us note that $\Omega_{\sigma}$
extends to a functional on $\mathcal{P}(M,S)$: if $q$ is of class
$W^{1,2}$ then we choose the filling $\bar{q}\in W^{1,2}([0,1]\times[0,1],M)\cap C^{0}([0,1]\times[0,1],M)$.
We will study the \emph{free time action functional}\textbf{ }
\[
\mathcal{S}_{L}:\mathcal{P}_{\alpha}(M,S)\times\mathbb{R}^{+}\rightarrow\mathbb{R}
\]
which is defined by 
\[
\mathcal{S}_{L}(q,\tau):=\Omega_{\sigma}(q)+\tau\int_{0}^{1}L\left(q,\frac{\dot{q}}{\tau}\right)dt.
\]
In the case $\sigma=0$, the functional $\mathcal{S}_{L}$ has been
extensively studied in \cite{ContrerasIturriagaPaternainPaternain2000,Contreras2006}.
The pair $(\sigma,g)$ defines a bundle endomorphism $Y=Y_{\sigma,g}\in\Gamma(\mbox{End}(TM))$
called the\textbf{ }\emph{Lorentz force} of $\sigma$ via:
\[
\sigma_{q}(u,v)=\left\langle Y(q)u,v\right\rangle .
\]
A standard computation tells us that if $(q,\tau)\in\mathcal{P}_{\alpha}(M,S)\times\mathbb{R}^{+}$
and $(q_{s},\tau_{s})_{s\in(-\varepsilon,\varepsilon)}\subset\mathcal{P}_{\alpha}(M,S)\times\mathbb{R}^{+}$
is a variation of $(q,\tau)$ of class $C^{2}$ with $\frac{\partial}{\partial s}\bigl|_{s=0}q_{s}(t)=:\eta(t)$
and $\frac{\partial}{\partial s}\bigl|_{s=0}\tau_{s}=:h$ then setting
$\gamma(t):=q(t/\tau)$ and $\upsilon(t):=\eta(t/\tau)$ one has 
\begin{align*}
\frac{\partial}{\partial s}\Bigl|_{s=0}\mathcal{S}_{L}(q_{s},\tau_{s}) & =\int_{0}^{\tau}\left\langle \nabla L^{\mathsf{h}}(\gamma,\dot{\gamma})-\nabla_{t}(\nabla L^{\mathsf{v}}(\gamma,\dot{\gamma}))-Y(\gamma)\dot{\gamma},\upsilon\right\rangle dt,\\
 & -\frac{h}{\tau}\int_{0}^{\tau}E(\gamma,\dot{\gamma})dt+\left[\left\langle \nabla L^{\mathsf{v}}(\gamma(1),\dot{\gamma}(1),\upsilon(1)\right\rangle -\left\langle \nabla L^{\mathsf{v}}(\gamma(0),\dot{\gamma}(0),\upsilon(0)\right\rangle \right],
\end{align*}
where $E:TM\rightarrow\mathbb{R}$ is defined by 
\[
E(q,v):=H(\nabla L^{\mathsf{v}}(q,v)).
\]
 Thus $\frac{\partial}{\partial s}\bigl|_{s=0}\mathcal{S}_{L}(q_{s},\tau_{s})=0$
for all such variations $(q_{s},\tau_{s})$ if and only if $\gamma(t):=q(t/\tau)$
satisfies the \emph{Euler-Lagrange equations}\textbf{ }
\begin{equation}
\nabla L^{\mathsf{h}}(\gamma,\dot{\gamma})-\nabla_{t}(\nabla L^{\mathsf{v}}(\gamma,\dot{\gamma}))-Y(\gamma)\dot{\gamma}=0\label{eq:EL-1}
\end{equation}
together with the\textbf{\emph{ }}\emph{energy constraint}\textbf{
}
\begin{equation}
\int_{0}^{1}E(\gamma,\dot{\gamma})dt=0,\label{eq:energy constaint}
\end{equation}
and the boundary conditions
\begin{equation}
\left\langle \nabla L^{\mathsf{v}}(\gamma(1),\dot{\gamma}(1),u\right\rangle =\left\langle \nabla L^{\mathsf{v}}(\gamma(0),\dot{\gamma}(0),v\right\rangle \ \ \ \mbox{for all }u\in T_{\gamma(0)}S\mbox{ and }v\in T_{\gamma(1)}S.\label{eq:the boundary condition}
\end{equation}
In particular, note that any critical point $(q,\tau)$ of $\mathcal{S}_{L}$
actually belongs to $P_{\alpha}(M,S)\times\mathbb{R}^{+}$ (i.e. $q$
is smooth). Since $L$ is electromagnetic at infinity, $\mathcal{S}_{L}$
is of class $C^{1,1}$ on $\mathcal{P}_{\alpha}(M,S)\times\mathbb{R}^{+}$
(see \cite{AbbondandoloSchwarz2009,AbbondandoloSchwarz2009a}). It
will also be useful to consider the \emph{fixed period action functional}.
Given $\tau\in\mathbb{R}^{+}$ let us denote by $\mathcal{S}_{L}^{\tau}$
the functional defined by 
\[
\mathcal{S}_{L}^{\tau}(q):=\mathcal{S}_{L}(q,\tau).
\]
Note that 
\[
d\mathcal{S}_{L}^{\tau}(q)(\eta)=d\mathcal{S}_{L}(q,\tau)(\eta,0).
\]
Thus if $(q,\tau)\in\mbox{Crit}^{\alpha}(\mathcal{S}_{L})$ then $q\in\mbox{Crit}^{\alpha}(\mathcal{S}_{L}^{\tau})$.
By definition, the \emph{Morse index }$m(q,\tau)$ of a critical point
$(q,\tau)\in\mbox{Crit}^{\alpha}(\mathcal{S}_{L})$ is the maximal
dimension of a subspace $W^{1,2}(q^{*}TM)\times\mathbb{R}$ on which
the Hessian $\nabla^{2}\mathcal{S}_{L}(q,\tau)$ of $\mathcal{S}_{L}$
at $(q,\tau)$ is negative definite. Similarly we denote by $m_{\tau}(q)$
the Morse index\textbf{ }of a critical point $q\in\mbox{Crit}^{\alpha}(\mathcal{S}_{L}^{\tau})$,
that is, the maximal dimension of a subspace of $W^{1,2}(q^{*}TM)$
on which the Hessian $\nabla^{2}\mathcal{S}_{L}^{\tau}$ of $\mathcal{S}_{L}^{\tau}$
is negative definite. It is well known that for Tonelli Lagrangians
the Morse index $m_{\tau}(q)$ is always finite \cite[Section 1]{Duistermaat1976}.
We define the \emph{nullity}\textbf{ }$n(q,\tau)$ of a critical point
of $\mathcal{S}_{L}$ to be 
\[
n(q,\tau):=\dim\,\ker(\nabla_{q}^{2}\mathcal{S}_{L}^{\tau}),
\]
and we say that a critical point $(q,\tau)\in\mathcal{P}_{\alpha}(M,S)\times\mathbb{R}^{+}$
is \emph{non-degenerate}\textbf{ }if $n(q,\tau)=0$. Since we have
assumed that our fixed Hamiltonian $H$ is non-degenerate, it actually
follows that every critical point of $\mathcal{S}_{L}$ is non-degenerate.
This is because there is a simple relationship between the critical
points of $\mathcal{S}_{L}$ and those of $\mathcal{A}_{H}$, which
we will discuss further in Lemma \ref{lem:(Properties-S and A} below.
Moreover Lemma \ref{lem:(Properties-S and A}, together with the discussion
\vpageref{correction term}, implies that for each critical point
$(q,\tau)$ of $\mathcal{S}_{L}$, there exists a unique family $(q_{s},\tau(s))\in\mbox{Crit}(\mathcal{S}_{L+e(s)})$
for $s\in(-\varepsilon,\varepsilon)$, where $(q_{0},\tau(0))=(q,\tau)$
and $e(0)=0$. Moreover we have $\tau'(0)\ne0$ and $e'(0)\ne0$.
We can therefore define the \emph{correction term}:
\[
\chi(q,\tau):=\mbox{sign}\left(-\frac{e'(0)}{\tau'(0)}\right)\in\{-1,1\}.
\]
A proof of the following result can be found in \cite[Section 10.2]{Merry2011}
(see also \cite[Theorem 1.2]{MerryPaternain2010}).
\begin{thm}
\label{thm:relating the Morse indices}Let $(q,\tau)\in\mbox{\emph{Crit}}^{\alpha}(\mathcal{S}_{L})$.
Then 
\[
m(q,\tau)=m_{\tau}(q)+\frac{1}{2}-\frac{1}{2}\chi(q,\tau).
\]

\end{thm}

\subsection{The Palais-Smale condition}

Work of Abbondandolo and Schwarz \cite{AbbondandoloSchwarz2009a,AbbondandoloSchwarz2009}
implies that we can find a\textbf{ }smooth\textbf{ }bounded vector
field $\mathbf{G}$ on $\mathcal{P}_{\alpha}(M,S)\times\mathbb{R}^{+}$
with the following two properties:
\begin{itemize}
\item There exists a continuous function $\delta\in C^{\infty}(\mathcal{P}_{\alpha}(M,S)\times\mathbb{R}^{+},\mathbb{R})$
such that for all $(q,\tau)\in\mathcal{P}_{\alpha}(M,S)\times\mathbb{R}^{+}$
one has 
\[
d\mathcal{S}_{L}(q,\tau)(\mathbf{G}(q,\tau))\geq\delta(\mathcal{S}_{L}(q,\tau))\left\Vert d\mathcal{S}_{L}(q,\tau)\right\Vert _{W^{1,2}}.
\]

\item For each $(q,\tau)\in\mathcal{P}_{\alpha}(M,S)\times\mathbb{R}^{+}$
one has
\[
\mathcal{S}_{L}(q,\tau)\in\mbox{Crit}^{\alpha}(\mathcal{S}_{L})\ \ \ \Leftrightarrow\ \ \ \mathbf{G}(q,\tau)=0,
\]
and moreover if $(q,\tau)\in\mbox{Crit}(\mathcal{S}_{L})$ then 
\[
\nabla^{2}\mathcal{S}_{L}(q,\tau)=\nabla\mathbf{G}(q,\tau),
\]
where $\nabla\mathbf{G}(q,\tau)$ denotes the \emph{Jacobian}\textbf{
}of $\mathbf{G}$, defined by $\nabla\mathbf{G}(q,\tau)(\xi,h):=[\mathbf{G},V](q,\tau)$,
with $V$ any vector field on $\mathcal{P}_{\alpha}(M,S)\times\mathbb{R}^{+}$
such that $V(q,\tau)=(\xi,h)$.
\end{itemize}
Moreover in the case $\alpha=0$, we may additionally insist that
the following two properties hold:
\begin{itemize}
\item There exists $k_{1}>0$ such that 
\begin{equation}
\left\langle \left\langle \mathbf{G}(q,\tau),\frac{\partial}{\partial\tau}\right\rangle \right\rangle _{W^{1,2}}<0\ \ \ \mbox{if}\ \ \ \mathcal{S}_{L}(q,\tau)\geq k_{1}\tau\label{eq:def of k1}
\end{equation}
(see \cite[Section 11]{AbbondandoloSchwarz2009} or \cite[Lemma 10.3]{Merry2011}).
\item If $\Upsilon^{s}$ denotes the local flow of $-\mathbf{G}$ then the
submanifold $S\times\mathbb{R}^{+}\subset\mathcal{P}_{0}(M,S)\times\mathbb{R}^{+}$
of constant curves is invariant under $\Upsilon^{s}$, that is, whenever
defined one has
\[
\Upsilon^{s}(S\times\mathbb{R}^{+})\subset S\times\mathbb{R}^{+}.
\]

\end{itemize}
We shall refer to a smooth bounded vector field $\mathbf{G}$ that
satisfies these four properties as a \emph{refined pseudo-gradient}
for $\mathcal{S}_{L}$. For a given point $(q,\tau)\in\mathcal{P}_{\alpha}(M,S)\times\mathbb{R}^{+}$,
we denote by $(\omega_{-}(q,\tau),\omega_{+}(q,\tau))\subset\mathbb{R}$
the maximal interval of existence of the flow line $s\mapsto\Upsilon^{s}(q,\tau)$.
The next result is the key to defining the Morse (co)complex of $\mathcal{S}_{L}$
(compare \cite[Proposition 11.1, Proposition 11.2]{AbbondandoloSchwarz2009}).
A full proof in our setting is given in \cite{Merry2011}.
\begin{thm}
\textup{\emph{\label{thm:(Properties-of-)-1}}}Let $\mathbf{G}$ denote
a refined pseudo-gradient for $\mathcal{S}_{L}$, and let $\Upsilon^{s}$
denote the local flow of $-\mathbf{G}$. Then:
\begin{enumerate}
\item If $\alpha\ne0$ then the pair $(\mathcal{S}_{L},\mathbf{G})$ satisfies
the Palais-Smale condition at the level $a$ for all $a\in\mathbb{R}$.
If $\alpha=0$ then the pair $(\mathcal{S}_{L},\mathbf{G})$ satisfies
the Palais-Smale condition at the level $a$ for all $a\in\mathbb{R}\backslash\{0\}$. 
\item $\mathcal{S}_{L}$ is bounded below on $\mathcal{P}_{\alpha}(M,S)\times\mathbb{R}^{+}$.
\item If $\alpha\ne0$ then $\omega_{+}(q,\tau)=\infty$ for all $(q,\tau)\in\mathcal{P}_{\alpha}(M,S)\times\mathbb{R}^{+}$.
Moreover if $(q,\tau)\in\mathcal{P}_{\alpha}(M,S)\times\mathbb{R}^{+}$
and $(q_{s},\tau_{s}):=\Upsilon^{s}(q,\tau)$ then $\tau_{s}$ is
bounded strictly away from zero as $s\rightarrow\infty$.
\item If $\alpha=0$ and $(q,\tau)\in\mathcal{P}_{0}(M,S)\times\mathbb{R}^{+}$
satisfies $\omega_{+}(q,\tau)<\infty$, then if we define $(q_{s},\tau_{s}):=\Upsilon^{s}(q,\tau)$
one has $\mathcal{S}_{L}(q_{s},\tau_{s})\rightarrow0$, $\tau_{s}\rightarrow0$
and $q_{s}$ converges to a constant path as $s\rightarrow\omega_{+}(q,\tau)$.
If instead $\omega_{+}(q,\tau)=\infty$ then $\tau_{s}$ is strictly
bounded away from zero as $s\rightarrow\infty$.
\item If $\alpha\ne0$ then $\omega_{-}(q,\tau)=-\infty$ for all $(q,\tau)\in\mathcal{P}_{\alpha}(M,S)\times\mathbb{R}^{+}$. 
\item Given $a>0$ define
\[
\mathcal{O}(a):=\left\{ (q,\tau)\in\mathcal{P}_{0}(M,S)\times(0,k_{1}a)\mid\mathcal{S}_{L}(q,\tau)<a\right\} ,
\]
where $k_{1}>0$ was defined in \eqref{eq:def of k1}. Then $\mathcal{O}(a)\cap\mbox{\emph{Crit}}(\mathcal{S}_{L})=\emptyset$
for all $a>0$, and for any $a>0$, if $(q,\tau)\in\mathcal{O}(a)$
then $\Upsilon^{s}(q,\tau)\in\mathcal{O}(a)$ for all $s\in(\omega_{-}(q,\tau),0]$.
Finally if $(q,\tau)\in\mathcal{P}_{0}(M,S)\times\mathbb{R}^{+}$
is such that $\omega_{-}(q,\tau)>-\infty$ and $\mathcal{S}_{L}(q,\tau)\geq a$
then there exists $s<0$ such that $\Upsilon^{s}(q,\tau)\in\mathcal{O}(a)$.
\end{enumerate}
\end{thm}
\label{enu:exteded unstable}Given a refined pseudo-gradient $\mathbf{G}$
for $\mathcal{S}_{L}$ and a critical point $(q,\tau)$ of $\mathcal{S}_{L}$,
we denote by $\mathbf{W}^{u}((q,\tau),-\mathbf{G})$ the \emph{extended
unstable manifold}, which by definition is the union of the unstable
manifold $W^{u}((q,\tau),-\mathbf{G})$ together with the set of points
one finds by following the forward orbit under $\Upsilon^{s}$ of
elements $(q',\tau')\in W^{u}((q,\tau),-\mathbf{G})$ which do not
converge in $\mathcal{P}(M,S)\times\mathbb{R}^{+}$ as $s\rightarrow\omega_{+}(q',\tau')$.
By Theorem \ref{thm:(Properties-of-)-1}.4 these are all of the form
$(y,0)$ for some point $y\in S$ (interpreted as usual as a constant
path). These are the so-called \emph{critical points at infinity}\textbf{
}in the sense of Bahri \cite{Bahri1989}. In a similar vein it is
convenient to define the following subset of $\mathcal{P}(M,S)\times[0,\infty)$:
\[
\underline{\mbox{Crit}}(\mathcal{S}_{L}):=\mbox{Crit}(\mathcal{S}_{L})\cup\left(S\times\{0\}\right).
\]
Our non-degeneracy assumption implies that the functional $\mathcal{S}_{L}$
is actually Morse, but it is not ``Morse at infinity'', in the sense
that the critical points at infinity (i.e. the set $S\times\{0\}$)
form a Morse-Bott component of $\underline{\mbox{Crit}}(\mathcal{S}_{L})$.
Thus we will need to work with flow lines with cascades to define
the Morse (co)homology of $\mathcal{S}_{L}$, as we shall now explain.

\subsection{\label{sub:The-Morse-complex}The Morse complex}

To define the Morse complex we will need three pieces of auxiliary
data. As with the construction of Rabinowitz Floer homology in Section
\ref{sub:The-definition-of}, the construction is much simpler if
$\alpha\ne0$ (we can ignore the Morse function $\ell$ and there
is no need for cascades). Nevertheless, for the sake of a uniform
presentation, we do not treat this case separately. 
\begin{itemize}
\item Firstly, let $\mathbf{G}$ denote a refined pseudo-gradient for $\mathcal{S}_{L}$,
and as before write $\Upsilon^{s}$ for the local flow of $-\mathbf{G}$. 
\item Choose a Morse function $\ell:S\rightarrow\mathbb{R}$. In order to
fit in with the approach taken in Section \ref{sub:The-definition-of},
it will be convenient to formally regard $\ell$ also as a function
$\ell:\underline{\mbox{Crit}}^{\alpha}(\mathcal{S}_{L})\rightarrow\mathbb{R}$
by setting $\ell(q,\tau):=0$ for $(q,\tau)\in\mbox{Crit}^{\alpha}(\mathcal{S}_{L})$
and setting 
\[
\ell(y,0):=\ell(y)\ \ \ \mbox{for }(y,0)\in S\times\{0\}.
\]
We denote by $C^{\alpha}(\ell)\subset\underline{\mbox{Crit}}^{\alpha}(\mathcal{S}_{L})$
the set of critical points of $\ell$ (so that for $\alpha\ne0$,
$C^{\alpha}(\ell)=\mbox{Crit}^{\alpha}(\mathcal{S}_{L})$ and for
$\alpha=0$, $C^{0}(\ell)=\mbox{Crit}^{0}(\mathcal{S}_{L})\cup\mbox{Crit}(\ell)$),
and given $-\infty\leq a<b\leq\infty$ we define 
\[
C^{\alpha}(\ell)_{a}^{b}:=C^{\alpha}(\ell)\cap\underline{\mbox{Crit}}^{\alpha}(\mathcal{S}_{L})_{a}^{b},
\]
where by definition $\mathcal{S}_{L}(q,0):=0$. It follows from Theorem
\ref{thm:(Properties-of-)-1} that if $b-a<\infty$ then $C^{\alpha}(\ell)_{a}^{b}$
is always finite. 
\item Thirdly, let $\nu$ denote a Riemannian metric on $S$ such that the
flow $\psi^{t}$ of $-\nabla\ell=-\nabla_{\nu}\ell$ is Morse-Smale.
As with $\ell$, we can formally regard $\psi^{t}$ as a flow on $\underline{\mbox{Crit}}^{\alpha}(\mathcal{S}_{L})$
by defining $\psi^{t}(q,\tau):=(q,\tau)$ for all $(q,\tau)\in\mbox{Crit}^{\alpha}(\mathcal{S}_{L})$
and $t\in\mathbb{R}$. 
\end{itemize}
Given $(q,\tau)\in C^{\alpha}(\ell)$, we denote by $i_{\ell}(q,\tau)$
the Morse index of $(q,\tau)$, seen as a critical point of $\ell$,
so that $i_{\ell}(q,\tau):=\dim\, W^{u}((q,\tau);-\nabla\ell)$. Thus
$i_{\ell}(q,\tau)=0$ unless $\tau=0$ and $q(t)\equiv y$ for some
$y\in\mbox{Crit}(\ell)$. Finally, let us define 
\[
m_{\ell}(q,\tau):=m(q,\tau)+i_{\ell}(q,\tau),
\]
for $(q,\tau)\in C^{\alpha}(\ell)$, where by definition $m(y,0):=0$.
The Morse complex is defined with the aid of the spaces $\mathcal{W}((q_{-},\tau_{-}),(q_{+},\tau_{+});\ell)$
of gradient flow lines with cascades running between two critical
points $(q_{-},\tau_{-})$ and $(q_{+},\tau_{+})$ of $C^{\alpha}(\ell)$.
These spaces are defined entirely analogously to the spaces $\mathcal{M}((x_{-},\tau_{-}),(x_{+},\tau_{+}))$
from Definition \ref{thm:the moduli spaces}, only we work with $\mathcal{S}_{L}$
and $\ell$ rather than $\mathcal{A}_{H}$ and $f$. We use the letter
$\mathcal{W}$ instead of $\mathcal{M}$ to help distinguish between
the two, and we include the ``$\ell$'' in the notation because
later on we will use these spaces with different choices of Morse
function $\ell$. The next theorem, together with Theorem \ref{thm:morse homology}
below, follow from Theorem \ref{thm:(Properties-of-)-1} exactly as
in \cite[Section 11]{AbbondandoloSchwarz2009}. The main ingredients
are Abbondandolo and Majer's infinite dimensional Morse theory \cite{AbbondandoloMajer2006}
and Frauenfelder's cascades approach to Morse-Bott homology (as described
in \cite{Frauenfelder2004} and Section \ref{sub:The-definition-of}). 
\begin{thm}
\label{thm:generic morse}For a generic choice of $\mathbf{G}$ and
$\nu$ the sets $\mathcal{W}((q_{-},\tau_{-}),(q_{+},\tau_{+});\ell)$
are all smooth manifolds of finite dimension 
\[
\dim\,\mathcal{W}((q_{-},\tau_{-}),(q_{+},\tau_{+});\ell)=m_{\ell}(q_{-},\tau_{-})-m_{\ell}(q_{+},\tau_{+})-1.
\]
Moreover if $m_{\ell}(q_{-},\tau_{-})-m_{\ell}(q_{+},\tau_{+})=1$
then $\mathcal{W}((q_{-},\tau_{-}),(q_{+},\tau_{+});\ell)$ is compact,
and hence a finite set. 
\end{thm}
Denote by 
\[
\mbox{CM}_{*}^{\alpha}(L,\ell)_{a}^{b}:=C_{*}^{\alpha}(\ell)_{a}^{b}\otimes\mathbb{Z}_{2},
\]
where the grading $*$ is given by the function $m_{\ell}$. Given
$(q_{\pm},\tau_{\pm})\in C^{\alpha}(\ell)_{a}^{b}$ with $m_{\ell}(q_{-},\tau_{-})=m_{\ell}(q_{+},\tau_{+})+1$,
we define the number $n_{\textrm{Morse}}((q_{-},\tau_{-}),(q_{+},\tau_{+});\ell)\in\mathbb{Z}_{2}$
to be the parity of the finite set $\mathcal{W}((q_{-},\tau_{-}),(q_{+},\tau_{+});\ell)$.
If $m_{\ell}(q_{-},\tau_{-})\ne m_{\ell}(q_{+},\tau_{+})+1$, set
$n_{\textrm{Morse}}((q_{-},\tau_{-}),(q_{+},\tau_{+});\ell)=0$. Now
we define the boundary operator 
\[
\partial_{a}^{b}=\partial_{a}^{b}(L,\mathbf{G},\ell,\nu):\mbox{CM}_{*}^{\alpha}(L,\ell)_{a}^{b}\rightarrow\mbox{CM}_{*-1}^{\alpha}(L,\ell)_{a}^{b}
\]
as the linear extension of 
\[
(q_{-},\tau_{-})\mapsto\sum_{(q_{+},\tau_{+})\in C^{\alpha}(\ell)_{a}^{b}}n_{\textrm{Morse}}((q_{-},\tau_{-}),(q_{+},\tau_{+});\ell)(q_{+},\tau_{+}).
\]
The next result is the \emph{Morse homology theorem}. Let us write
\[
\Lambda_{L}^{b}(\alpha):=\left\{ (q,\tau)\in\mathcal{P}_{\alpha}(M,S)\times\mathbb{R}^{+}\mid\mathcal{S}_{L}(q,\tau)<b\right\} .
\]

\begin{thm}
\label{thm:morse homology}For a generic choice of $\mathbf{G}$ and
$\nu$, it holds that $\partial_{a}^{b}\circ\partial_{a}^{b}=0$.
Thus $\{\mbox{\emph{CM}}_{*}^{\alpha}(L,\ell),\partial_{a}^{b}\}$
forms a chain complex. The isomorphism class of this complex is independent
of the choice of $\mathbf{G}$, $\ell$ and $\nu$. The associated
homology, known as the \emph{Morse homology of} $\mathcal{S}_{L}$,
is isomorphic to the singular (co)homology of the pair $(\Lambda_{L}^{b}(\alpha),\Lambda_{L}^{a}(\alpha))$.
\[
\mbox{\emph{HM}}_{*}^{\alpha}(L)_{a}^{b}\cong\mbox{\emph{H}}_{*}(\Lambda_{L}^{b}(\alpha),\Lambda_{L}^{a}(\alpha);\mathbb{Z}_{2}).
\]
In particular, if $b=\infty$ and $a<\inf\mathcal{S}_{L}$ then
\[
\mbox{\emph{HM}}_{*}^{\alpha}(L):=\mbox{\emph{HM}}_{*}^{\alpha}(L)_{a}^{\infty}\cong\mbox{\emph{H}}_{*}(\mathcal{P}_{\alpha}(M,S)\times\mathbb{R}^{+};\mathbb{Z}_{2})\cong\mbox{\emph{H}}_{*}(P_{\alpha}(M,S);\mathbb{Z}_{2}).
\]

\end{thm}
One can also play the same game with cohomology. For reasons that
will become clear in Section \ref{sub:The-chain-map AS}, it is convenient
to use the Morse function $-\ell$ when defining the Morse cohomology.
Given $-\infty\leq a<b\leq\infty$, let $C^{\alpha}(-\ell)_{a}^{b}$
denote the set of critical points $(q,\tau)$ of $-\ell$ with $a<\mathcal{S}_{L}(a,\tau)<b$.
We grade $C^{\alpha}(-\ell)$ with $m_{-\ell}$. Note that $C^{\alpha}(\ell)_{a}^{b}=C^{\alpha}(-\ell)_{a}^{b}$
as sets but in general not as graded sets. Now set 
\[
\mbox{CM}_{\alpha}^{*}(L,-\ell)_{a}^{b}:=\prod_{(q,\tau)\in C_{*}^{\alpha}(-\ell)_{a}^{b}}\mathbb{Z}_{2}(q,\tau),
\]
and define 
\[
\delta_{a}^{b}=\delta_{a}^{b}(L,\mathbf{G},-\ell,\nu):\mbox{CM}_{\alpha}^{*}(L,-\ell)_{a}^{b}\rightarrow\mbox{CM}_{\alpha}^{*+1}(L,-\ell)_{a}^{b}
\]
as the linear extension of 
\[
(q_{+},\tau_{+})\mapsto\sum_{(q_{-},\tau_{-})\in C^{\alpha}(-\ell)_{a}^{b}}n_{\textrm{Morse}}((q_{-},\tau_{-}),(q_{+},\tau_{+});-\ell)(q_{-},\tau_{-})
\]
(here $n_{\textrm{Morse}}((q_{-},\tau_{-}),(q_{+},\tau_{+});-\ell)$
denotes the parity of the corresponding finite set 
$$\mathcal{W}((q_{-},\tau_{-}),(q_{+},\tau_{+});-\ell)).$$
Then $\delta_{a}^{b}\circ\delta_{a}^{b}=0$, and hence $\{\mbox{CM}_{\alpha}^{*}(L,-\ell)_{a}^{b},\delta_{a}^{b}\}$
forms a cochain complex, whose cohomology computes the singular cohomology
of the pair $(\Lambda_{L}^{b}(\alpha),\Lambda_{L}^{a}(\alpha))$.

\subsection{\label{sub:Relating-the-two}Relating the two functionals $\mathcal{S}_{L}$
and $\mathcal{A}_{H}$}

We will now study the relationship between the two functionals $\mathcal{S}_{L}$
and $\mathcal{A}_{H}$. The next lemma follows readily from the definitions. 
\begin{lem}
\label{lem:(Properties-S and A}The following relationships between
$\mbox{\emph{Crit}}^{\pm\alpha}(\mathcal{S}_{L})$ and $\mbox{\emph{Crit}}^{\alpha}(\mathcal{A}_{H})$
hold:
\begin{enumerate}
\item Given $(q,\tau)\in\mbox{\emph{Crit}}^{\alpha}(\mathcal{S}_{L})$,
define 
\[
\psi_{+}(q,\tau):=(x,\tau)\ \ \ \mbox{where }x(t):=\left(q(t),\nabla L^{\mathsf{v}}(q(t),\dot{q}(t))\right).
\]
\[
\psi_{-}(q,\tau):=(\mathbb{I}(x),-\tau),
\]
where $\mathbb{I}(x)(t):=x(1-t)$. Then if $\alpha\ne0$, one has
\[
\mbox{\emph{Crit}}^{\alpha}(\mathcal{A}_{H})=\psi_{+}\left(\mbox{\emph{Crit}}^{\alpha}(\mathcal{S}_{L})\right)\cup\psi_{-}\left(\mbox{\emph{Crit}}^{-\alpha}(S_{L})\right)
\]
and moreover one has 
\[
\mathcal{A}_{H}(\psi_{\pm}(q,\tau))=\pm\mathcal{S}_{L}(q,\tau).
\]

\item Given any $(x,\tau)\in P(T^{*}M,N^{*}S)\times\mathbb{R}$ with $\tau>0$,
if $q:=\pi\circ x$ then
\[
\mathcal{A}_{H}(x,\tau)\leq\mathcal{S}_{L}(q,\tau),
\]
\[
\mathcal{A}_{H}(\mathbb{I}(x),-\tau)\geq-\mathcal{S}_{L}(q,\tau)
\]
with equality if and only if $x=(q,\nabla L^{\mathsf{v}}(q,\dot{q}))$.
\item Let $(q,\tau)\in P_{\alpha}(M,S)\times\mathbb{R}^{+}$ denote a critical
point of $\mathcal{S}_{L}$. Then for all $(\xi,h)$ it holds that
\[
d^{2}\mathcal{A}_{H}(\psi_{+}(q,\tau))((\xi,h),(\xi,h))\leq d^{2}\mathcal{S}_{L}(q,\tau)((\xi^{\mathsf{h}},h),(\xi^{\mathsf{h}},h)),
\]
Let $(q,\tau)\in P_{-\alpha}(M,S)\times\mathbb{R}^{+}$ denote a critical
point of $\mathcal{S}_{L}$. Then for all $(\xi,h)$ it holds that
\[
d^{2}\mathcal{A}_{H}(\psi_{-}(q,\tau))((\xi,h),(\xi,h))\geq-d^{2}\mathcal{S}_{L}(q,\tau)((\mathbb{I}(\xi)^{\mathsf{h}},-h),(\mathbb{I}(\xi)^{\mathsf{h}},-h)).
\]

\item Given a critical point $(q,\tau)$, a pair $(\xi,h)$ lies in the
kernel of the Hessian of $\mathcal{A}_{H}$ at $\psi_{+}(q,\tau)$
{[}resp. $\psi_{-}(q,\tau)${]} if and only if the pair $(\xi^{\mathsf{h}},h)$
{[}resp. $(\mathbb{I}(\xi)^{\mathsf{h}},-h)${]} lies in the kernel
of the Hessian of $\mathcal{S}_{L}$ at $(q,\tau)$.
\item If $(q,\tau)\in\mbox{\emph{Crit}}(\mathcal{S}_{L})$ then 
\[
\chi(q,\tau)=\chi(\psi_{+}(q,\tau))=-\chi(\psi_{-}(q,\tau)).
\]

\end{enumerate}
\end{lem}
Next, we discuss the relations between the indices of the critical
points. We first recall the following statement, which is an extension
of the \emph{Morse index theorem} of Duistermaat \cite{Duistermaat1976}
to the twisted symplectic form $\omega$.
\begin{thm}
\textbf{\emph{\label{thm:(The-Morse-index}}}Let $(q,\tau)\in\mbox{\emph{Crit}}^{\alpha}(\mathcal{S}_{L})$.
Then 
\[
m_{\tau}(q)=\mu_{\mathsf{Ma}}(\psi_{-}(q,\tau))+d-\frac{n}{2}.
\]
\end{thm}
\begin{proof}
We deduce this from the equivalent statement for the standard symplectic
form $d\lambda_{\textrm{can}}$ (specifically, from \cite[Corollary 4.2]{AbbondandoloPortaluriSchwarz2008})
by arguing as follows: take a tubular neighborhood $W$ of $q([0,1])$
in $M$. Since $H^{2}(W)=0$, $\sigma|_{W}=d\varphi$ for some $\varphi\in\Omega^{1}(W)$.
The flow $\phi_{H}^{t}|_{W}$ is conjugate to the flow $\psi_{H_{\varphi}}^{t}:T^{*}W\rightarrow T^{*}W$,
where $H_{\varphi}(q,p)=H(q,p-\varphi_{q})$ and $\psi_{H_{\varphi}}^{t}$
denotes the flow of the symplectic gradient of $H_{\varphi}$ with
respect to the standard symplectic form $d\lambda_{\textrm{can}}$.
Since both the Maslov index and the Morse index are local invariants,
the theorem now follows directly from \cite[Corollary 4.2]{AbbondandoloPortaluriSchwarz2008}. 
\end{proof}
Recall that in order to define the Lagrangian Rabinowitz Floer chain
complex we need to pick a Morse function $f$ on $\mbox{Crit}^{\alpha}(\mathcal{A}_{H})$.
It is convenient to choose $f$ and $\ell$ so that they satisfy the
following properties.
\begin{enumerate}
\item For all $(q,\tau)\in\mbox{Crit}^{\alpha}(\mathcal{S}_{L})$ one has
\[
\ell(q,\tau)=f(\psi_{-}(q,\tau))=f(\psi_{+}(q,\tau)).
\]

\item The function $\ell$ has a unique minimum $y_{\min}$ and a unique
maximum $y_{\max}$ for two points $y_{\min},y_{\max}\in S$ and is
\emph{self-indexing}, that is, $\ell(y)=i_{\ell}(y)$ for all $y\in\mbox{Crit}(\ell)$.
Note that\textbf{ }if $d=\dim\, S=0$ (i.e. $S=\{y\}$ and $N^{*}S=T_{y}^{*}M$)
then we obviously have $y_{\min}=y_{\max}=y$, but that in all other
cases clearly $y_{\min}\ne y_{\max}$.
\item For all $x\in\Sigma\cap N^{*}S$, we have $\ell(\pi(x))\leq f(x,0)\leq\ell(\pi(x))+1/2$.
\item Every critical point of $f|_{\Sigma\cap N^{*}S\times\{0\}}$ lies
above a critical point of $\ell$ and moreover for each critical point
$y$ of $\ell$ there are exactly two critical points of $f|_{\Sigma\cap N^{*}S\times\{0\}}$
in $\Sigma\cap T_{y}^{*}M\times\{0\}$. Denoting these two critical
points by $\psi_{\pm}(y,0)$, it holds that 
\[
\ell(y)=f(\psi_{-}(y,0))=f(\psi_{+}(y,0))-1/2,
\]
\[
i_{\ell}(y)=i_{f}(\psi_{-}(y,0))=i_{f}(\psi_{+}(y,0))-n+d+1.
\]

\end{enumerate}
That such functions exist is explained in detail in \cite[Appendix B]{AbbondandoloSchwarz2009}.
With this choice of functions $f$ and $\ell$ the following relationships
hold \textemdash{} the proof is an immediate application of Theorem
\ref{thm:relating the Morse indices}, Lemma \ref{lem:(Properties-S and A}.5,
and Theorem \ref{thm:(The-Morse-index}.
\begin{cor}
Let $(q,\tau)\in C^{\alpha}(\ell)$. Then
\[
m_{\ell}(q,\tau)=\begin{cases}
\mu_{f}(\psi_{+}(q,\tau)), & \tau>0,\\
-\mu_{f}(\psi_{-}(q,\tau))+2d-n+1, & \tau>0,\\
\mu_{f}(\psi_{+}(q,\tau)), & \tau=0,\\
\mu_{f}(\psi_{-}(q,\tau))-d+n-1, & \tau=0
\end{cases}
\]
and 
\[
m_{-\ell}(q,\tau)=\begin{cases}
\mu_{f}(\psi_{+}(q,\tau)), & \tau>0,\\
-\mu_{f}(\psi_{-}(q,\tau))+2d-n+1, & \tau>0,\\
-\mu_{f}(\psi_{+}(q,\tau))+d, & \tau=0,\\
-\mu_{f}(\psi_{-}(q,\tau))+2d-n+1, & \tau=0.
\end{cases}
\]

\end{cor}

\subsection{Computing the Lagrangian Rabinowitz Floer homology}

In this section we state the key technical result of this paper, which
is the extension of \cite[Theorem 2]{AbbondandoloSchwarz2009} to
our setting. 
\begin{thm}
\label{thm:theorem A precise}\textbf{\emph{(Computation of the Lagrangian
Rabinowitz Floer homology)}}

Let $f:\mbox{\emph{Crit}}^{\alpha}(\mathcal{A}_{H})\rightarrow\mathbb{R}$
and $\ell:S\rightarrow\mathbb{R}$ be Morse functions as specified
above. Let $m$ and $\nu$ denote generically chosen Riemannian metrics
on $\mbox{\emph{Crit}}(\mathcal{A}_{H})$ and $S$ respectively, such
that the flows $\varphi^{t}$ and $\psi^{t}$ of $-\nabla f=-\nabla_{m}f$
and $-\nabla\ell=-\nabla_{\nu}\ell$ are Morse-Smale. Let $\mathbf{G}$
denote a generically chosen refined pseudo-gradient for $\mathcal{S}_{L}$,
and let $\mathbf{J}=(J_{t})\subset\mathcal{J}(X,\omega)$ denote
a generic family of almost complex structures, such that $\sup_{t}\left\Vert J_{t}-J_{g}\right\Vert _{L^{\infty}}$
is sufficiently small.

Fix $-\infty<a<b<\infty$. Then there exists:
\begin{enumerate}
\item An injective chain map 
\[
(\Phi_{\mathsf{SA}})_{a}^{b}:\mbox{\emph{CM}}_{*}^{\alpha}(L,\ell)_{a}^{b}\rightarrow\mbox{\emph{CRF}}_{*}^{\alpha}(H,f)_{a}^{b}
\]
 which admits a left inverse $(\widehat{\Phi}_{\mathsf{SA}})_{a}^{b}:\mbox{\emph{CRF}}_{*}^{\alpha}(H,f)_{a}^{b}\rightarrow\mbox{\emph{CM}}_{*}^{\alpha}(L,\ell)_{a}^{b}$.
\item A surjective chain map 
\[
(\Phi_{\mathsf{AS}})_{a}^{b}:\mbox{\emph{CRF}}_{*}^{\alpha}(H,f)_{a}^{b}\rightarrow\mbox{\emph{CM}}_{-\alpha}^{-*+2d-n+1}(L,-\ell)_{-b}^{-a}
\]
 which admits a right inverse $(\widehat{\Phi}_{\mathsf{AS}})_{a}^{b}:\mbox{\emph{CM}}_{-\alpha}^{-*+2d-n+1}(L,-\ell)_{-b}^{-a}\rightarrow\mbox{\emph{CRF}}_{*}^{\alpha}(H,f)_{b}^{a}$.
\end{enumerate}
Moreover:
\begin{enumerate}
\item If $d<n/2$ then $\Phi_{\mathsf{SA}}$ and $\Phi_{\mathsf{AS}}$ define
chain complex isomorphisms 
\[
(\Phi_{\mathsf{SA}})_{a}^{b}:\mbox{\emph{CM}}_{*}^{\alpha}(L,\ell)_{a}^{b}\cong\mbox{\emph{CRF}}_{*}^{\alpha}(H,f)_{a}^{b},
\]
\textup{
\[
(\Phi_{\mathsf{AS}})_{a}^{b}:\mbox{\emph{CRF}}_{*}^{\alpha}(H,f)_{a}^{b}\rightarrow\mbox{\emph{CM}}_{-\alpha}^{-*+2d-n+1}(L,-\ell)_{-b}^{-a},
\]
and }thus in the limit $a\rightarrow-\infty$, $b\rightarrow\infty$,
if we identify $\mbox{\emph{HM}}_{*}^{\alpha}(L,\ell)\cong\mbox{\emph{H}}_{*}(P_{\alpha}(M,S);\mathbb{Z}_{2})$
and $\mbox{\emph{HM}}_{-\alpha}^{*}(L,-\ell)\cong\mbox{\emph{H}}^{*}(P_{-\alpha}(M,S);\mathbb{Z}_{2})$
we deduce that 
\[
\mbox{\emph{RFH}}_{*}^{\alpha}(H)\cong\begin{cases}
\mbox{\emph{H}}_{*}(P_{\alpha}(M,S);\mathbb{Z}_{2}), & *\geq0,\\
0, & 2d-n+1<*<0,\\
\mbox{\emph{H}}^{-*+2d-n+1}(P_{-\alpha}(M,S);\mathbb{Z}_{2}), & *\leq2d-n+1.
\end{cases}
\]

\item If $\alpha\ne0$ and $d\geq n/2$, or if $\alpha=0$ and $d=n/2$
and $n\geq4$ then the composition $(\Phi_{\mathsf{AS}})_{a}^{b}\circ(\Phi_{\mathsf{SA}})_{a}^{b}:\mbox{\emph{CM}}_{*}^{\alpha}(L,\ell)_{a}^{b}\rightarrow\mbox{\emph{CM}}_{-\alpha}^{-*+2d-n+1}(L,-\ell)_{-b}^{-a}$
is chain homotopic to zero, that is, there exists a homomorphism
\[
\Theta_{a}^{b}:\mbox{\emph{CM}}_{*}^{\alpha}(L,\ell)_{a}^{b}\rightarrow\mbox{\emph{CM}}_{-\alpha}^{-*+2d-n}(L,-\ell)_{-b}^{-a}
\]
 such that 
\[
(\Phi_{\mathsf{AS}})_{a}^{b}\circ(\Phi_{\mathsf{SA}})_{a}^{b}=\Theta_{a}^{b}\circ\partial_{a}^{b}+\delta_{-b}^{-a}\circ\Theta_{a}^{b}.
\]
Setting
\[
\Psi_{a}^{b}:=(\Phi_{\mathsf{SA}})_{a}^{b}-(\widehat{\Phi}_{\mathsf{AS}})_{a}^{b}\circ\Theta_{a}^{b}\circ\partial_{a}^{b}-\partial_{a}^{b}\circ(\widehat{\Phi}_{\mathsf{AS}})_{a}^{b}\circ\Theta_{a}^{b},
\]
the map $\Psi_{a}^{b}$ is chain homotopic to $(\Phi_{\mathsf{SA}})_{a}^{b}$,
and satisfies $(\Phi_{\mathsf{AS}})_{a}^{b}\circ\Psi_{a}^{b}=0$.
Thus we obtain a short exact sequence of chain complexes
\[
0\rightarrow\mbox{\emph{CM}}_{*}^{\alpha}(L,\ell)_{a}^{b}\overset{\Psi_{a}^{b}}{\rightarrow}\mbox{\emph{CRF}}_{*}^{\alpha}(H,f)_{a}^{b}\overset{(\Phi_{\mathsf{AS}})_{a}^{b}}{\rightarrow}\mbox{\emph{CM}}_{-\alpha}^{-*+2d-n+1}(L,-\ell)_{-b}^{-a}\rightarrow0.
\]
Thus in the limit $a\rightarrow-\infty$, $b\rightarrow\infty$, if
we identify $\mbox{\emph{HM}}_{*}^{\alpha}(L,\ell)\cong\mbox{\emph{H}}_{*}(P_{\alpha}(M,S);\mathbb{Z}_{2})$
and $\mbox{\emph{HM}}_{-\alpha}^{*}(L,-\ell)\cong\mbox{\emph{H}}^{*}(P_{-\alpha}(M,S);\mathbb{Z}_{2})$,
then we obtain the long exact sequence
\[
\xymatrix{\dots\ar[r] & \mbox{\emph{H}}_{j}(P_{\alpha}(M,S);\mathbb{Z}_{2})\ar[r]^{\Theta_{*}} & \mbox{\emph{RFH}}_{j}^{\alpha}(H)\ar[d]^{(\Psi_{\textrm{\emph{AS}}})_{*}}\\
 &  & \mbox{\emph{H}}^{-j+2d-n+1}(P_{-\alpha}(M,S);\mathbb{Z}_{2})\ar[r]^{\Delta} & \mbox{\emph{H}}_{j-1}(P_{\alpha}(M,S);\mathbb{Z}_{2})\ar[r] & \dots
}
\]
The connecting homomorphism $\Delta$ is identically zero unless $\alpha=0$
and $j=1$, in which case it is given by (recall by assumption when
$\alpha=0$ one has $d=n/2$):
\[
\xymatrix{\mbox{\emph{H}}^{0}(P_{0}(M,S);\mathbb{Z}_{2})\ar[r]^{\Delta}\ar[d] & \mbox{\emph{H}}_{0}(P_{0}(M,S);\mathbb{Z}_{2})\\
\mbox{\emph{H}}^{0}(S;\mathbb{Z}_{2})\ar[r]_{\Xi} & \mbox{\emph{H}}_{0}(S;\mathbb{Z}_{2})\ar[u]
}
\]
where 
\[
\Xi(c):=\mbox{\emph{PD}}(c\smallsmile e(N^{*}S)),
\]
($e(N^{*}S)$ is the Euler class of $N^{*}S\rightarrow S$) and the
vertical maps are the isomorphisms induced by the inclusion $S\hookrightarrow P_{0}(M,S)$. 
\end{enumerate}
\end{thm}
The proof of this theorem is based on the corresponding result by
Abbondandolo-Schwarz in \cite{AbbondandoloSchwarz2009}, and its extension
to twisted cotangent bundles in our earlier paper \cite{Merry2011a}.
We will therefore omit many of the technical details in the exposition
here, referring the reader to the beautiful and lucid exposition in
\cite{AbbondandoloSchwarz2009}, or the more detailed treatment given
in \cite{Merry2011}. Theorem \ref{thm:big 2} is an immediate consequence
of this result.

\subsection{The extended unstable manifolds with cascades $\mathbf{W}^{u}((q,\tau);-\mathbf{G},-\nabla\ell)$}

We use the notation from Theorem \ref{thm:theorem A precise}. Recall
the definition of the extended unstable manifold $\mathbf{W}^{u}((q,\tau);-\mathbf{G})$
introduced \vpageref{enu:exteded unstable}. We now introduce the
\emph{extended unstable manifold with cascades}, which we denote by
$\mathbf{W}^{u}((q,\tau);-\mathbf{G},-\nabla\ell)$. Fix $(q,\tau)\in C^{\alpha}(\ell)$.
Given $k\in\mathbb{N}$, let $\widetilde{\mathcal{W}}_{k}^{u}((q,\tau);-\nabla l)$
denote the set of tuples $(\mathbf{q},\boldsymbol{\tau})=(q_{j},\tau_{j})_{1\leq j\leq k}$
such that $(q_{j},\tau_{j})\in\mathcal{P}_{\alpha}(M,S)\times\mathbb{R}^{+}$
for $j=1,\dots,k-1$ and $(q_{k},\tau_{k})$ either belongs to $\mathcal{P}_{\alpha}(M,S)\times\mathbb{R}^{+}$
or is of the form $(y,0)$ for some point $y\in S$. Moreover we insist
that 
\[
(q_{1},\tau_{1})\in\mathbf{W}^{u}(W^{u}((q,\tau));-\nabla\ell);-\mathbf{G}),
\]
\[
\lim_{s\rightarrow-\infty}\Upsilon^{s}(q_{j+1},\tau_{j+1})\in\psi^{[0,\infty)}\left(\lim_{s\rightarrow\infty}\Upsilon^{s}(q_{j},\tau_{j})\right)\ \ \ \mbox{for }j=1,\dots,k-1.
\]
Of course, if $\alpha\ne0$ then $(q_{k},\tau_{k})$ is always in
$\mathcal{P}_{\alpha}(M,S)\times\mathbb{R}^{+}$. Let $\mathcal{W}_{k}^{u}((q,\tau);-\nabla l)$
denote the quotient of $\widetilde{\mathcal{W}}_{k}^{u}((q,\tau);-\nabla l)$
under the free $\mathbb{R}^{k-1}$ action given by
\[
(q_{j},\tau_{j})_{1\leq j\leq k}\mapsto\left((\Upsilon^{s_{j}}(q_{j},\tau_{j}))_{1\leq j\leq k-1},(q_{k},\tau_{k})\right),\ \ \ (s_{1},\dots,s_{k-1})\in\mathbb{R}^{k-1}.
\]
Then set 
\[
\mathbf{W}^{u}((q,\tau);-\mathbf{G},-\nabla\ell):=\bigcup_{k\in\mathbb{N}}\mathcal{W}_{k}^{u}((q,\tau);-\nabla l).
\]
There is a well defined evaluation map 
\[
\mbox{ev}:\mathbf{W}^{u}((q,\tau);-\mathbf{G},-\nabla\ell)\rightarrow\left(\mathcal{P}_{\alpha}(M,S)\times\mathbb{R}^{+}\right)\cup(S\times\{0\}),
\]
given by 
\[
\mbox{ev}(\mathbf{q},\boldsymbol{\tau}):=(q_{k},\tau_{k})\ \ \ \mbox{for}\ \ \ (\mathbf{q},\boldsymbol{\tau})\in\mathcal{W}_{k}^{u}((q,\tau);-\nabla l).
\]
For a generic choice of $\mathbf{G}$ and $\nu$ the spaces $\mathbf{W}^{u}((q,\tau);-\mathbf{G},-\nabla\ell)$
admit the structure of smooth manifolds of finite dimension 
\[
\dim\,\mathbf{W}^{u}((q,\tau);-\mathbf{G},-\nabla\ell)=m_{\ell}(q,\tau).
\]
This can be proved using \cite[Corollary A.16]{Frauenfelder2004},
and details can be found in \cite[Section 12.1]{Merry2011}.

\subsection{The chain map $\Phi_{\mathsf{SA}}$}

In this section we define a chain map 
\[
(\Phi_{\mathsf{SA}})_{a}^{b}:\mbox{CM}_{*}^{\alpha}(L,\ell)_{a}^{b}\rightarrow\mbox{CRF}_{*}^{\alpha}(H,h)_{a}^{b}.
\]
In order to define the chain map $\Phi_{\textrm{SA}}$, one needs
to construct a suitable moduli space. The first step is to define
the\textbf{ }space of\emph{ positive half flow lines with cascades}\textbf{
}for $\mathcal{A}_{H}$, denoted by $\mathbf{M}^{s}(x,\eta)$. Here
we write critical points of $\mathcal{A}_{H}$ as pairs $(x,\eta)$
rather than $(x,\tau)$, so to minimize confusion below. Fix $(x,\eta)\in C^{\alpha}(f)$.
Given $k\in\mathbb{N}$ let $\widetilde{\mathcal{M}}_{k}^{s}(x,\eta)$
denote the denote the set of $m$-tuples of maps $\mathbf{u}=(u_{1},\dots,u_{k})$
such that 
\[
u_{1}:[0,\infty)\rightarrow P_{\alpha}(T^{*}M,N^{*}S)\times\mathbb{R};
\]
\[
u_{2},\dots,u_{k}:\mathbb{R}\rightarrow P_{\alpha}(T^{*}M,N^{*}S)\times\mathbb{R},
\]
 are all gradient flow lines of $(H,\mathbf{J})$ (that are possibly
stationary solutions) and such that
\[
\lim_{s\rightarrow\infty}u_{k}(s)\in W^{s}((x,\eta);-\nabla f);
\]
\[
\lim_{s\rightarrow-\infty}u_{j+1}(s)\in\varphi^{[0,\infty)}(\lim_{s\rightarrow\infty}u_{j}(s))\ \ \ \mbox{for }j=1,\dots,k-1.
\]
Let $\mathcal{M}_{k}^{s}(x,\eta)$ denote the quotient of $\widetilde{\mathcal{M}}_{k}^{s}(x,\eta)$
under the free $\mathbb{R}^{k-1}$ action given by translation along
the flow lines $u_{2},\dots,u_{k}$. Then put 
\[
\mathbf{M}^{s}(x,\eta):=\bigcup_{k\in\mathbb{N}}\mathcal{M}_{k}^{s}(x,\eta).
\]
The space $\mathbf{M}^{s}(x,\eta)$ is not finite dimensional. However,
by restricting where $u_{1}$ can ``begin'', we can cut it down
to something finite dimensional. This is precisely what the moduli
space $\mathcal{M}_{\mathsf{SA}}((q,\tau),(x,\eta))$ does. Fix $(q,\tau)\in C^{\alpha}(\ell)$
and define $\mathcal{M}_{\mathsf{SA}}((q,\tau),(x,\eta))$ to be the
following subset of $\mathbf{W}^{u}((q,\tau);-\mathbf{G},-\nabla\ell)\times\mathbf{M}^{s}(x,\eta)$.
Namely, an element 
\[
(\mathbf{q},\boldsymbol{\tau},\mathbf{u})\in\mathbf{W}^{u}((q,\tau);-\mathbf{G},-\nabla\ell)\times\mathbf{M}^{s}(x,\eta)
\]
with $(\mathbf{q},\boldsymbol{\tau})\in\mathcal{W}_{k}^{u}((q,\tau);-\nabla\ell)$
belongs to n $\mathcal{M}_{\mathsf{SA}}((q,\tau),(x,\eta))$ if and
only if, writing $u_{1}=(x_{1},\eta_{1})$ one has
\[
(q_{k},\tau_{k})=(\pi\circ x_{1}(0,\cdot),\eta_{1}(0)).
\]
This defines a Lagrangian boundary condition. This implies that we
have a Fredholm problem, and since generically $\mathbf{W}^{u}((q,\tau);-\mathbf{G},-\nabla\ell)$
is a finite dimensional manifold, it follows that $\mathcal{M}_{\mathsf{SA}}((q,\tau),(x,\eta))$
can be seen as the zero set of a Fredholm operator. In fact, more
is true.
\begin{thm}
\label{thm:moduli space SA}For a generic choice of $\mathbf{G}$,
$\mathbf{J}$, $m$ and $\nu$, the spaces $\mathcal{M}_{\mathsf{SA}}((q,\tau),(x,\eta))$
are precompact smooth manifolds of finite dimension 
\[
\dim\,\mathcal{M}_{\mathsf{SA}}((q,\tau),(x,\eta))=m_{\ell}(q,\tau)-\mu_{f}(x,\eta).
\]
\end{thm}
\begin{proof}
The only complication with obtaining transversality is the presence
of stationary solutions, which can appear if $(x,\eta)=\psi_{+}(q,\tau)$
or $(q,\tau)=(y,0)$ for some $y\in S$ and $(x,\eta)=\psi_{\pm}(y,0)$.
In the former case the first inequality of the third statement of
Lemma \ref{lem:(Properties-S and A} forces the linearized operator
defining the moduli space $\mathcal{M}_{\mathsf{SA}}((q,\tau),\psi_{+}(q,\tau))$
to be an isomorphism (see \cite[Lemma 6.2]{AbbondandoloSchwarz2009}
or \cite[Proposition 3.7]{AbbondandoloSchwarz2006}), and in the second
two cases the four assumptions made earlier on the Morse functions
$f$ and $\ell$ guarantee that the linearized operator defining the
moduli spaces $\mathcal{M}_{\mathsf{SA}}((y,0),\psi_{\pm}(y,0))$
is surjective (see \cite[Lemma 6.3]{AbbondandoloSchwarz2009}). The
index computation can be proved by combining \cite[Theorem 5.24]{AbbondandoloSchwarz2010}
(a special case of this is given in \cite[Proposition 7.3]{AbbondandoloPortaluriSchwarz2008})
and the arguments from \cite[Section 4]{CieliebakFrauenfelder2009}.
Full details can be found in \cite[Theorem 12.3]{Merry2011}. Finally
we address the precompactness statement. The key point here is the
following chain of inequalities, which follow from Lemma \ref{lem:(Properties-S and A}.2:
\begin{equation}
\mathcal{S}_{L}(q,\tau)\geq\mathcal{S}_{L}(q_{k},\tau_{k})=\mathcal{S}_{L}(\pi\circ x_{1}(0,\cdot),\eta_{1}(0))\geq\mathcal{A}_{H}(u_{1}(0,\cdot))\geq\mathcal{A}_{H}(x,\eta).\label{eq:bound}
\end{equation}
More details can be found in \cite[Section 6]{AbbondandoloSchwarz2009}
and \cite[Theorem 12.3]{Merry2011}. 
\end{proof}
Putting this together, we deduce that when $m_{\ell}(q,\tau)=\mu_{f}(x,\eta)$,
the space $\mathcal{M}_{\mathsf{SA}}((q,\tau),(x,\eta))$ is a finite
set, and hence we can define $n_{\mathsf{SA}}((q,\tau),(x,\eta))\in\mathbb{Z}_{2}$
to be its parity. If $m_{\ell}(q,\tau)\ne\mu_{f}(x,\eta)$, set $n_{\mathsf{SA}}((q,\tau),(x,\eta))=0$.
Then define $(\Phi_{\mathsf{SA}})_{a}^{b}:\mbox{CM}_{*}^{\alpha}(L,\ell)_{a}^{b}\rightarrow\mbox{CRF}_{*}^{\alpha}(H,f)_{a}^{b}$
as the linear extension of 
\[
(q,\tau)\mapsto\sum_{(x,\eta)\in C^{\alpha}(f)_{a}^{b}}n_{\mathsf{SA}}((q,\tau),(x,\eta))(x,\eta)
\]
(we are implicitly using \eqref{eq:bound} here to ensure that the
choice of action window makes sense). A standard gluing argument shows
that $(\Phi_{\mathsf{SA}})_{a}^{b}$ is a chain map.

\subsection{\label{sub:The-chain-map AS}The chain map $\Phi_{\mathsf{AS}}$}

In this section we define a chain map 
\[
(\Phi_{\mathsf{AS}})_{a}^{b}:\mbox{CRF}_{*}^{\alpha}(H,f)_{a}^{b}\rightarrow\mbox{CM}_{-\alpha}^{-*+2d-n+1}(L,-\ell)_{-b}^{-a}.
\]
It is defined in much the same way. One begins by defining a space
$\mathbf{M}^{u}(x,\eta)$ of\textbf{ }\emph{negative half flow lines
with cascades}. Given $k\in\mathbb{N}$ let $\widetilde{\mathcal{M}}_{k}^{u}(x,\eta)$
denote the denote the set of tuples of maps $\mathbf{u}=(u_{1},\dots,u_{k})$
such that 
\[
u_{1},\dots,u_{k-1}:\mathbb{R}\rightarrow P_{\alpha}(T^{*}M,N^{*}S)\times\mathbb{R};
\]
\[
u_{k}:(-\infty,0]\rightarrow P_{\alpha}(T^{*}M,N^{*}S)\times\mathbb{R},
\]
which are gradient flow lines of $(H,\mathbf{J})$ (that are possibly
stationary solutions) and such that
\[
\lim_{s\rightarrow-\infty}u_{k}(s)\in W^{u}((x,\eta);-\nabla f),
\]
 and such that 
\[
\lim_{s\rightarrow-\infty}u_{j+1}(s)\in\varphi^{[0,\infty)}(\lim_{s\rightarrow\infty}u_{j}(s))\ \ \ \mbox{for }j=1,\dots,k-1.
\]
 Let $\mathcal{M}_{k}^{u}(x,\eta)$ denote the quotient of $\widetilde{\mathcal{M}}_{k}^{u}(x,\eta)$
under the free $\mathbb{R}^{k-1}$ action and put 
\[
\mathbf{M}^{u}(x,\eta):=\bigcup_{k\in\mathbb{N}}\mathcal{M}_{k}^{u}(x,\eta).
\]
Now if $(x,\eta)\in C^{\alpha}(f)$ and $(q,\tau)\in C^{-\alpha}(-\ell)$,
we define $\mathcal{M}_{\mathsf{AS}}((x,\eta),(q,\tau))$ to be the
following subset of $\mathbf{W}^{u}((q,\tau);-\mathbf{G},\nabla\ell)\times\mathbf{M}^{u}(x,\eta)$
(note here we are using the Morse function $-\ell$). Namely, an element
\[
(\boldsymbol{q},\boldsymbol{\tau},\mathbf{u})\in\mathbf{W}^{u}((q,\tau);-\mathbf{G},\nabla\ell)\times\mathbf{M}^{u}(x,\eta)
\]
 with $(\mathbf{q},\boldsymbol{\tau})\in\mathcal{W}_{k}^{u}((q,\tau);\nabla l)$
and $\mathbf{u}\in\mathcal{M}_{p}^{u}(x,\eta)$ belongs to $\mathcal{M}_{\mathsf{AS}}((x,\eta),(q,\tau))$
if and only if, writing $u_{p}=(x_{p},\eta_{p})$ one has
\[
(q_{k},\tau_{k})=(\pi\circ x_{p}(0,-\cdot),-\eta_{p}(0)).
\]
The following theorem is proved in the same way as Theorem \ref{thm:moduli space SA}.
Details can be found in \cite[Section 9]{AbbondandoloSchwarz2009}
and \cite[Section 12.3]{Merry2011}.
\begin{thm}
\label{thm:moduli space AS}For a generic choice of $\mathbf{G}$,
$\mathbf{J}$, $m$ and $\nu$, the spaces $\mathcal{M}_{\mathsf{AS}}((x,\eta),(q,\tau))$
are precompact smooth manifolds of finite dimension 
\[
\dim\,\mathcal{M}_{\mathsf{AS}}((x,\eta),(q,\tau))=\mu_{f}(x,\tau)+m_{-\ell}(q,\tau)+n-2d-1.
\]

\end{thm}
We remark only that this time the key inequality responsible for compactness
is the following: if $(\boldsymbol{q},\boldsymbol{\tau},\mathbf{u})\in\mathcal{M}_{\mathsf{AS}}((x,\eta),(q,\tau))$
with $(\mathbf{q},\boldsymbol{\tau})\in\mathcal{W}_{k}^{u}((q,\tau);\nabla l)$
and $\mathbf{u}\in\mathcal{M}_{p}^{u}(x,\eta)$ then
\begin{equation}
\mathcal{A}_{H}(x,\eta)\geq\mathcal{A}_{H}(u_{p}(0,\cdot))\geq-\mathcal{S}_{L}(\pi\circ x_{p}(0,-\cdot),-\tau_{p}(0))\geq-\mathcal{S}_{L}(q_{k},\tau_{k})\geq-\mathcal{S}_{L}(q,\tau).\label{eq:bound 2}
\end{equation}
Putting this together, we deduce that when $\mu_{f}(x,\eta)=-m_{-\ell}(q,\tau)+2d-n+1$,
the moduli space $\mathcal{M}_{\mathsf{AS}}((x,\eta),(q,\tau))$ is
a finite set, and hence we can define $n_{\mathsf{AS}}((x,\eta),(q,\tau))\in\mathbb{Z}_{2}$
to be its parity. If $\mu_{f}(x,\eta)\ne-m_{-\ell}(q,\tau)+1-n+2d$
set $n_{\mathsf{AS}}((x,\eta),(q,\tau))=0$. Then defines $(\Phi_{\mathsf{AS}})_{a}^{b}:\mbox{CRF}_{*}^{\alpha}(H,f)_{a}^{b}\rightarrow\mbox{CM}_{-\alpha}^{-*+2d-n+1}(L,-\ell)_{-b}^{-a}$
as the linear extension of 
\[
(x,\eta)\mapsto\sum_{(q,\tau)\in C^{-\alpha}(\ell)_{-b}^{-a}}n_{\mathsf{AS}}((x,\eta),(q,\tau))(q,\tau)
\]
(this time we are implicitly using \eqref{eq:bound 2} here to ensure
that the choice of action window makes sense). A standard gluing argument
shows that $(\Phi_{\mathsf{AS}})_{a}^{b}$ is a chain map.

\subsection{The chain homotopy $\Theta$}

\emph{We assume throughout this section that $d\geq n/2$, and if
$\alpha=0$ then we additionally assume $n\geq4$ and that $d=n/2$.}

We will construct a chain homotopy 
\[
\Theta_{a}^{b}:\mbox{CM}_{*}^{\alpha}(L,\ell)_{a}^{b}\rightarrow\mbox{CM}_{-\alpha}^{-*+2d-n}(L,-\ell)_{-b}^{-a}
\]
which will have the property that
\[
(\Phi_{\mathsf{AS}})_{a}^{b}\circ(\Phi_{\mathsf{SA}})_{a}^{b}=\Theta_{a}^{b}\circ\partial_{a}^{b}+\delta_{-b}^{-a}\circ\Theta_{b}^{a}.
\]
This will involve counting a slightly different sort of object. Let
$\mathcal{F}_{0}$ denote the set of pairs $(u,R)$ where $R\in\mathbb{R}^{+}$
and $u=(x,\eta):[-R,R]\times[0,1]\rightarrow T^{*}M\times\mathbb{R}$
satisfies the Rabinowitz Floer equation. Given $k\geq1$, let $\widetilde{\mathcal{F}}_{k}$
denote the set of tuples $(v,\mathbf{u},w)$ where $\mathbf{u}=(u_{1},\dots,u_{k-1})$
are gradient flow lines of $\mathcal{A}_{H}$ such that 
\[
\lim_{s\rightarrow-\infty}u_{j+1}(s)\in\varphi^{[0,\infty)}\left(\lim_{s\rightarrow\infty}u_{j}(s)\right)\ \ \ \mbox{for }j=1,\dots,k-2.
\]
Next, 
\[
v:[0,\infty)\rightarrow P_{\alpha}(T^{*}M,N^{*}S)\times\mathbb{R},
\]
\[
w:(-\infty,0]\rightarrow P_{\alpha}(T^{*}M,N^{*}S)\times\mathbb{R}
\]
both satisfy the Rabinowitz Floer equation, with
\[
\lim_{s\rightarrow-\infty}u_{1}(s)\in\varphi^{[0,\infty)}\left(\lim_{s\rightarrow\infty}v(s)\right),\ \ \ \lim_{s\rightarrow-\infty}w(s)\in\varphi^{[0,\infty)}\left(\lim_{s\rightarrow\infty}u_{k-1}(s)\right).
\]
Let $\mathcal{F}_{k}$ denote the quotient of $\widetilde{\mathcal{F}}_{k}$
by dividing through by the $\mathbb{R}^{k-1}$ action on the curves
$u_{1},\dots,u_{k-1}$. Put
\[
\mathcal{F}=\bigcup_{k\in\mathbb{N}\cup\{0\}}\mathcal{F}_{k}.
\]
Given $(q_{-},\tau_{-})\in C^{\alpha}(\ell)$ and $(q_{+},\tau_{+})\in C^{-\alpha}(-\ell)$,
we define $\mathcal{M}_{\Theta}((q_{-},\tau_{-}),(q_{+},\tau_{+}))$
to be the subset of points in 
\[
\mathbf{W}^{u}((q_{-},\tau_{-});-\mathbf{G},-\nabla\ell)\times\mathcal{F}\times\mathbf{W}^{u}((q_{+},\tau_{+});-\mathbf{G},\nabla\ell)
\]
 satisfying:
\begin{enumerate}
\item If $((\mathbf{q},\boldsymbol{\tau}),(u,R),(\mathbf{q}',\boldsymbol{\tau}'))\in\mathcal{M}_{\Theta}((q_{-},\tau_{-}),(q_{+},\tau_{+}))$
with $(u,R)\in\mathcal{F}_{0}$, $(\mathbf{q},\boldsymbol{\tau})\in\mathcal{W}_{i}^{u}((q_{-},\tau_{-});-\nabla l)$,
and $(\mathbf{q}',\boldsymbol{\tau}')\in\mathcal{W}_{p}^{u}((q_{+},\tau_{+});\nabla l)$,
then writing $u=(x,\eta)$, we require that 
\[
(\pi\circ x(-R,\cdot),\eta(-R))=(q_{i},\tau_{i}),\ \ \ (\pi\circ x(R,-\cdot),-\eta(R))=(q_{p}',\tau_{p}').
\]

\item If $((\mathbf{q},\boldsymbol{\tau}),(v,\mathbf{u},w),(\mathbf{q}',\boldsymbol{\tau}'))\in\mathcal{M}_{\Theta}((q_{-},\tau_{-}),(q_{+},\tau_{+}))$
with $(v,\mathbf{u},w)\in\mathcal{F}_{k}$ for some $k\geq1$, $(\mathbf{q},\boldsymbol{\tau})\in\mathcal{W}_{i}^{u}((q_{-},\tau_{-});-\nabla l)$,
and $(\mathbf{q}',\boldsymbol{\tau}')\in\mathcal{W}_{p}^{u}((q_{+},\tau_{+});\nabla l)$,
then writing $v=(x,\eta)$ and $w=(x',\eta')$, we require that
\[
(\pi\circ x(0,\cdot),\eta(0))=(q_{i},\tau_{i}),\ \ \ (\pi\circ x'(0,-\cdot),-\eta'(0))=(q_{p}',\tau_{p}').
\]

\end{enumerate}
Let us note if $((\mathbf{q},\boldsymbol{\tau}),(u,R),(\mathbf{q}',\boldsymbol{\tau}'))\in\mathcal{M}_{\Theta}((q_{-},\tau_{-}),(q_{+},\tau_{+}))$
with $(u,R)\in\mathcal{F}_{0}$ , $(\mathbf{q},\boldsymbol{\tau})\in\mathcal{W}_{i}^{u}((q_{-},\tau_{-});-\nabla l)$,
and $(\mathbf{q}',\boldsymbol{\tau}')\in\mathcal{W}_{p}^{u}((q_{+},\tau_{+});\nabla l)$,
then we have 
\begin{equation}
\mathcal{S}_{L}(q_{-},\tau_{-}))\geq\mathcal{S}_{L}(q_{i},\tau_{i})\geq\mathcal{A}_{H}(u(-R,\cdot))\geq\mathcal{A}_{H}(u(R,\cdot))\geq-\mathcal{S}_{L}(q_{p}',\tau_{p}')\geq-\mathcal{S}_{L}(q_{+},\tau_{+}).\label{eq:chain homotopy ineq 2}
\end{equation}
Similarly if $((\mathbf{q},\boldsymbol{\tau}),(v,\mathbf{u},w),(\mathbf{q}',\boldsymbol{\tau}'))\in\mathcal{M}_{\Theta}((q_{-},\tau_{-}),(q_{+},\tau_{+}))$
with $\mathbf{u}\in\mathcal{F}_{k}$ for some $k\geq1$, $(\mathbf{q},\boldsymbol{\tau})\in\mathcal{W}_{i}^{u}((q_{-},\tau_{-});-\nabla l)$,
and $(\mathbf{q}',\boldsymbol{\tau}')\in\mathcal{W}_{p}^{u}((q_{+},\tau_{+});\nabla l)$,
then we have 
\begin{equation}
\mathcal{S}_{L}(q_{-},\tau_{-})\geq\mathcal{S}_{L}(q_{i},\tau_{i})\geq\mathcal{A}_{H}(v(0,\cdot))\geq\mathcal{A}_{H}(w(0,\cdot))\geq-\mathcal{S}_{L}(q_{p}',\tau_{p}')\geq-\mathcal{S}_{L}(q_{+},\tau_{+}).\label{eq:chain homotopy ineq 1}
\end{equation}
This time we have the following result. For more details we refer
the reader to \cite[Section 8]{AbbondandoloSchwarz2009} or \cite[Section 12.4]{Merry2011}.
The latter reference explains exactly where the assumption that $d=n/2$
with $n\geq4$ if $\alpha=0$ is used.
\begin{thm}
Denote by $C_{\Theta}^{\alpha}(\ell,-\ell)\subset C^{\alpha}(\ell)\times C^{-\alpha}(-\ell)$
the set of pairs $(q_{\pm},\tau_{\pm})$ of critical points that satisfy
\[
m_{\ell}(q_{-},\tau_{-})+m_{-\ell}(q_{+},\tau_{+})\in\{2d-n,2d-n+1\}.
\]
Then for a generic choice of $\mathbf{G}$, $\mathbf{J}$, $m$ and
$\nu$, the spaces $\mathcal{M}_{\Theta}((q_{-},\tau_{-}),(q_{+},\tau_{+}))$
for $(q_{\pm},\tau_{\pm})\in C_{\Theta}^{\alpha}(\ell,-\ell)$ are
precompact smooth manifolds of finite dimension
\[
\dim\,\mathcal{M}_{\Theta}((q_{-},\tau_{-}),(q_{+},\tau_{+}))=m_{\ell}(q_{-},\tau_{-})+m_{-\ell}(q_{+},\tau_{+})+n-2d.
\]

\end{thm}
Now we move onto the key proposition which implies Theorem \ref{thm:theorem A precise}.
The first statement of Theorem \ref{pro:MP precompactness} below
shows that under our assumptions, if we are given $(q_{-},\tau_{-})\in C^{\alpha}(\ell)$
and $(q_{+},\tau_{+})\in C^{-\alpha}(-\ell)$ with $m_{\ell}(q_{-},\tau_{-})+m_{-\ell}(q_{+},\tau_{+})=2d-n$
then we can define $n_{\Theta}((q_{-},\tau_{-}),(q_{+},\tau_{+}))$
as the parity of the finite set $\mathcal{M}_{\Theta}((q_{-},\tau_{-}),(q_{+},\tau_{+}))$.
This defines the chain map $\Theta_{a}^{b}$ (this time we are implicitly
using \eqref{eq:chain homotopy ineq 2} and \eqref{eq:chain homotopy ineq 1}
in order to ensure that the choice of action window makes sense).
The fact that $\Theta_{a}^{b}$ is a chain homotopy between $(\Phi_{\mathsf{SA}})_{a}^{b}$
and $(\Phi_{\mathsf{AS}})_{a}^{b}$ involves studying the compactification
of $\mathcal{M}_{\Theta}((q_{-},\tau_{-}),(q_{+},\tau_{+}))$ by adding
in the broken trajectories, and is the content of the second statement
of the proposition below, which is taken from \cite[Proposition 8.1]{AbbondandoloSchwarz2009}.
Details of the proof in the Lagrangian case we study here can be found
in \cite[Section 12.10]{Merry2011}.
\begin{prop}
\label{pro:MP precompactness}Fix critical points $(q_{-},\tau_{-})\in C_{i}^{\alpha}(\ell)_{a}^{b}$
and $(q_{j},\tau_{j})\in C_{j}^{-\alpha}(-\ell)_{-b}^{-a}$. Recall
we always assume $d\geq n/2$ in this section, and if $\alpha=0$
then we require $d=n/2$ and $n\geq4$.
\begin{enumerate}
\item If $i+j=2d-n$ then the moduli space $\mathcal{M}_{\Theta}((q_{-},\tau_{-}),(q_{+},\tau_{+}))$
is compact.
\item If $i+j=2d-n+1$ then the moduli space $\mathcal{M}_{\Theta}((q_{-},\tau_{-}),(q_{+},\tau_{+}))$
is precompact, and we can identify the boundary $\partial\overline{\mathcal{M}}_{\Theta}((q_{-},\tau_{-}),(q_{+},\tau_{+}))$
of the compactification $\overline{\mathcal{M}}_{\Theta}((q_{-},\tau_{-}),(q_{+},\tau_{+}))$
as follows:
\begin{eqnarray*}
\partial\overline{\mathcal{M}}_{\Theta}((q_{-},\tau_{-}),(q_{+},\tau_{+})) & = & \left\{ \bigcup_{(x,\eta)\in C_{i}^{\alpha}(f)_{a}^{b}}\mathcal{M}_{\mathsf{SA}}((q_{-},\tau_{-}),(x,\eta))\times\mathcal{M}_{\mathsf{AS}}((x,\eta),(q_{+},\tau_{+}))\right\} \\
 &  & \bigcup\left\{ \bigcup_{(q,\tau)\in C_{i-1}^{\alpha}(\ell)_{b}^{a}}\mathcal{W}((q_{-},\tau_{-}),(q,\tau);\ell)\times\mathcal{M}_{\Theta}((q,\tau),(q_{+},\tau_{+}))\right\} \\
 &  & \bigcup\left\{ \bigcup_{(q',\tau')\in C_{j-1}^{\alpha}(-\ell)_{-b}^{-a}}\mathcal{M}_{\Theta}((q_{-},\tau_{-}),(q',\tau'))\times\mathcal{W}((q_{+},\tau_{+}),(q',\tau');-\ell)\right\} .
\end{eqnarray*}

\end{enumerate}
\end{prop}
Theorem \ref{thm:theorem A precise} follows from this proposition;
see \cite[Section 9]{AbbondandoloSchwarz2009} or \cite[Section 12]{Merry2011}
for the details.

\bibliographystyle{amsalpha}
\bibliography{willmacbibtex}

\providecommand{\bysame}{\leavevmode\hbox to3em{\hrulefill}\thinspace}
\providecommand{\MR}{\relax\ifhmode\unskip\space\fi MR }
\providecommand{\MRhref}[2]{%
  \href{http://www.ams.org/mathscinet-getitem?mr=#1}{#2}
}
\providecommand{\href}[2]{#2}
\begin{thebibliography}{MMP12}

\bibitem[Abb99]{Abbas1999}
C.~Abbas, \emph{A note of {V}. {I}. {A}rnold's chord conjecture}, Int. {M}ath.
  {R}es. {N}ot. \textbf{4} (1999), 217--222.

\bibitem[AF09]{AlbersFrauenfelder2009}
P.~Albers and U.~Frauenfelder, \emph{Floer {H}omology for negative line bundles
  and {R}eeb chords in prequantization spaces}, J. Modern Dynamics \textbf{3}
  (2009), no.~3, 407--456.

\bibitem[AF10]{AlbersFrauenfelder2010c}
\bysame, \emph{Leaf-wise intersections and {R}abinowitz {F}loer homology}, J.
  Topol. Anal. \textbf{2} (2010), no.~1, 77--98.

\bibitem[AF12]{AlbersFrauenfelder2008}
\bysame, \emph{Infinitely many leaf-wise intersection points on cotangent
  bundles}, Global {D}ifferential {G}eometry, Proc. in {M}athematics,
  Springer-Verlag, 2012, pp.~437--461.

\bibitem[AG90]{ArnoldGivental1990}
V.~I. Arnold and A.~B. Givental, \emph{Symplectic {G}eometry}, Dynamical
  Systems, Encyclopedia of Mathematical Sciences, vol.~IV, Springer-Verlag,
  1990.

\bibitem[AM06]{AbbondandoloMajer2006}
A.~Abbondandolo and P.~Majer, \emph{Lectures on the {M}orse complex for
  infinite dimensional manifolds}, Morse {T}heoretic {M}ethods in {N}onlinear
  {A}nalysis and {S}ymplectic {T}opology (P~Biran, O.~Cornea, and F.~Lalonde,
  eds.), Nato {S}cience {S}eries {II}: {M}athematics, {P}hysics and
  {C}hemistry, vol. 217, Springer-Verlag, 2006, pp.~1--74.

\bibitem[APS08]{AbbondandoloPortaluriSchwarz2008}
A.~Abbondandolo, A.~Portaluri, and M.~Schwarz, \emph{The homology of path
  spaces and {F}loer homology with conormal boundary conditions}, J. Fixed
  Point Theory Appl. \textbf{4} (2008), no.~2, 263--293.

\bibitem[Arn86]{Arnold1986}
V.~I. Arnold, \emph{First steps in symplectic topology}, Russ. Math. Surv.
  \textbf{41} (1986), no.~1-21.

\bibitem[AS06]{AbbondandoloSchwarz2006}
A.~Abbondandolo and M.~Schwarz, \emph{On the {F}loer homology of cotangent
  bundles}, Comm. Pure Appl. Math. \textbf{59} (2006), 254--316.

\bibitem[AS09a]{AbbondandoloSchwarz2009}
\bysame, \emph{Estimates and computations in {R}abinowitz-{F}loer homology}, J.
  Topol. Anal. \textbf{1} (2009), no.~4, 307--405.

\bibitem[AS09b]{AbbondandoloSchwarz2009a}
\bysame, \emph{A smooth pseudo-gradient for the {L}agrangian action
  functional}, Adv. Nonlinear Studies \textbf{9} (2009), 597--623.

\bibitem[AS10a]{AbbondandoloSchwarz2010}
\bysame, \emph{Floer homology of cotangent bundles and the loop product},
  Geometry and Topology \textbf{14} (2010), 1569--1722.

\bibitem[AS10b]{AbouzaidSeidel2010}
M.~Abouzaid and P.~Seidel, \emph{An open string analogue of {V}iterbo
  functoriality}, Geometry and Topology \textbf{14} (2010), 627--718.

\bibitem[Bah89]{Bahri1989}
A.~Bahri, \emph{Critical points at infinity in some variational problems},
  Pitman {R}esearch {N}otes in {M}athematics, vol. 182, Longman, 1989.

\bibitem[BEE09]{BourgeoisEkholmEliashberg2009}
F.~Bourgeois, T.~Ekholm, and Y.~Eliashberg, \emph{Effect of {L}egendrian
  {S}urgery}, arXiv:0911.0026 (2009).

\bibitem[BO09]{BourgeoisOancea2009}
F.~Bourgeois and A.~Oancea, \emph{Symplectic homology, autonomous
  {H}amiltonians, and {M}orse-{B}ott moduli spaces}, Duke Math. J. \textbf{146}
  (2009), no.~1, 71--174.

\bibitem[Bou]{Bounya}
C.~Bounya, \emph{An exact triangle for the wrapped {F}loer homology of the
  {L}agrangian {R}abinowitz functional}, In preparation.

\bibitem[Bou03]{Bourgeois2003}
F.~Bourgeois, \emph{A {M}orse--{B}ott approach to contact homology}, Symplectic
  and contact topology: interactions and perspectives (Y.~Eliashberg,
  B.~Khesin, and F.~Lalonde, eds.), Fields {I}nstitute {C}ommunications,
  vol.~35, Amer. Math. Soc., 2003.

\bibitem[BP02]{BurnsPaternain2002}
K.~Burns and G.~P. Paternain, \emph{Anosov magnetic flows, critical values and
  topological entropy}, Nonlinearity \textbf{15} (2002), 281--314.

\bibitem[CF09]{CieliebakFrauenfelder2009}
K.~Cieliebak and U.~Frauenfelder, \emph{A {F}loer homology for exact contact
  embeddings}, Pacific J. Math. \textbf{239} (2009), no.~2, 216--251.

\bibitem[CFO10]{CieliebakFrauenfelderOancea2010}
K.~Cieliebak, U.~Frauenfelder, and A.~Oancea, \emph{Rabinowitz {F}loer homology
  and symplectic homology}, Ann. Inst. Fourier \textbf{43} (2010), no.~6,
  957--1015.

\bibitem[CFP10]{CieliebakFrauenfelderPaternain2010}
K.~Cieliebak, U.~Frauenfelder, and G.~P. Paternain, \emph{Symplectic topology
  of {M}a\~n\'e's critical values}, Geometry and Topology \textbf{14} (2010),
  1765--1870.

\bibitem[CGK04]{CieliebakGinzburgKerman2004}
K.~Cieliebak, V.~Ginzburg, and E.~Kerman, \emph{Symplectic homology and
  periodic orbits near symplectic submanifolds}, Comment. {M}ath. {H}elv.
  \textbf{74} (2004), 554--581.

\bibitem[CI99]{ContrerasIturriaga1999}
G.~Contreras and R.~Iturriaga, \emph{Global {M}inimizers of {A}utonomous
  {L}agrangians}, Colloqio Brasileiro de Matematica, vol.~22, IMPA, Rio de
  Janeiro, 1999.

\bibitem[Cie02]{Cieliebak2002}
K.~Cieliebak, \emph{Handle attaching in symplectic homology and the chord
  conjecture}, J. Eur Math Soc \textbf{4} (2002), no.~2, 115--142.

\bibitem[CIPP00]{ContrerasIturriagaPaternainPaternain2000}
G.~Contreras, R.~Iturriaga, G.~P. Paternain, and M.~Paternain, \emph{The
  {P}alais-{S}male condition and {M}a\~n\'e's critical values}, Ann. Henri
  Poincar\'e \textbf{1} (2000), no.~4, 655--684.

\bibitem[Con06]{Contreras2006}
G.~Contreras, \emph{The {P}alais-{S}male condition on contact type energy
  levels for convex {L}agrangian systems}, Calc. Var. Partial Differential
  Equations \textbf{27} (2006), no.~3, 321--395.

\bibitem[Dui76]{Duistermaat1976}
J.~J. Duistermaat, \emph{On the {M}orse {I}ndex in {V}ariational {C}alculus},
  Adv. Math. \textbf{21} (1976), 173--195.

\bibitem[Fra04]{Frauenfelder2004}
U.~Frauenfelder, \emph{The {A}rnold-{G}ivental conjecture and moment {F}loer
  homology}, Int. {M}ath. {R}es. {N}ot. \textbf{42} (2004), 2179--2269.

\bibitem[Gin96]{Ginzburg1996}
V.~Ginzburg, \emph{On closed trajectories of a charge in a magnetic field. {A}n
  application of symplectic geometry}, Contact and symplectic geometry
  ({C}ambridge, 1994) (C.~B. Thomas, ed.), Publications of the {N}ewton
  {I}nstitute, vol.~8, Cambridge {U}niversity {P}ress, 1996, pp.~131--148.

\bibitem[Giv90a]{Givental1990a}
A.~Givental, \emph{Nonlinear generalization of the {M}aslov index}, Adv. Sov.
  Math. \textbf{1} (1990), 71--103.

\bibitem[Giv90b]{Givental1990}
\bysame, \emph{The nonlinear {M}aslov index}, Geometry of {L}ow-{D}imensional
  manifolds (S.~Donaldson and C.~B. Thomas, eds.), LMS Lecture Note Series,
  vol. 151, Cambridge {U}niversity {P}ress, 1990, pp.~35--43.

\bibitem[HT11a]{HutchingsTaubes2011}
M.~Hutchings and C.~Taubes, \emph{Proof of the {A}rnold chord conjecture in
  three dimensions {I}}, Math. Research Letters \textbf{18} (2011), no.~2,
  295--313.

\bibitem[HT11b]{HutchingsTaubes2011a}
\bysame, \emph{Proof of the {A}rnold chord conjecture in three dimensions
  {II}}, arXiv:1111.3324 (2011).

\bibitem[HWZ98]{HoferWysockiZehnder1998}
H.~Hofer, K.~Wysocki, and E.~Zehnder, \emph{The {D}ynamics on
  {T}hree-{D}imensional {S}trictly {C}onvex {E}nergy {S}urfaces}, Ann. of Math.
  \textbf{148} (1998), no.~1, 197--289.

\bibitem[Ma{\~n}96]{Mane1996}
R.~Ma{\~n}{\'e}, \emph{Lagrangian flows: the dynamics of globally minimizing
  orbits}, Pitman Research Notes in Math., vol. 362, Longman, 1996,
  pp.~120--131.

\bibitem[Mer11a]{Merry2011a}
W.~J. Merry, \emph{On the {R}abinowitz {F}loer homology of twisted cotangent
  bundles}, Calc. Var. Partial Differential Equations \textbf{42} (2011),
  no.~3-4, 355--404.

\bibitem[Mer11b]{Merry2011}
\bysame, \emph{Rabinowitz {F}loer homology and {M}a\~n\'e supercritical
  hypersurfaces}, Ph.D. thesis, University of Cambridge, 2011, Available online
  at http://www.math.ethz.ch/~merrywi.

\bibitem[MMP12]{MacariniMerryPaternain2011}
L.~Macarini, W.~J. Merry, and G.~P. Paternain, \emph{On the growth rate of
  leaf-wise intersections}, J. Symplectic Geometry \textbf{10} (2012),
  no.~601-653.

\bibitem[Moh01]{Mohnke2001}
K.~Mohnke, \emph{Holomorphic disks and the {C}hord {C}onjecture}, Ann. Math.
  \textbf{154} (2001), 219--222.

\bibitem[Mos78]{Moser1978}
J.~Moser, \emph{A fixed point theorem in symplectic geometry}, Acta. Math.
  \textbf{141} (1978), no.~1-2, 17--34.

\bibitem[MP11]{MerryPaternain2010}
W.~J. Merry and G.~P. Paternain, \emph{Index computations in {R}abinowitz
  {F}loer homology}, J. Fixed Point Theory Appl. \textbf{10} (2011), no.~1,
  88--111.

\bibitem[MS12]{McDuffSalamon2004}
D.~McDuff and D.~Salamon, \emph{{$J$}-holomorphic curves and symplectic
  topology}, Colloquim Publications, vol.~52, Amer. Math. Soc., 2012.

\bibitem[Pat06]{Paternain2006}
G.~P. Paternain, \emph{Magnetic {R}igidity of {H}orocycle {F}lows}, Pacific J.
  Math. \textbf{225} (2006), 301--323.

\bibitem[Poz99]{Pozniak1999}
M.~Pozniak, \emph{Floer homology, {N}ovikov rings and clean intersections},
  Northern {C}alifornia {S}ymplectic {G}eometry {S}eminar (Y.~Eliashberg,
  D.~Fuchs, T.~Ratiu, and A.~Weinstein, eds.), 2, vol. 196, Amer. Math. Soc.,
  1999, pp.~119--181.

\bibitem[Rit13]{Ritter2010}
A.~F. Ritter, \emph{Topological quantum field theory structure on symplectic
  cohomology}, J. Topol. (2013).

\bibitem[RS93]{RobbinSalamon1993}
J.~Robbin and D.~Salamon, \emph{The {M}aslov index for paths}, Topology
  \textbf{32} (1993), 827--844.

\bibitem[Sal99]{Salamon1999}
D.~Salamon, \emph{Lectures on {F}loer {H}omology}, Symplectic {G}eometry and
  {T}opology (Y.~Eliashberg and L.~Traynor, eds.), IAS/Park City Math. Series,
  vol.~7, Amer. Math. Soc., 1999, pp.~143--225.

\bibitem[Sch00]{Schwarz2000}
M.~Schwarz, \emph{On the action spectrum for closed symplectically aspherical
  manifolds}, Pacific J. Math. \textbf{193} (2000), no.~2, 419--461.

\end{thebibliography}

\noindent \emph{Address: }

\noindent Department of Mathematics, ETH Z\"urich, Switzerland\newline 

\noindent \emph{Email:}\texttt{ }

\noindent \texttt{merry@math.ethz.ch}
\end{document}